\documentclass[english,12pt,a4paper]
{amsart}
\usepackage{amssymb}
\usepackage{amsfonts}
\usepackage{amscd}
\usepackage{color}
\usepackage{tikz}
\usepackage{tikz-cd}

\providecommand{\noopsort}[1]{}

\usepackage[colorlinks=true, breaklinks=true, urlcolor= black, allcolors= black,bookmarksopen=true,linktocpage=true,plainpages=false,pdfpagelabels]{hyperref}
\usepackage{cleveref}

\definecolor{immi}{rgb}{0,.6,.1}

\newbox\removebox
\newcommand\remove[2][blue]{%
\setbox\removebox=\ifmmode\hbox{$#2$}\else\hbox{#2}\fi%
\leavevmode
\rlap{\textcolor{#1}{\vrule height0.8ex depth-0.5ex width\wd\removebox}}%
\box\removebox
}
\long\def\bigremove#1{%
\par\setbox\removebox=\vbox{#1}%
\vbox{%
\vbox to0pt{\hbox{\tikz\draw[color=blue,thick] (0,0) -- (\wd\removebox,-\ht\removebox)  (\wd\removebox,0) -- (0,-\ht\removebox);}}
\box\removebox
}
}

\usepackage{mathrsfs} 

\newcommand{\cCM}{{\cC_M
^{\mathrm{exp}}}}
\newcommand{\cCoM}{{\cC_M}}
\newcommand{\cCeM}{{\cC_M^{\mathrm{e}}}}

\newcommand{\cQD}{{\cQ}}
\newcommand{\cQDo}{{\cQD{}^{0}}}
\newcommand{\cQDoloc}{{\cCoM}} 

\newcommand{\cCexp}{\cC^{\mathrm{exp}}}

\newcommand{\cCe}{\cC^{\mathrm{e}}}

\newcommand{\grad}{\operatorname{grad}}

\newcommand{\id}{\operatorname{id}}

\def\VF{\mathrm{VF}}

\def\VG{\mathrm{VG}}

\newcommand{\RF}{{\rm RF}}

\newcommand{\RR}{{\rm R}}

\newcommand{\enrich}{{\rm enrich}}

\def\ac{{
{\rm ac}}}

\def\cMmot{{\MM}}

\def\cross{{\overline{\rm cross}}}

\def\LPas{\cL_{\rm DP}}

\def\gTPas
{{\rm gDP}}

\def\TPres
{{\rm PRES}}

\def\Pres{\operatorname{Pres}}

\def\res{\operatorname{res}}

\def\11{{\mathbf 1}}

\def\CC{{\mathbb C}}

\def\FF{{\mathbb F}}

\def\LL{{\mathbb L}}
\def\MM{{\mathbb M}}
\def\NN{{\mathbb N}}

\def\QQ{{\mathbb Q}}

\def\ZZ{{\mathbb Z}}

\def\cC{{{\mathcal C}}}
\def\cD{{\mathcal D}}

\def\cF{{\mathcal F}}

\def\cK{{\mathcal K}}
\def\cL{{\mathcal L}}
\def\cM{{\mathcal M}}

\def\cO{{\mathcal O}}

\def\cQ{{\mathcal Q}}

\def\cS{{\mathcal S}}

\def\llp{\mathopen{(\!(}}

\def\rrp{\mathopen{)\!)}}

\newtheorem{thm}[subsubsection]{Theorem}
\newtheorem{lem}[subsubsection]{Lemma}
\newtheorem{cor}[subsubsection]{Corollary}
\newtheorem{prop}[subsubsection]{Proposition}

\newtheorem{addendum}[subsubsection]
{Addendum}
\Crefname{prop}{Proposition}{Propositions}
\Crefname{thm}{Theorem}{Theorems}
\Crefname{lem}{Lemma}{Lemmas}
\Crefname{cor}{Corollary}{Corollaries}
\Crefname{conj}{Conjecture}{Conjectures}
\Crefname{rem}{Remark}{Remarks}
\Crefname{addendum}{Addendum}{Addenda}

\theoremstyle{definition}
\newtheorem{defn}[subsubsection]{Definition}

\newtheorem{def-lem}[subsubsection]{Lemma-Definition}
\Crefname{defn}{Definition}{Definitions}
\Crefname{def-prop}{Proposition-Definition}{Proposition-Definitions}
\Crefname{def-lem}{Lemma-Definition}{Lemma-Definitions}
\newtheorem{def-prop}[subsubsection]{Proposition-Definition}
\newtheorem{def-thm}[subsubsection]{Theorem-Definition}

\theoremstyle{remark}
\newtheorem{remark}[subsubsection]{Remark}
\newtheorem{rem}[subsubsection]{Remark}

\theoremstyle{plain}

\numberwithin{equation}{subsection}

  {\par\medskip\noindent #1\par\begingroup%
    \advance\leftskip by 1em\advance\rightskip by 1em}%
  {\par\endgroup}

\newcommand{\ord}{\operatorname{ord}}
\newcommand{\ordalt}{\operatorname{\overline{ord}}}
\newcommand{\Jac}{\operatorname{Jac}}

\usepackage[T1]{fontenc}
\usepackage{mathrsfs}

\newcommand{\R}{{\rm R}}

\author[Cluckers]
{Raf Cluckers}
\address{Univ.~Lille,
CNRS, UMR 8524 - Laboratoire Paul Painlevé, F-59000 Lille, France, and
KU Leuven, Department of Mathematics, B-3001 Leu\-ven, Bel\-gium}
\email{Raf.Cluckers@univ-lille.fr}
\urladdr{http://rcluckers.perso.math.cnrs.fr/}

\author[Loeser]
{Fran\c {c}ois Loeser}
\address{Sorbonne Universit\'e, Institut de Math\'ematiques de Jussieu-Paris
Rive Gauche, CNRS, Campus Pierre et Marie Curie, case 247, 4 place Jussieu, 75252 Paris cedex 5, France.
}
\email{francois.loeser@imj-prg.fr}
\urladdr{https://webusers.imj-prg.fr/$\sim$francois.loeser/}

\author[Nguyen]
{Kien Huu Nguyen}
\address{KU Leuven, Department of Mathematics,
Celestijnenlaan 200B, B-3001 Leu\-ven, Bel\-gium}
\email{kien.nguyenhuu@kuleuven.be}

\author[Vermeulen]
{Floris Vermeulen}
\address{University of M\"unster, Mathematics M\"unster,
Einsteinstrasse 62, 48149 M\"unster, Germany}
\email{florisvermeulen.math@gmail.com}
\urladdr{https://sites.google.com/view/floris-vermeulen/homepage}



\title[Motivic Mellin transforms]{Motivic Mellin transforms}

\begin{document}
\begin{abstract}
This work brings Mellin transforms into the realm of motivic integration.
The new, larger class of motivic functions is stable under motivic Mellin and Fourier transforms, with general Fubini results and change of variables formulas. It specializes to $p$-adic integrals and $p$-adic Mellin transforms uniformly in $p$, with transfer principles between zero and positive characteristic local fields.
In particular, it generalizes previous set-ups of motivic integration with Fubini from among others~\cite{CLoes, CLexp,CLbounded, Kien:rational, CHallp} 
and simplifies some aspects on the way by using the ideas of \cite{CH-eval}.
\end{abstract}
\maketitle





\section{Introduction}\label{sec:intro}

Mellin transforms of functions over non-archimedean local fields are ubiquitous in Number Theory  and Representation Theory. They have a rich history going back to   Tate's thesis \cite{tate}
(where they were called local $\zeta$-functions) and   Igusa \cite{igusa-crelle1}.
If $g$ is a complex valued function over a non-archimedean local field $K$,
its Mellin transform is the function
\[\chi \longmapsto  \int_{K} \chi (x) g (x) \vert dx \vert\]
with $\chi$ running over the set of quasi-characters, that is, continuous  morphisms
$K^\times \to \mathbb{C}^\times$.
An important example is given by Igusa's local zeta function $Z_{f, \varphi} (\chi)$ of a polynomial $f \in K [x_1, \cdots x_n]$ which is the Mellin transform of the fiber integral
$I(t) = \int_{f^{-1} (t) } \varphi  \frac{\vert dx \vert }{\vert df \vert}$, with $\varphi$ some test function.

In the work \cite{CLoes}, the first two authors constructed a general
version of motivic integration in equicharacteristic zero for a class of functions called constructible motivic functions.
The class of constructible motivic functions is stable with respect to integration over parameters. The associated theory of motivic integration
satisfies the expected properties (Fubini, change of variables formula) and specializes to
$p$-adic integration for $p$ large enough. The construction uses in a crucial way
cell decomposition results  for definable sets in the Denef-Pas language.
In subsequent work \cite{CLexp} this theory was extended to include a larger class of functions, exponential motivic functions, that are
functions involving  additive characters. In this framework a motivic version of Fourier transformation satisfying Fourier inversion is developed.
In \cite{CLexp} a general transfer principle allows one to compare identities between functions defined by integrals over local fields with the same residue field.
It  was shown in \cite{CHL} that this  transfer principle applies in particular to  the integrals occuring in the Fundamental Lemma of Langlands-Shelstad.
The theory was then extended in
\cite{CLbounded}
to mixed characteristic discretely valued Henselian fields with
bounded ramification. This   required to work in an expansion of the Denef-Pas language
allowing to consider higher order angular components maps with values in residue rings, as opposed to the
angular components maps taking its values in the residue field.

The main goal of the present work is to provide a framework for  motivic integration which allows one to deal with multiplicative characters and
 which extends the constructions in
\cite{CLoes}  \cite{CLexp} \cite{CLbounded}.
More precisely, for $X$ belonging to a class of definable sets encompassing the ones considered in loc.\ cit., we define an algebra of motivic functions
$\cCM(X)$, containing the ones in loc.\ cit.\ and values of multiplicative characters of (generalized) residue rings.
As in loc.\ cit.\ we develop inductively  an integration theory for such functions
and we  prove that this class of functions is stable under integration over parameters, and satisfies Fubini and a change of variables formula.
We also construct a motivic Mellin transform and show that $\cCM$ is stable under Mellin and Fourier transform.
Though we prove that this motivic  Mellin transform is injective,  we are not able to describe its image. This is due to the fact that the space of multiplicative characters  is quite wild, unlike the space of additive characters. Note that a similar feature also occurs for
the $\ell$-adic Mellin transform of \cite{GL}.
We also prove that our constructions admit specializations to their $p$-adic counterparts and obtain a transfer principle.

Let us now describe in some more detail the structure of the paper.
In Section  \ref{sec:p-adic:mellin} we provide a basic example of our general theory by
recasting the classical theory of Mellin
over $p$-adic fields within our general framework. This is intended to provide the reader  intuition and insight  for the more abstract constructions that are developed in the rest of the paper.
The model theoretic framework we will work with is specified in Section \ref{sec3}.
Our language $\cL_D$ is an expansion of the Denef-Pas language with sorts for general residue rings indexed by some terms (basic terms) and symbols for higher angular components.
This will allow us to make the residue rings vary in a definable way. We also introduce the general notion of  $\cS$-definable set for $\cS$ a collection of $\mathcal{L}_D$-structures.
This generalizes  the notion of definable subassignments from \cite{CLoes} \cite{CLexp} and will allow us to use the notion of evaluation of motivic functions from \cite{CH-eval}.
In order to get a good grasp on such definable sets we prove  quantifier  elimination results in
section  \ref{sec:QE:Pres}  and we prove a cell decomposition theorem for $\cS$-definable sets in
Section \ref{sec:S}.
Section \ref{sec:motivic} is devoted to the construction of the ring $\cCM(X)$
for $X$ an $\cS$-definable set. 
The core of the paper lies in Section \ref{sec:iterated.int} and Section \ref{sec:int:rev}
in which we construct  motivic integrals for functions
in $\cCM(X)$ using cell decomposition from Section \ref{sec:S}.
Finally, in Section \ref{sec:p-adic-transf}, we prove
specialization (Theorem \ref{thm:special})
and transfer (Theorem \ref{thm:transfer})
in our setting.
The Appendix contains a few complements in connection with the notion
of Hensel minimality from \cite{CHR} and \cite{CHRV}.

\subsection*{Acknowledgements} The authors thank J. Gordon, T. Hales and I. Halupczok for valuable discussions during the preparation of the paper. R.C.\ was partially supported by the European Research Council under the European Community's Seventh Framework Programme (FP7/2007-2013) with ERC Grant Agreement nr. 615722 MOTMELSUM, by KU Leuven IF C16/23/010, and by the Labex CEMPI  (ANR-11-LABX-0007-01).
F.L.\ was partially supported by the Institut Universitaire de France.
K.H.N.\ was partially supported by the Fund for Scientific Research - Flanders (Belgium)
(F.W.O.) 12X3519N and 1270923N, by the Excellence Research Chair “FLCarPA: L-functions in positive characteristic and applications” financed by the Normandy Region. F.V.\ was partially supported by the Fund for Scientific Research - Flanders (Belgium) (F.W.O.) 11F1923N, and partially by the Humboldt Foundation.
Part of the paper was finalized during a stay of the four authors at the Bernoulli Center in Lausanne. The authors would like to thank the Bernoulli Center for providing such an ideal working environment.

\section{Mellin transforms on $p$-adic fields}\label{sec:p-adic:mellin}

We first develop the very concrete situation of Mellin transform on a fixed $p$-adic field. More precisely, we will introduce algebras of $\CC$-valued functions on definable subsets of $\QQ_p^n$ which are stable under integration, under taking Fourier transforms, and under Mellin transforms. It contains the ideas and constructions that will become more abstract from the next part on. The results in this section will follow from the more abstract results on motivic Mellin transform and will be proved after proving the motivic results.

\subsection{Notation and set-up}\label{subs:notQp}
Let $K$ be a finite field extension of $\QQ_p$ for some prime number $p$. Write $\cO_K$ for the valuation ring of $K$ with maximal ideal $\cM_K$ and residue field $k_K$ with $q_K$ elements. We write the valuation $\ord$ additively, with value group $\ZZ$, where we extend $\ord$  by $+\infty$ on $0$. Let $\pi_K$ be a uniformizer of $\cO_K$ (that is, an element of $\cM_K$ of valuation $1$). We will introduce a language $\cL_D$ on $K$ in several steps; it is a variant of the Denef-Pas language with more general `depths' of the angular component maps.

Consider the language $\LPas$ on the three sorted structure with sorts $K$, $\ZZ$ and $k_K$, a copy of the ring language on both $K$ and on $k_K$, the Presburger language (namely, with function symbols $+,-$ and $\max$, constant symbols $0,1$, the order relation  $\leq$, and a congruence relation $\equiv_n$ for each integer $n>1$)\footnote{Note that we added the binary maximum function to the usual Presburger language, for naturality of subsequent definitions.} 
on $\ZZ$, a function symbol $\ordalt$ for the valuation map $K^\times\to\ZZ$ extended by mapping $0$ to $0$, and a symbol $\ac$ for the map $K\to k_K$ sending $0$ tot $0$ and nonzero $x$ to the reduction of $x\pi_K^{-\ord x}$ modulo $\cM_K$.

Let $\Lambda$ be a countably infinite set of symbols (which are supposed to be different from all symbols of $\LPas$) and let $\LPas(\Lambda)$ be $\LPas$ together with a constant symbol $\lambda$ in the sort $\VG$ for each $\lambda\in\Lambda$.

Consider a term $t$ in $\LPas(\Lambda)$. Call $t$ a \emph{basic term} if $t$ contains no variables and it is $\VG$-valued. Thus, in a basic term $t$ there occur only constant and function symbols from $\LPas(\Lambda)$, and $t$ takes values in $\VG$. As an example of a basic term one can think of $t$ being $\ordalt(N)+\max\{\lambda_1, \lambda_2+2\}-\lambda_3$ for some integer $N>0$ and some $\lambda_1, \lambda_2,\lambda_3\in\Lambda$.

For $t$ a basic term, denote by $B_t(0)$ the open ball of radius $t$ around $0$. That is, $B_t(0) = \{x\in \cO_K\mid \ord(x) > t\}$.

\begin{defn}\label{defnLdacQp}
Let $\cL_D$ be the language $\LPas(\Lambda)$ together with, for each basic term $t$,
a residue ring sort $\RR_{t}$ which is a new sort when $t$ is not the term $0$, and which coincides with the residue field sort from $\LPas$ if $t$ is the term $0$,
and with the following symbols, for each basic terms $t$ and $t'$:
\begin{itemize}
\item the ring language on $\RR_{t}$, which coincides with the ring language on the residue field sort from $\LPas$ when $t$ is the term $0$,

\item a function symbol $$\ac_{t}:\VF\to \RR_{t}$$ for an angular component map modulo $B_{t}(0)$, which coincides with the symbol $\ac$ from $\LPas$ when $t$ is the term $0$,

\item a function symbol $\res_{t,t'}:\RR_t\to\RR_{t'}$.
\end{itemize}
\end{defn}

The $\cL_D$-structure on $K$ is fixed by fixing the constant symbols $\lambda$ from $\Lambda$ in $\ZZ$, by the following corresponding interpretations, where we note that any basic term $t$ has a specific value in $\ZZ$ if all the $\lambda$ from $\Lambda$ are fixed. For any basic term $t$, the universe of $\RR_t$ is $\{0\}$ when $t<0$, and when $t\ge 0$ it is the residue ring
$$
\cO_K/ B_t(0)
$$
where $B_t(0)$ is as above the open ball of (valuative) radius $t$ around $0$, which is the ideal $\cM_K^{t+1}$ of $\cO_K$; if $t<0$ then $\ac_{t}$ is the zero map; if $t\ge 0$ then $\ac_{t}$ is the map sending zero to zero and nonzero $x$ to the reduction of $x\pi_K^{-\ord x}$ modulo $B_t(0)$, and the function symbol $\res_{t,t'}:\RR_t\to\RR_{t'}$ is the natural projection $\RR_{t}\to \RR_{t'}$
when $t\ge t'$, and the constant map with value $0$ when $t<t'$.

We call a tuple $t = (t_1,\ldots,t_\ell)$ of basic terms a \emph{basic tuple}, and we write
\begin{equation}\label{rr(i)}
\RR_{t}\quad  \mbox{ for }\quad  \prod_{j=1}^\ell \RR_{t_j},
\end{equation}
which by convention denotes $\{0\}$ when $\ell=0$.

\subsection{Functions on a definable set}\label{subs:fixed-functions:Qp}

Consider an $\cL_D$-structure on $K$ which is a finite field extension of $\QQ_p$ as above, and fix furthermore an additive character
$$
\psi : K \to \CC^\times
$$
such that $\psi(\cM_K)=1$ and $\psi(\cO_K)\not=1$. Recall that an additive character is just a continuous group homomorphism from $(K,+)$ to $(\CC^\times,\cdot)$. Write $\cD_K$ for the collection of all such additive characters on $K$. Furthermore, for each basic term $t$, fix a multiplicative character $\chi_t$ on $\R_t$, that is, a group homomorphism from $(\R_t^\times,\cdot)$ to $\CC^\times$.\footnote{Recall that the unit group of the trivial ring $\{0\}$ is the trivial group $\{0\}$.}

Let $X$ be an $\cL_D$-definable set in the structure $K$, that is, $X$ is a subset of $K^n\times \ZZ^m\times R_t$ for some basic tuple $t$ and some integers $n\ge 0$, $m\ge 0$, given by an $\cL_D$-formula.
By an $\cL_D$-imaginary function $g$ from $X$ to $K/\cM_K$ we mean a function $g:X\to \VF_K/\cM_K$ such that $G$ is an $\cL_D$-definable set, where $G\subset X\times K$ is the preimage of the graph of $g$ under the projection $X\times K\to X\times K/\cM_K$. 

By a point on such $X$ we mean in this case simply an element $x$ of $X$. By an \emph{enriched point on $X$} we mean a point $x$ on $X$ together with a choice of $\psi$ and $\chi_t$ for all $t$ as above. We will denote such a point as $(x,\psi, (\chi_t)_t)$. Denote by $X^\enrich$ the set of enriched points on $X$.

We are now ready to define our algebras of functions on definable sets, which will be stable under integration, Mellin transform, and Fourier transform.
\begin{defn}\label{Qp-functions}
Let $X$ be an $\cL_D$-definable set in the structure $K$.
Define
$$
\cCexp_{M,(K,\cL_D)} (X^\enrich)
$$
as the 
$\ZZ$-algebra of functions on $X^\enrich$ taking values in $\CC(T_0,T_1,T_2,\ldots)$
generated by all functions of the form $F_1$ and $F_2$ just below:
\begin{itemize}
\item For any integer $a<0$ and any $b \in \NN^J$ for a finite set $J\subset \NN$, the function $F_1$ sending any enriched point on $X$ to the element
    $$
\frac{1}{1-q_K^a T^b}
    $$
 $\CC(T_0,T_1,\ldots)$, with the multi-index notation $T^b$ for $\prod_{j\in J}T_j^{b_j}$. We denote this function $F_1$ by $\frac{1}{1-q_K^a T^b}$.



\item 

For any $\cL_D$-definable set $Z\subset X\times \RR_{t}$ with $t$ a basic tuple,  
any $\cL_D$-imaginary function $ g:Z\to K/\cM_K$ and any
$\cL_D$-definable functions
$$
\alpha:Z\to\ZZ^n,\ \beta:Z\to \ZZ,\ \gamma:Z\to \ZZ^J,  \mbox{ and } h:Z\to \RR_i,
$$
with $n\ge 0$, 
a finite set $J\subset \NN$, and 
basic tuple $i=(i_1,\ldots,i_\ell)$, the function $F_2$ sending an enriched point $(x,\psi, (\chi_t)_t)$ on $X$ to the following finite sum over $\xi$ running in $Z_x := \{\xi\in \R_t\mid (x,\xi)\in Z\}$, and with $\xi'=(x,\xi)$,
$$
\sum_{\xi\in Z_x} (\prod_{i=1}^n \alpha_i(\xi') )q_K^{\beta(\xi')} T^{\gamma(\xi')}  \psi ( g (\xi') ) \big(\prod_{j=1}^\ell \chi_{i_j}( p_j\circ h(\xi'))\big),
$$
where $p_j$ is the projection $R_i\to R_{i_j}$ onto the $j$th coordinate of $R_i$, and where $T^{\gamma(\xi')}$ stands for multi-index notation, that is, for $\prod_{j\in J} T_j^{\gamma_j(\xi')}$. 

\end{itemize}
\end{defn}

In the next subsection we will build in the possibility to move $\psi$ as well as the $\lambda_i$ from $\Lambda$ and the characters $\chi_t$, so that in a way they all become additional variables that our functions take as input. In the motivic setting from \Cref{sec:motivic} on, we will also move $K$ and even allow it to be not locally compact. To that end, we will have to introduce some more abstract relations and the motivic functions will no longer be complex valued. 

\subsection{Functions on uniformly enriched definable sets}\label{subs:enriched-functions:Qp}

Let $K$ be a finite field extension of $\QQ_p$, and let $\cS_{D,K}$ be the collection of all the $\cL_D$-structures on $K$ as described above, that is, when the choice of uniformizer $\pi_K$ of $\cO_K$ (which is used for the angular component maps) and the $\lambda\in \Lambda$ vary over all possible choices.

By an \emph{$\cS_{D,K}$-definable set} we now mean a collection of sets $X = (X_{K'})_{K'\in \cS_{D,K}}$ such that there is an $\cL_D$-formula $\varphi$ such that the definable set given by $\varphi$ in the structure $K'$ is precisely $X_{K'}$.

By an \emph{$\cS_{D,K}$-definable function} we mean a collection of functions $f = (f_{K'}:X_{K'}\to Y_{K'})_{K'\in \cS_{D,K}}$ whose collection of graphs forms an $\cS_{D,K}$-definable set, and such that $(Y_{K'})_{K'\in \cS_{D,K}}$ is a definable set as well.

By a \emph{point} on an $\cS_{D,K}$-definable set $X$, we mean a pair $x=(x_0,K')$ where $x_0$ lies on $X_{K'}$ and $K'$ is in $\cS_{D,K}$. 

By an \emph{enriched point} on an $\cS_{D,K}$-definable set $X$, we mean a tuple $x=(x_0,K',\psi,(\chi_t)_{t},(\lambda_{K'})_{\lambda \in\Lambda})$ where $(x_0,K')$ is a point on $X$, where $\psi:K\to \CC^\times$ is an additive character which lies in $\cD_K$, where for each basic term $t$, $\chi_t$ is a multiplicative character $R_t^\times\to \CC$, 
and where $\lambda_{K'}$ is the interpretation of $\lambda$ in $\ZZ$ according to the $\cL_D$-structure $K'$.

Now we can define a generalization of the functions from \Cref{Qp-functions}, where now $\psi$ as well as the $\lambda$ and the $\chi_t$ can move. For this, we work with the set $X^\enrich$ of enriched points on an $\cS_{D,K}$-definable set $X$. 

\begin{defn}\label{Qp-functions:enrich}
Let $K$ be a finite field extension of $\QQ_p$, and let $X$ be an $\cS_{D,K}$-definable set.
Define
$$
\cCexp_{M,K} (X^\enrich)
$$
as the 
$\ZZ$-algebra of functions on $X^\enrich$ taking values in $\CC(T_0,T_1,T_2,\ldots)$
generated by all functions of the form as $F_1$ and $F_2$ just below:
\begin{itemize}
\item For any integer $a<0$ and any $b \in \NN^J$ for a finite set $J\subset \NN$, the function $F_1$ sending any enriched  point $x$ on $X$ to the element
    $$
    F_1(x) = \frac{1}{1-q_K^a T^b},
    $$
with the multi-index notation $T^b$ for $\prod_{j\in J}T_j^{b_j}$. Denote this function $F_1$ by $\frac{1}{1-q_K^a T^b}$.



\item 

For any $\cS_{D,K}$-definable set $Z\subset X\times \RR_{t}$ with $t$ a basic tuple,  
any $\cS_{D,K}$-imaginary function $ g:Z\to K/\cM_K$ and any
$\cS_{D,K}$-definable functions
$$
\alpha:Z\to\ZZ^n,\ \beta:Z\to \ZZ,\ \gamma:Z\to \ZZ^J,  \mbox{ and } h:Z\to \RR_i,
$$
with $n\ge 0$, 
a finite set $J\subset \NN$, and 
basic tuple $i=(i_1,\ldots,i_\ell)$, the function $F_2$ sending an enriched point
$$
x=(x_0,K',\psi,(\chi_t)_{t},(\lambda)_{\lambda_\in\Lambda})
$$
on $X$ to the following finite sum over $\xi$ running in $Z_x := \{\xi\in \R_t\mid (x_0,\xi)\in Z_{K'}\}$, and with $\xi'=(x,\xi)$,
$$
\sum_{\xi\in Z_x} (\prod_{i=1}^n \alpha_i(\xi') )q_K^{\beta(\xi')} T^{\gamma(\xi')}  \psi ( g (\xi') ) \big(\prod_{j=1}^\ell \chi_{i_j}( p_j\circ h(\xi'))\big),
$$
where $p_j$ is the projection $R_i\to R_{i_j}$ onto the $j$th coordinate of $R_i$, and where $T^{\gamma(\xi')}$ stands for multi-index notation, that is, for $\prod_{j\in J} T_j^{\gamma_j(\xi')}$, and where $\alpha_i(\xi')$ is shorthand for $\alpha_{i,K'}(\xi')$ and so on for the other $\cS_{D,K}$-definable functions.
\end{itemize}
\end{defn}

\subsection{Results on integration}\label{sec:res:Qp:int}

We have the following results on the rings defined above. These follow from the more general motivic results in Section~\ref{sec:motivic} and a specialization argument, as will be explained in more detail in Section~\ref{sec:p-adic-transf}. Note that we will define a more formal form of integrability in \Cref{sec:iterated.int} for motivic functions, which, when specialized to the $p$-adic case, implies $L^1$ in the measure theoretic sense, see the following lemma. Note that it may be interesting to investigate the main results on integrability of \cite{CGH} in our setting, namely whether formal integrability (for functions on $p$-adic fields) is equivalent to being $L^1$, and to study the corresponding loci of integrability and transfer results for integrability. We leave this for some future. In this section, we say integrable as in \Cref{sec:iterated.int} for the formal notion of integrability (that we don't specify here yet), and we say $L^1$ for the classical measure theoretic notion of integrability. In this section, it is sufficient to know the following implication between these notions of integrability.

\begin{lem}[Integrability notions]\label{defn:integrability:K}
Let $K$ be a finite field extension of $\QQ_p$, let $X$ be $\cS_{D,K}$-definable and let $F\in \cCexp_{M, \cS_{D,K}}(X^{\enrich})$. Suppose that $F$ is  integrable as in \Cref{sec:iterated.int}. Then, for all choices of $\lambda\in \ZZ$ for $\lambda\in \Lambda$, of a uniformizer $\pi_K$ of $\cO_K$, of an additive character $\psi: K\to \CC$ trivial on $\cM_K$ and non-trivial on $\cO_K$, a collection of multiplicative characters $\chi_t: \RR_t^\times \to \CC^\times$, and a collection of $\tau_i\in \CC^\times$ with $|\tau_i|\leq 1$ for $i\in \NN$, the function
\[
X_{(K,\cL_D)}\to \CC: x\mapsto F(x,(\lambda)_{\lambda\in \Lambda}, \psi, (\chi_t)_t) (\tau_0, \tau_1, \ldots)
\]
is $L^1$, where the structure $(K,\cL_D)$ comes from fixing $\pi_K$ and the $\lambda$ in $\ZZ$.  Here, for the $L^1$ measure we take the normalized Haar measure on $K$ and the counting measures on $\RR_t$ and on $\ZZ$, and product measures of these.
\end{lem}

Let $X$ be an $\cS_{D,K}$-definable set. Then we say that an $\cS_{D,K}$-definable set $Y\subset X$ has \emph{measure zero} (in $X$) if for every choice of $\cL_D$-structure on $K$, the resulting $\cL_D$-definable set $Y$ has measure zero in the $\cL_D$-definable set $X$, with the product measures as in \Cref{defn:integrability:K}. By the notion \emph{almost all} (enriched) points on an $\cS_{D,K}$-definable set $X$ we mean all points outside a measure zero $\cS_{D,K}$-definable subset of $X$.

\begin{thm}[Stability under integration]
Let $X\subset Y\times Z$ be $\cS_{D,K}$-definable sets, and let $F\in \cCexp_{M,\cS_{D,K}}(X^{\enrich})$. Assume that for almost every $y\in Y^{\enrich}$ 
the function $F(y, \cdot)$ is integrable over $Z$ (as an $\cL_D(y)$-definable set, and as in \Cref{sec:iterated.int}). Then there exists a unique function $G\in \cCexp_{M, \cS_{D,K}}(Y^{\enrich})$ such that for almost every $y$ in $Y^{\enrich}$ we have that
\[
G(y) = \int_{z\in X_y} F(y,z)dz.
\]
\end{thm}

Let us explain the final expression $G(y) = \int_{z\in X_y} F(y,z)dz$ of this theorem more precisely. The point $y\in Y$ is an enriched point, and so comes with a choice of characters $\psi, (\chi_t)_t$ and a choice of $\cL_D$-structure on $K$, i.e.\ a choice of $\lambda\in \ZZ$ for $\lambda\in \Lambda$ and of uniformizer $\pi_K$. Then $G(y)$ is an element of $\CC((T_i)_{i\in \NN})$ and $X_y$ becomes an $\cL_D(y)$-definable set in the structure $K$. For $z\in X_y$, also $F(y,z)$ becomes an element of $\CC((T_i)_{i\in \NN})$. The equality $G(y) = \int_{z\in X_y} F(y,z)dz$ now means that for every choice of $\tau_i\in \CC^\times$ with $|\tau_i|\leq 1$ we have that
\[
G(y)(\tau_0, \tau_1, \ldots) = \int_{z\in X_y} F(y,z)(\tau_0, \tau_1, \ldots)  dz,
\]
where the integration is against the product measure described in \Cref{defn:integrability:K}.

The following result is very close to the ideas about how our formal notion of integrability is defined in \Cref{sec:iterated.int}.

\begin{thm}[Fubini-Tonelli]
Let $X\subset Y\times Z$ be $\cS_{D,K}$-definable sets, and let $F\in \cCexp_{M, \cS_{D,K}}(X^{\enrich})$. Then $F$ is integrable over $X$ (as in \Cref{sec:iterated.int}) if and only if there exists an $\cS_{D,K}$-definable measure zero set $Z_0\subset Z$ such that for every $\cS_{D,K}$-definable set $A\subset X$ the following holds.

There is $G_A$ in $\cCexp_{M, \cS_{D,K}}(Z^{\enrich})$ which is integrable (as in \Cref{sec:iterated.int}), and for each enriched point $z$ on $Z\setminus Z_0$ the function $(\11_A\cdot F)(\cdot, z)\in \cCexp_{M, (K, \cL_d(z))}(Y_z)$ is integrable (as in \Cref{sec:iterated.int}) with integral over $Y_z$ equal to $G_A(z)$.

If $F$ is integrable as in \Cref{sec:iterated.int}, then we furthermore have that
\[
\int_{X}\11_A\cdot F = \int_{Z} G_A.
\]
\end{thm}

If $f: K^n\to K^n$ is an $\cS_{D,K}$-definable function, then differentiability of $f$ and its Jacobian are defined in the usual way. If $f$ is differentiable at a point $x$, we denote by $(\Jac f)(x)$ the Jacobian of $f$ at $x$. Since we are working in a $p$-adic field, definable functions are automatically differentiable almost everywhere.

\begin{thm}[Change of variables]
Let $X,Y$ be $\cS_{D,K}$-definable sets in $K^n$ for some $n$, and let $f: X\to Y$ be an $\cS_{D,K}$-definable bijection. Let $F\in \cCexp_{M,\cS_{D,K}}(Y^{\enrich})$. Then
\[
\int_{X}f^*(F) q_K^{-\ord \Jac_f} = \int_Y F,
\]
where $f^*(F)$ is the composition $F\circ f$.
\end{thm}

Since the characters $\psi$ and $\chi_t$ appear explicitly in the notion of enriched points, we can move these and consider them as variables. This allows us to define Mellin transforms on the motivic level as follows.

\begin{defn}[Mellin transform]
Let $F\in \cCexp_{M, \cS_{D,K}}(((K^\times)^n)^{\enrich})$ be integrable as in \Cref{sec:iterated.int}. Take $j_1, \ldots, j_n$ positive integers such that the $T_{j_i}$, resp.~the $\lambda_{j_i}$, do not appear in any of the generators, resp.~formulas, in any of the data defining $F$. We define the \emph{Mellin transform of $F$} as
\[
\MM(F) = \int_{x\in (K^\times)^n} F(x)\prod_{i=1}^n T_{j_i}^{\ord x_i} \chi_{\lambda_{j_i}}(\ac_{\lambda_{j_i}}(x_i))\in \cCexp_{M, \cS_{D,K}}(\{0\}^{\enrich}).
\]
\end{defn}

In this definition, the elements $\lambda_{j_i}\in \Lambda$ and the characters $\chi_{\lambda_{j_i}}$ should be thought of as variables for the function $\MM(F)$, for $i=1,\ldots,n$. In more detail, by varying the values of the $\lambda_{j_i}$ and the $\chi_{\lambda_{j_i}}$, the object $\MM(F)$ naturally becomes a function on the dual space of $(K^\times)^n$, which is nothing but the classical Mellin transform.

Note that using any choice of additive character $\psi$, one can also take the Fourier transform $\cF(f)$ of an integrable (as in \Cref{sec:iterated.int}) function $f$ in $\cCexp_{M, \cS_{D,K}}((K^n)^{\enrich})$, which  is given by
$$
\cF(f)(y) =  \int_{x\in K^n} f(x) \psi(x\cdot y ).
$$

A key point is that $\cF(f)$ and $\MM(F)$ live in our frameworks, more precisely they belong to $\cCexp_{M, \cS_{D,K}}((K^n)^{\enrich})$, resp.~to $\cCexp_{M, \cS_{D,K}}(\{0\}^{\enrich})$, by the stability under integration results stated above in \Cref{sec:res:Qp:int}.

For all of the results above in this section, one can also formulate family results. This situation will be dealt with in greater generality in Section~\ref{sec:motivic}.

\subsection{Link with classical Mellin transforms}\label{subs:classical:Mel}

Let us make a link with the classical Mellin transform on the multiplicative group of locally compact fields, which is injective on $L^1$ functions by the classical theory of Fourier transforms on locally compact abelian groups.
Let $f$ be a complex-valued $L^1$ function on $(K^\times)^n$ for some finite field extension $K$ of $\QQ_p$, against the measure which is induced from the Haar measure $|dx|$ on $(K^n,+)$ normalized so that $\cO_K^n$ has measure $1$.\footnote{Note that it would also be natural to use the Haar measure on $(K^\times)^n$, which is $|dx|/|\prod_{i=1}^nx_i|$, but this will just amount to a shift by $1$ in the variables $s_i$ later below, and also, one can always divide or multiply $f$ with $|\prod_{i=1}^nx_i|$ to go from one Haar measure to the other.} Then the classical Mellin transform $\cM(f)$ of $f$ is usually defined as the function in complex variables $s_i$ with real part at least $0$ and in multiplicative characters $\chi_i$ on $\cO_K^\times$, with finite image and factorizing through, say, $(\cO_K / \cM^{N_i})^\times$ for some integers $N_i>0$, by the integral
$$
\cM(f) = \int_{x\in (K^\times)^n} f(x) \prod_{i=1}^n |x_i|^{s_i} \chi_i(\ac_{N_i}(x))|dx|.
$$
Up to a shift of $s_i$ by plus or minus one, this is in fact the standard definition of Mellin transform on the multiplicative group of a local field $K$. Note that one may replace the expression $q_K^{-s_i}$ by another complex variable $T_i$ which runs correspondingly over nonzero complex numbers of modulus at most $1$. Then the restriction of $\cM(f)$ to the unit circle in the $T_i$-variables (and up to the shift of the $s_i$) corresponds to the usual Fourier transform on $(K^\times)^n$, since the collection of such tuples $(\chi_i)_i$ together with the $n$-th Cartesian power of the unit circle form the dual group of $(K^\times)^n$, see e.g.~\cite{Igusa:intro}.
The injectivity for Fourier transforms thus yields that $\cM(f)$ is zero if and only if $f$ is almost everywhere zero.

Note that if $f$ is furthermore like as in \Cref{subs:fixed-functions:Qp}, more precisely that $f$ belongs to $\cCexp_{M,(K,\cL_D)} ( (K^\times)^n )$, and if $f$ is the specialization of some motivic, integrable function, then $\MM(f)$ and $\cM(f)$ are essentially the same function, taking as input the variables $\chi_i$ and the $T_i$ for $i=1,\ldots,n$.

\subsection{Positive characteristic local fields}\label{subs:Fqt}
All of the definitions from this section make sense also when $K$ is $\FF_q\llp t\rrp$, where $q$ is some power of a prime $p$, and even more generally for any local field of positive characteristic. We can even maintain all the same notation when $K=\FF_q\llp t\rrp$ since $K$ is visually present in the objects, often as an index $K$. Note however that none of the results of the above part of \Cref{sec:p-adic:mellin} hold for such $K=\FF_q\llp t\rrp$, as already the study of the definable sets and functions in $K^n$ remains elusive up to date.

In the final section of this paper, in \Cref{sec:p-adic-transf}, we will show that given a finite amount of objects, like finitely many motivic functions, one can specialize them to be as in the above definitions with $K= \FF_q\llp t\rrp$ and, as soon as the characteristic of $K$ is large enough (compared to the given objects), some of the results will still hold for these specializations to $K$. We will also give partial transfer results to switch from mixed characteristic to positive equicharacteristic, when studying equalities of integrals and more generally of (specializations of) functions of class $\cCM$ as defined below.

\section{Valued fields and $\cS$-definable objects}\label{sec3}

We begin working towards our definitions of rings of motivic functions. In this section we introduce the relevant notation on valued fields, and we introduce the language we will work with.

\subsection{Valued fields} 
\label{sec:D-bounded}

We fix some notation for valued fields, which in this paper we always assume to be non-trivially valued fields. Recall first that a (non-trivially) valued field $K$ is by definition the field of fractions of a valuation ring $\cO\subsetneq K$, together with the information of $\cO$ as a proper subring of $K$.

Let $K$ be a valued field with valuation ring $\cO_K$ and let $\cM_K$ be the maximal ideal of $\cO_K$. The value group of $K$ is by definition $K^\times/\cO_K^\times$ in multiplicative notation, and comes with a natural map $K^\times\to K^\times/\cO_K^\times$ that we denote by $|\cdot|$; there is a natural ordering on the value group where one has $|x|< |y|$ for $x$ and $y$ in $K^\times$ if and only if $x/y$ lies in $\cM_K$.  We will more often use additive notation for the value group, and we denote the additively written value group of $K$ by $\VG_K$, with valuation map $\ord : K \to \VG_K\cup \{+\infty\}$.
For $i$ in $\VG_{K}$, write
$$
B_i(0) := \{ x\in \cO_K\mid \ord ( x )  >  i  \}
$$
for the open ball of (valuative) radius $i$ around $0$.
For $i$ in $\VG_{K}$ 
define
$$
\RR_{i,K}:=\cO_K / B_i(0),
$$
as the quotient ring of $\cO_K$ modulo the ideal $B_i(0)$ if $i\ge 0$, and as the trivial ring $\{0\}$ when $i<0$. The notation  $\RR$ stands for \emph{residue ring}.
We also write $\RF_K$ or $k_K$ instead of $\RR_{0,K}$; it is the residue field of $\cO_K$. Write $\VG_{\geq 0,K}$ for the set of $i\in \VG_K$ with $i\geq 0$.
For $j$ and $i$ in $\VG$, write
$$
\res_{j,i} : \RR_{j,K}\to \RR_{i,K}
$$
for the natural projection map when $j\ge i$, and for the constant map with value $0$ when $j<i$.
By an angular component map modulo $B_i(0)$ we mean a multiplicative map
$$
\ac_i : K \to \RR_{i,K}
$$
sending zero to zero and coinciding on the unit group $\cO_K^\times$ with the restriction to $\cO_K^\times$ of the projection map $\cO_K\to \RR_{i,K}$.

We call angular component maps $\ac_i$ and $\ac_j$ for $i<j$ in $\VG_{\geq 0,K}$ compatible when they make a commutative diagram with the map $\res_{j,i}$. If one has $\VG_K\cong\ZZ$, then pairwise compatible angular component maps modulo $B_i(0)$ exist for all $i\ge 0$, and they can be constructed by using a uniformizer.


For $i$ in $\VG_K$, we call $ \{ x\in \cO_K\mid \ord ( x- a  )  \ge i  \}$ a closed ball of valuative radius $i$ around $a$, and similarly $ \{ x\in \cO_K\mid \ord ( x- a  ) > i  \}$ the open ball of valuative radius $i$ around $a$.

\subsection{Language and structures}\label{subsec:def}

We use the first order logic notions of multi-sorted languages, structures, theories, models, formulas, 
and so on.
We introduce  a variant $\cL_D$ of the generalized Denef-Pas language from \cite{CHallp} in \Cref{defnLdac}.
We first recall the classical Denef-Pas language $\LPas$, from which we will build $\cL_D$, similarly as in \Cref{sec:p-adic:mellin}.

\begin{defn}[The Denef-Pas language]\label{def:LPas}
Let $\LPas$ be the language with sorts\\

\begin{itemize}
\item[(1)] the valued field sort $\VF$,
\item[(2)] the value group sort $\VG$, 
\item[(3)] the residue field sort $\RR_0$,
\end{itemize}
and with  the following symbols\\

\begin{itemize}
\item[(a)]  the ring language on the sort $\VF$ (namely, with symbols $+,-,\cdot,0,1$),
\item[(b)] the Presburger language on $\VG$ (namely, with binary function symbols $+,-$ and $\max$, constant symbols $0,1$, and the relation symbols $\leq$ and $\equiv_n$ for each integer $n>1$), 

\item[(c)] the ring language on $\RR_0$,

\item[(d)] a function symbol $$\ac:\VF\to \RR_0$$ for an angular component map modulo the maximal ideal $B_{0}(0)$,

\item[(e)] a function symbol
$$
\ordalt:\VF\to\VG
$$ for the valuation map 
on $\VF^\times$ extended by $0$ on $0$.
\end{itemize}


\end{defn}

Let $\Lambda$ be a countably infinite set of symbols (which are supposed to be different from all symbols of $\LPas$) and let $\LPas(\Lambda)$ be $\LPas$ together with a constant symbol $\lambda$ in the sort $\VG$ for each $\lambda\in\Lambda$.

Consider a term $t$ in $\LPas(\Lambda)$. Call $t$ a basic term $t$ if $t$ contains no variables and it is $\VG$-valued. Thus, in a basic term $t$ there occur only constant and function symbols from $\LPas(\Lambda)$, and $t$ takes values in $\VG$. As an example of a basic term one can think of $t$ being $\ordalt(N)+\max\{\lambda_1, \lambda_2+2\}-\lambda_3$ for some integer $N>0$ and some $\lambda_1, \lambda_2,\lambda_3\in\Lambda$.


\begin{defn}\label{defnLdac}
Let $\cL_D$ be the language $\LPas(\Lambda)$ together with, for each basic term $t$, \\  

\begin{itemize}
\item[(3$_t$)] a residue ring sort $\RR_{t}$ which is a new sort when $t$ is not the term $0$, and which coincides with the residue field sort from $\LPas$ if $t$ is the term $0$.
     \\
\end{itemize}
and with the following symbols, for each basic terms $t$ and $t'$,\\

\begin{itemize}
\item[(c$_t$)] the ring language on $\RR_{t}$, which coincides with the ring language on $\RR_0$ from $\LPas$ when $t$ is the term $0$,

\item[(d$_t$)] a function symbol $$\ac_{t}:\VF\to \RR_{t}$$ for an angular component map modulo $B_{t}(0)$, which coincides with the symbol $\ac$ from $\LPas$ when $t$ is the term $0$,

\item[(f)] a function symbol $$\res_{t,t'}:\RR_t\to\RR_{t'}.$$
\end{itemize}
\end{defn}


We will work with the following collection of $\cL_D$-structures. 


 \begin{defn}\label{defnSdac}
Let $\cS_D$ be the collection of all $\cL_D$-structures with as $\VF$-sort a henselian valued field $K$ of characteristic zero with value group $\VG_K\simeq\ZZ$ and with natural interpretations of all symbols and where furthermore the angular component maps are mutually compatible. In more detail, $\cK$ in $\cS_D$ consists of a (non-trivially) henselian valued field $K$ of characteristic zero as the universe for the $\VF$-sort, the ordered abelian group $\VG_K$ as universe for the sort $\VG$ with furthermore $\VG_K\simeq\ZZ$, for each basic term $t$ with interpretation $t_\cK$ in $\VG_K$ the residue ring  $\RR_{t_\cK,K}$ as universe for the sort $\RR_{t}$,
and, where the symbols from (a) up to (f) in $\cL_D$ have their natural meaning, and where the angular component maps from (d$_t$) are mutually compatible with each other.
\end{defn}

\begin{rem}
By the mentioned natural meaning for the symbols in \Cref{defnSdac} it is meant that, in an $\cL_D$-structure as in \Cref{defnSdac}, each $\lambda$ from $\Lambda$ is interpreted as some chosen element $\lambda_\cK$ of $\VG_K$,
one has the ring structure (for the ring language) on $K$ and also on each of the $R_t$, the valuation map $\ord:K^\times \to \VG_K$ extended by zero on zero as meaning for the symbol $\ordalt$, compatible angular component maps, the Presburger structure on $\VG_K$ (where $\equiv_n$ stands for the equivalence relation modulo $n$ for $x,y$ in $\ZZ$, and $\max(x,y)$ for the maximum of $x,y$ in $\ZZ$), and the maps $\res_{t_\cK,t'_\cK}$ as interpretations for the symbols $\res_{t,t'}$.
\end{rem}

\begin{rem}
Note that, in a valued field $K$, the only choices for the symbols of the language $\cL_D$ occur for the (mutually compatible) angular component maps and for the constant symbols $\lambda$ from $\Lambda$. If one would interpret all $\lambda$ from $\Lambda$ as $0$ in $\VG_K$ and if one would allow only basic terms of the form $\ord N$ for integers $N>0$, then one would recover the generalized Denef-Pas structure on $K$ from \cite{CHallp}.
\end{rem}

One has elimination of $\VF$ and of $\VG$-quantifiers in a slightly bigger language than $\cL_D$ which has the same definable sets as $\cL_D$; this is explained in \Cref{sec:QE:Pres} and is a variant of \cite[Theorem 5.1.2]{CHallp}.
Note that all characteristics are allowed for the residue field of $\VF_K$ for $K$ in $\cS_D$, and any finite ramification.

The index $D$ in $\cL_D$ and in $\cS_D$ stands for the freely varying \emph{depth} of the residue rings $\RR_t$. By a small abuse of notation, we often write $K$ for a structure in $\cS_D$ as well as for its valued field sort.

\subsection{Definable sets and points, in collections of structures}\label{sec:def:set:gen}\label{def:DLS}\label{sec:def:set}


In this \Cref{sec:def:set} we work temporarily with a more general set-up, namely, with a general first order language $\cL$, to recall some notions of \cite{CH-eval}.
Let $\cL$ be any language and let $\cS$ be any nonempty collection of $\cL$-structures. (A reader may at first choose to work with the case of $\cL$ being $\cL_D$ and  $\cS$ being $\cS_D$.) 
We recall the general notion of $\cS$-definable objects from \cite[section 2]{CH-eval}. 

\begin{defn}
By an
$\cS$-definable set
we mean a collection of sets
$$
X = (X_K)_{K\in \cS}
$$
such that there is an $\cL$-formula $\varphi$ with $\varphi(K)=X_K$ for all $K$ in $\cS$, with $\varphi(K)$ the solution set of the formula $\varphi$, namely the set of tuples $x$ in the structure $K$ such that $\phi(x)$ holds in $K$. By an $\cS$-definable function $f:X\to Y$ between $\cS$-definable sets $X$ and $Y$, we mean a collection of functions $ (f_K:X_K\to Y_K)_{K \in \cS}$ such that the collection of the graphs of the $f_K$ forms an $\cS$-definable set.
\end{defn}


\begin{defn}
For an $\cS$-definable set $X$, a pair $x=(x_0,K_0)$ with $x_0$ in $X_{K_0}$ and  $K_0$ in $\cS$ is called a \emph{point} on 
$X$. We also call $x$ an \emph{$\cS$-point}.
\end{defn}

\begin{defn}
Given a point $x=(x_0,K_0)$ on an $\cS$-definable set $X$, write $\cL(x)$ for the language $\cL$ expanded by constant symbols for the entries of the tuple $x_0$, and write $\cS(x)$ for the collection of $\cL(x)$-structures which are $\cL(x)$-expansions of $\cL$-structures in $\cS$ that  are elementarily equivalent to the $\cL(x)$-structure $K_0$. (In $K_0$, the tuple of the new constant symbols is interpreted by $x_0$.)
\end{defn}


Let $x=(x_0,K_0)$ be an $\cS$-point. By a small abuse of notation, we denote by $\{x\}$ the $\cS(x)$-definable set sending $K$ in $\cS(x)$ to  $\{x\}_K := \{x_{0,K}\}$, where $x_{0,K}$ is the interpretation in the $\cL(x)$-structure $K$ of the tuple of constant symbols introduced for $x_0$. With a similarly harmless abuse of notation, we also sometimes just write $x$ instead of $\{x\}$.
Likewise, if $g:X\to Z$ is an $\cS$-definable function and $x=(x_0,K_0)$ is a point on $X$, then we write $g(x)$ for the point $(z_0,K_0)$ on $Z$ with
$z_0=g_{K_0}(x_0)$. 

Consider an $\cS$-point $x$. Then any $\cS$-definable set $Z$
naturally determines an $\cS(x)$-definable set $Z_{\cS(x)}$, namely, $Z_{\cS(x)}$ associates to $K$ in $\cS(x)$ the set $Z_{K_{|\cL}}$, where $K_{|\cL}$ in $\cS$ is the $\cL$-reduct of $K$. Note that the notation $Z_{\cS(x)}$ is in close analogy to the usual notation for base change in algebraic geometry.
Similarly, any $\cS$-definable function $g:Z\to Y$ between $\cS$-definable sets determines an $\cS(x)$-definable function $g_{\cS(x)}:Z_{\cS(x)}\to Y_{\cS(x)}$ for any $\cS$-point $x$.
We use natural related notation, for example, for a point $y$ on $Y$, we write  $g^{-1}(y)$ for the $\cS(y)$-definable subset $(g_{\cS(y)})^{-1}(\{y\})$ of $Z_{\cS(y)}$, namely, for the pre-image of $\{y\}$ under $g_{\cS(y)}$.
When $p:W\subset X\times Y \to Y$ is the projection to the $Y$-coordinate with $X,Y,W$ some $\cS$-definable sets,  we sometimes write $W_y$ for 
$p^{-1}(y)$, which is an $\cS(y)$-definable subset of $W_{\cS(y)}$.
If furthermore $h:W\to V$ is an $\cS$-definable function, then we write $h(\cdot,y)$ for the restriction of $h_{\cS(y)}$ to $W_y$.

We employ the obvious terminology of set-theory, for Cartesian products, injections, surjection, bijections, fiber products. For example, we call an $\cS$-definable set $A$ nonempty if $A_K$ is nonempty for at least one $K$ in $\cS$.
We call a finite collection of nonempty $\cS$-definable sets $A_i$ a finite partition of an $\cS$-definable set $X$ if for each $K$ in $\cS$, the finite collection consisting of those sets $A_{i,K}$ which are nonempty forms a partition of $X_K$. 
We call an $\cS$-definable set $A$ a singleton if $A_K$ is a singleton for each $K$ in $\cS$.
An $\cS$-definable function $f:X\to Y$ is called a finite cover of $Y$ if $f$ is surjective and has finite fibers.

\begin{rem}\label{rem:class-set}
To avoid set-theoretical issues, one may either consider $\cS$ as a class (instead of as a set), and the operations on $\cS$ as class operations, or, one may fix a (suitably large) cardinal upper bound for the $\cL$-structures in $\cS$ and consider $\cS$ as a set.
\end{rem}

\begin{rem}\label{rem:assignment}
The terminology of $\cS$-definable set generalizes the terminology of definable subassignments from \cite{CLoes} \cite{CLexp}, and allows us to use the notion of evaluation of motivic functions as in \cite{CH-eval}.  
Descriptions of an $\cS$-definable set $X$ can be done in two simple ways: by giving an $\cL$-formula for $X$,  or, by specifying $X_K$ for each $K$ in $\cS$.
\end{rem}

\section{Motivic functions}\label{sec:motivic}

From now on, for the whole paper, let $\cL$ be a language containing $\cL_D$ from \Cref{defnLdac} and with the same sorts, and let $\cS$ be a nonempty collection of $\cL$-structures such that the $\cL_D$-reduct of any $K$ in $\cS$ belongs to $\cS_D$.
Furthermore, we suppose that $\cS$ is of one of the following forms: firstly $\cS$ is $\cS_D$, secondly, $\cS$ is $\cS_D(x)$ for some $\cS_D$-point $x$, and finally and most generally, $\cS$ is D-h-minimal as specified in \Cref{app:hensel-min}.
We speak about $\cS$-objects as in \Cref{def:DLS}.  


\subsection{Basic notation}\label{sec:basicnot}
We write $\VF$ for the $\cS$-definable set  $(\VF_K)_{K\in \cS}$. 
We write  $\cO$ for the $\cS$-definable set $(\cO_K)_{K\in \cS}$ with $\cO_K$ the valuation ring of $\VF_K$,  and similarly for $\cM$, $\VG$, $\VG_{\ge 0}$, $\RF$, etc. We sometimes write $K$ for $\VF_K$. For each basic term $t$ we write
\begin{equation}\label{rri}
\RR_t \quad  \mbox{ for }\quad   \big( \RR_{t_K,K}\big)_{K\in \cS}.
\end{equation}
We write $\{0\}$ for the $\cS$-definable set $R_{-1}$, and thus, one has $R_{-1,K}=\{0\}_K=\{0\}$ for each $K$ in $\cS$. 
More generally, for $\ell\ge 0$, we call a tuple $t = (t_1,\ldots,t_\ell)$ of basic terms a basic tuple, and we write
\begin{equation}
\RR_{t}\quad  \mbox{ for }\quad  \prod_{j=1}^\ell \RR_{t_j},
\end{equation}
which by convention denotes the $\cS$-definable set $\{0\}$ when $\ell=0$.

We sometimes write $\ZZ$ for $\VG$ and $\NN$ for $\VG_{\ge 0}$, when there is no confusion that they are considered as $\cS$-definable sets. 

\begin{defn}\label{def:im:DLS}
Let $X$ be an $\cS$-definable set. By an $\cS$-imaginary  function $g$ from $X$ to $\VF/\cM$ we mean a collection of functions $(g_K:X_K\to \VF_K/\cM_K)_{K\in \cS}$ such that $(G_K)_{K\in \cS}$ is an $\cS$-definable set, where $G_K\subset X_K\times \VF_K$ is the preimage of the graph of $g_K$ under the projection $X_K\times \VF_K\to X_K\times \VF_K/\cM_K$. 
\end{defn}

We already include one basic finiteness result for an $\cS(x)$-definable function $f: \RR_t \to \VG$ that we will need later in this section, namely for the definition of motivic functions in \Cref{defn:cCD}. Its proof will be given in \Cref{sec:weak.orth}.

\begin{prop}[Finiteness]\label{prop:finite0}
Let $t$ be a basic tuple and let $x$ be an $\cS$-point. Every $\cS(x)$-definable function $f: \RR_t \to \VG$ has finite image.
\end{prop}



\subsection{Motivic functions on a point}\label{subsec:mot:point} 

The ring of motivic functions $\cCM(x)$ with domain an $\cS$-point $x$ is fixed in \Cref{defn:cQDloc}, using \Cref{defn:cQD}. The intuitive meaning is explained in \Cref{sec:informal-meaning,sec:informal-meaning-CMexp}, and is closely related to the set-up from \Cref{sec:p-adic:mellin}. It generalizes and enriches the notion of the rings $\cCexp(x)$ from \cite{CH-eval}. 
Let us first define the Grothendieck ring part of $\cCM(x)$, denoted by $\cQ(x)$; it is generated by $x$-triples which we define first.

\begin{defn}\label{defn:gentrip}
Let $x$ be an $\cS$-point. 
By an \emph{$x$-triple} we mean a triple of the form
$$
(Z,g,h),
$$
where $Z$ is an $\cS(x)$-definable subset of $\RR_{t}$ for some basic tuple $t$,
$$
g:Z\to \VF/\cM
$$
is an $\cS(x)$-imaginary function,
and
$$
h:Z\to \RR_i
$$
is an $\cS(x)$-definable function  
for some 
basic tuple $i$.
\end{defn}



\begin{defn}\label{defn:cQD}
Let $x$ be an $\cS$-point. 
Write
\begin{equation}\label{def:Q(x)}
\cQD(x)
\end{equation}
for the quotient of the free abelian
group generated by $x$-triples
$$
(Z,g,h),
$$
and divided out by the following relations (R1) up to (R4).

\begin{itemize}

\item[(R1)]\textbf{Isomorphisms}

We have the relation
\begin{equation*}\label{eq1}
(Z_1,g_1,h_1)  = (Z_2,g_2,h_2)
\end{equation*}
if there exists an $\cS(x)$-bijection $\zeta:Z_1\to Z_2$  such that $g_1 = g_2\circ \zeta$ and $h_1=h_2\circ \zeta$.\\

\item[(R2)]\textbf{Disjoint unions}

We have the relation
\begin{equation*}\label{eq2}
(Z_1 \cup Z_2, g,h)  
= (Z_1,g_{|Z_1},h_{|Z_1})  + (Z_2,g_{|Z_2},h_{|Z_2}) ,
\end{equation*}
where $Z_1$ and $Z_2$ are disjoint $\cS(x)$-definable subsets of  $\RR_{t}$ for some common basic tuple $t$.\\

\item[(R3)]\textbf{Trivial fibrations}

Consider an $x$-triple $(Z,g,h)$ with $Z\subset \RR_t$ for some basic tuple $t=(t_1,\ldots,t_\ell)$ and put $t'=(t_1+1,t_2,\ldots,t_\ell)$. Assume that $t_1\ge -1$ holds in $\cS(x)$.
Let $p:\RR_{t'}\to \RR_t$ be the map which applies $\res_{t_1+1,t_1}$ to the first coordinate and the identity map to the other coordinates of $\RR_{t'}$, put $Z':=p^{-1}(Z)$ and let $p'$ be the restriction of $p$ to $Z'$. 
Then we have the relation
\begin{equation*}\label{eq3}
 (Z',g\circ p',h\circ p' )  = (\RF \times Z, g\circ \pi, h\circ \pi ),
\end{equation*}
with $\pi$ the coordinate projection $\RF \times Z\to Z$. \\

\item[(R4)]\textbf{Additive character relation for $g$}

Consider an $x$-triple $(Z,g,h)$ with $Z\subset \RR_t$ for some basic tuple $t=(t_1,\ldots,t_\ell)$ and put $t'=(t_1+1,t_2,\ldots,t_\ell)$. Assume that $t_1\ge -1$ holds in $\cS(x)$. Let $p$, $Z'$ and $p'$ be as for relation (R3). Consider an $x$-triple $(Z',g',h')$ such that $h\circ p' =h'$, and such that for each $z$ in $Z$, the map $g'$ restricted to $p'{}^{-1}(z)$ is a bijection onto $( g(z)+\cO )/\cM$. Then we have the relation
\begin{equation*}\label{eq4}
(Z',g',h')  = 0.
\end{equation*}
\\

\item[(R5)]\textbf{Commutativity relation for $h$}

Let $(Z,g,h:Z\to \RR_i)$ be an $x$-triple for some basic tuple $i=(i_1,\ldots,i_\ell)$ and let
$$
\rho:\RR_i= \prod_{j=1}^\ell \RR_{i_j} \to \RR_{i'} = \prod_{j=1}^\ell \RR_{i'_j}
$$
be a coordinate permutation. Then there is the relation
\begin{equation*}\label{eq5}
(Z,g,h:Z\to \RR_i)  = (Z,g,\rho\circ h:Z\to R_{i'}).
\end{equation*}

\end{itemize}

For an $x$-triple $(Z,g,h)$, we write $[Z,g,h]$ for the corresponding element in $\cQD(x)$.
We denote the subgroup of $\cQD(x)$ generated only by generators of the form $(Z,0,h)$ (namely with $g=0$) by $\cQDo(x)$. Elements of $\cQDo(x)$ are without dependence on the motivic additive character.
\end{defn}

\begin{remark}\label{sec:informal-meaning}
Before proceeding, let us explain the intended $p$-adic meaning of the objects $[Z,g,h]$ and the naturality of the relations (R1) -- (R4). This is slightly more general than the set-up for the rings of motivic functions from \Cref{sec:p-adic:mellin}. Fix a prime $p$ and suppose that  $\QQ_p$ carries an $\cL$-structure so that $\QQ_p$ belongs to $\cS$. 
Choose an additive character $\psi:\QQ_p\to\CC^\times$ which is trivial on $p\ZZ_p$ and non-trivial on $\ZZ_p$. Choose a commutative $\CC$-algebra $G$ with unit $1_G$ (possibly just $G=\CC$), and
for each basic term $t$ choose a function $H_t: \RR_{t_{\QQ_p}}\to G$.
Let $x$ be an $\cS$-point  of the form $(x_0,\QQ_p)$ and let $[Z,g,h:Z\to \RR_i]$ be in $\cQD(x)$ for some 
basic tuple $i=(i_{j})_{j=1}^\ell$. 
Then $[Z,g,h]$ is 
a motivic object behind the concrete finite sum 
\begin{equation}\label{eq:sum-Qp}
\sum_{z\in Z_{\QQ_p}} \Big( \psi(g_{\QQ_p}(z)) 
H_{i_1}(h_{\QQ_p,1}(z)) \cdot \ldots \cdot H_{i_\ell}(h_{\QQ_p,\ell}(z))  \cdot 1_G \Big),
\end{equation}
calculated in $G$, and where $h_{\QQ_p,j}$ are the component functions of $h_{\QQ_p}$. 
This is a finite exponential sum inside $G$, since $Z_{\QQ_p}$ is a finite set.
The motivic object $[Z,g,h]$ expresses this sum abstractly, in particular without fixing $\psi$ and the $H_t$ and the $\cL$-structure on $\QQ_p$.  The $p$-adic meaning of relation (R4) boils down to the orthogonality relation $\sum_{x\in\FF_p}\psi(x)=0$ for non-trivial additive characters on $\FF_p$.    The other relations have a natural and easy meaning at this $p$-adic level, and the sums of the form \Cref{eq:sum-Qp} are clearly preserved under the relations (R1) -- (R4).
Note that the ring structure on $\cQD(x)$ from \Cref{ring} below is compatible with this intended $p$-adic meaning.
The $p$-adic meaning uses choices of $\psi$, the $H_t$, and $\cL$-structure, and the set-up works uniformly over all choices as above.
\end{remark}


We equip $\cQD(x)$ with a ring structure via the following lemma.

\begin{lem}\label{ring}
Let $x$ be an $\cS$-point. 
One gets a unique ring structure on  the group $\cQD(x)$ from \Cref{defn:cQD} by defining multiplication on generators by
$$
[Z,g,h]\cdot[Z',g',h'] := [Z\times Z', g+g',  (h,h') ],
$$
where $Z \times Z'$ is the Cartesian product of $Z$ with $Z'$ with coordinate projections $p_Z:Z \times Z'\to Z$ and $p_{Z'}:Z \times Z'\to Z'$, where we write $g+g'$ for $g\circ p_Z + g'\circ p_{Z'}$ and where
$(h,h')$ sends $(z,z')$ to $(h(z),h'(z'))$.
Furthermore, $\cQDo(x)$ becomes a sub-ring of $\cQD(x)$. The rings  $\cQDo(x)$ and $\cQD(x)$ are commutative
rings with unit.
Here, we identify $(h,h')$ with $h$ when $h'$ is a function to $\RR_i$ with $i$ the empty tuple, and likewise we identify $(h,h')$ with $h'$ when the codomain of $h$ is $\RR_i$ with $i$ the empty tuple.
\end{lem}
\begin{proof}
Clearly the multiplication operation on generators yields a 
ring structure on the
free abelian group on the symbols $(Z,g,h)$ 
divided out only by relations of the form (R1). 
The subgroup generated by
the relations (R2), (R3), (R4) and (R5) is an ideal of this
ring, hence, the quotient by this subgroup is a ring. That $\cQDo(x)$ is a subring is clear.

The rings $\cQDo(x)$ and $\cQD(x)$ have zero element $[\emptyset,\emptyset,\emptyset] = 0$ and unit $[\{0\},0,0]=1$, where we write $[\{0\},0,0]$ for   $[\{0\},0,h:\{0\}\to \{0\}]$, with the $\cS$-imaginary function taking constant value $0$ in $\VF/\cM$ denoted by $0$, and where the codomain of $h$ is $R_i$ with $i$ the empty tuple (note that $R_i=\{0\}$ by convention). The commutativity follows from (R5).
\end{proof}






Write $\LL$ for $[\RF,0,0]$ in $\cQD(x)$, that is, for $[\RF,0,h:\RF\to \{0\}]$,  with notation as in the previous proof. Informally speaking, $\LL$ is the class of the affine line over the residue field. 
As a more suggestive notation for generators of $\cQD(x)$, we sometimes write
$$
[Z,E(g)\cdot H(h)]
$$
instead of $[Z,g,h]$.
This notation
lends itself well to some natural abbreviations: We omit $E(g)$  when $g=0$ on $Z$; we omit $H(h)$ when $h:Z\to R_i$ with $i$ the empty tuple, and we omit $Z$ and the brackets in the case that $Z$ is just $\{x\}$.

\begin{rem}\label{rem:non-commutative}
One could also choose to omit the relations of the form (R5), and doing so the altered ring $\cCM(x)$ would be non-commutative. This change would allow to consider non-commutative algebra $G$ instead of the commutative ones from \Cref{sec:informal-meaning}.
\end{rem}


\bigskip

We now define the ring of motivic functions $\cCM(x)$ with domain an $\cS$-point $x$.

\begin{defn}\label{defn:cQDloc} Let $x$ be an $\cS$-point and write $\LL$ for $[\RF ,0,0]$ in $\cQD(x)$. 
Write
$$
\cCM(x) 
$$
for the localisation of the polynomial ring
$$
\cQD(x)[T_0, T_1,T_2,T_3,\ldots]
$$
with respect to $\LL$,
the $T_i$ for all integers $i\ge 0$, and the elements of the form
$$
1-\LL^a\prod_{j\in J}T_j^{b_j},
$$
for integers $a<0$, $b_j\ge 0$ and a finite set $J\subset \NN$.

Write
$$
\cQDoloc(x)
$$
for the localisation of $\cQDo(x)[T_0,T_1,T_2,\ldots,]$ by the very same elements $\LL$, $T_i$ for all $i>0$, and $1-\LL^a T^b$, with the multi-index notation $T^b$ for $\prod_{j\in J}T_j^{b_j}$, integers $a<0$, $b_j\ge 0$ and finite set $J\subset \NN$.

Denote the subgroup of $\cQD(x)$ generated only by generators $[Z,g,0]$ (namely with $h=0$) by $\cQD_0  (x)$, and its localisation by the elements $\LL$ and $1-\LL^a$ for integers $a<0$ by
$$
\cC^{\mathrm{exp}}(x).
$$

Finally, we denote  the subgroup of $\cQD(x)$ generated only by generators $[ Z,0,0]$ by $\cQD_0  ^0(x)$, and its localisation by elements $1-\LL^a$ for integers $a<0$  by
$$
\cC(x).
$$


\end{defn}
By a slight abuse of notation, we write $[Z,g,h]$ for elements of $\cQD(x)$ as well as $\cCM(x)$, and similarly for $\LL$, standing for $[\RF ,0,0]$ in both $\cQD(x)$ and $\cCM(x)$.

The rings $\cCM(x)$ for $x$ varying in an $\cS$-definable set $X$ are the value rings for the motivic functions on $X$, as defined in the next section, see \Cref{defn:cCD}.


\begin{remark}\label{sec:informal-meaning-CMexp}
Let us explain the intended $p$-adic meaning behind the motivic functions $f$ in $\cCM(x)$, building on \Cref{sec:informal-meaning}. Let $p$, $\QQ_p$, $x$, $\psi$ and the $H_t$ be as in \Cref{sec:informal-meaning}, and let $\tau_i$ for $i\ge 0$ be nonzero complex numbers of modulus at most
$1$. Write $\tau=(\tau_i)_i$.
Then, in the important case that $H$ takes values in $G=\CC$, any motivic function $f$ of $\cCM(x)$ determines a complex number $f(x,\psi,H,\tau)$, by \Cref{eq:sum-Qp} for generators of $\cQD(x)$ and by evaluating $T_i$ at $\tau_i$. Note that $\LL$ simply becomes $p$ under this $p$-adic specialisation. In fact, a motivic function $f$ in  $\cCM(x)$ is a place holder for all these complex values $f(x,\psi,H,\tau)$ obtained by choosing $\psi$, $H$ and $\tau$. In general, if $H$ takes values in $G$, then $f(x,\psi,H,\tau)$ is an element of $G$ obtained as in \Cref{eq:sum-Qp} and by evaluating $T_i$ at $\tau_i$. As already mentioned 
from \Cref{defn:cCD} on we will vary the choice of $x$ inside an $\cS$-definable set. By varying the $\cL_D$-structure on $\QQ_p$, also the $\lambda$ in $\Lambda$ become variables (running over $\ZZ$), as already explained in \Cref{sec:p-adic:mellin}. The $T_i$ and the $h_i$ are there for the multiplicative (quasi-)characters, and the function $g$ is there as argument for the additive character. In this remark and in \Cref{sec:informal-meaning}, the $H_t$ are general functions from $R_{t_{\QQ_p}}$ into $G$, but they could in particular be chosen as multiplicative characters from $R_{t_{\QQ_p}}^\times$ to $\CC^\times$, extended by zero outside the units in $R_{t_{\QQ_p}}$.
\end{remark}




\subsection{Motivic functions}\label{subsec:mot-fun} 
We now come to our main definition of motivic functions on general domains, by considering $\cCM(x)$ for $x$ varying in the domain. 

For $X$ an $\cS$-definable set, write
$$
\cF(X)
$$
for the $\ZZ$-algebra of all functions on the points of $X$ that send a point $x$ on $X$ to an element of $\cCM(x)$ from \Cref{defn:cQDloc}. 


\begin{defn}[Motivic functions on $X$]\label{defn:cCD}
Let $X$ be an $\cS$-definable set. 
Define
$$
  \cCM(X) 
$$
as the 
sub-algebra of $\cF(X)$
generated by all functions of the form $F_1$ from (G1) and $F_2$ from (G2):\\

\begin{itemize}

\item[(G1)]  For any integer $a<0$ and any $b \in \NN^J$ for a finite set $J\subset \NN$, the function $F_1$ sending a point $x$ on $X$ to the element
    $$
    F_1(x) = \frac{1}{1-\LL^a T^b}
    $$
    of $\cCM(x)$, with the multi-index notation $T^b$ for $\prod_{j\in J}T_j^{b_j}$. We denote this function $F_1$ by $\frac{1}{1-\LL^a T^b}$.
    \\



\item[(G2)] 
For any $\cS$-definable set $Z\subset X\times \RR_{t}$ with $t$ a basic tuple,  
any $\cS$-imaginary function $ g:Z\to \VF/\cM$ and any
$\cS$-definable functions
$$
\alpha:Z\to\ZZ^n,\ \beta:Z\to \ZZ,\ \gamma:Z\to \ZZ^J,  \mbox{ and } h:Z\to \RR_i,
$$
with $n\ge 0$, 
a finite set $J\subset \NN$, and 
basic tuple $i$, the function $F_2$ sending a point $x$ on $X$ to the
unique element $F_2(x)$ of $\cCM(x)$ with the following property. If one partitions $Z_x$ into finitely many parts $Z_m$ for $m=1,\ldots,M=M(x)$ such that on each part  $Z_m$ each of the functions $\alpha, \beta,\gamma$ is constant, say, with respective values $\alpha_m\in\ZZ^n, \beta_m\in\ZZ,\gamma_m \in\ZZ^J$, then
$$
F_2(x) = \sum_{m=1}^M     (\prod_{i=1}^n \alpha_{m,i})   \LL^{\beta_m}  (\prod_{j\in J} T_j^{\gamma_{m,j}})   [   Z_m, g_{|Z_m},  h_{|Z_m}].
$$
We denote this function $F_2$ by $[Z,\alpha,\beta,\gamma,g,h]$ or $[Z,f]$ with $f$ is the tuple $(\alpha,\beta,\gamma,g,h)$, or, more suggestively, by
 \begin{equation}\label{eq:gen}
[Z,(\prod_{i=1}^n\alpha_i) \LL^\beta T^\gamma E(g) H(h)].
\end{equation}
That such a finite partition $Z$ into the parts $Z_m$ exists follows from the finiteness property of \Cref{prop:finite0}. That the element $F_2(x)$ of $\cF(X)$ does not depend on the choice of partition of $Z$ follows from relation (R2) and refining any given two partitions.
\end{itemize}



One defines the subring
$$
\cCoM(X )
$$
of $\cCM(X )$ by restricting in (G2) to generators of the form $[Z,(\prod_{i} \alpha_{i})\LL^{\beta} T^\gamma H(h)]$,
that is, like in (\ref{eq:gen}) but with $g=0$. One defines the subring
 $$
 \cCeM(X )
 $$
 of $\cCM(X )$ by restricting in (G2) to only those generators as in (\ref{eq:gen}) for which there exists a basic term $t$ such that
$\ordalt (g) \ge t $ holds on $Z$.  
We define the subrings $\cCexp(X )$ and $\cC(X )$ in the corresponding ways, namely with $\gamma=h=0$, resp.~with $\gamma=g=h=0$.  Finally, $\cCe(X )$ is the intersection of $\cCexp(X )$ with $\cCeM(X)$.
\end{defn}

Note that any element $F$ of $\cCM(X)$ for $X$ an $\cS$-definable set can be evaluated at each $\cS$-point $x$ on $X$, since $\cCM(X)$ is a sub-algebra of $\cF(X)$; we denote this evaluation of $F$ at $x$ by $F(x)$.

\begin{lem}\label{lem:id:in}
Let $X$ be an $\cS$-definable set. There are the following injective ring homomorphisms coming from inclusions
\[
\begin{tikzcd}
\cC(X) \arrow[r] \arrow[d] & \cCe(X) \arrow[r] \arrow[d] & \cCexp(X) \arrow[d] \\
\cCoM(X) \arrow[r] & \cCeM(X) \arrow[r] & \cCM(X)
\end{tikzcd}
\]
\end{lem}
\begin{proof} 
By the definitions of these rings.
\end{proof}

When the domain is a single $\cS$-point $x$, there is in fact no difference between $\cCeM(x)$ and $\cCM(x)$, as given by the following lemma.

\begin{lem}\label{lem:idx}Let $x$ be an $\cS$-point and consider the $\cS(x)$-definable set $\{x\}$.
There are natural isomorphisms of rings $\cCM(x)\to \cCM(\{x\})\to \cCeM(\{x\})$ from \Cref{defn:cQDloc,defn:cCD}. Likewise, there are natural isomophisms of rings $\cCoM(x)\to \cCoM(\{x\})$, $\cCexp(x)\to \cCexp(\{x\})\to \cCe(\{x\}) $, and   $\cC(x)\to \cC(\{x\})$.
\end{lem}
\begin{proof} 
By the definitions and \Cref{prop:finite0}. Indeed, the relations (R1) -- (R5)  hold for one structure in $\cS(x)$ if and only they hold for all structures in $\cS(x)$ (note that all these structures are elementarily equivalent), the elements by which to localise clearly correspond, and the generators $[Z,f]$ with their respective extra conditions 
also clearly correspond. By \Cref{prop:finite0}, the image of $\ordalt g$ is bounded below when the domain of $g$ is the $\cS(x)$-definable set $R_t$ with $t$ a basic tuple.
\end{proof}
By \Cref{lem:idx}, it 
remains harmless to denote $\{x\}$ sometimes by $x$. 

One may also introduce notation $\cCexp_T$ for when just $h=0$ for generators as in (G2), and correspondingly $\cCe_T$ and $\cC_T$, with similarly natural inclusions as in \Cref{lem:id:in} and identifications as in \Cref{lem:idx}.  


\subsection{Pull-backs of motivic functions}\label{sec:pull}

For any $\cS$-definable function $f:Y\to X$ between $\cS$-definable sets $X$ and $Y$, we can easily define the pull-back map
$$
f^*:\cCM(X ) \to \cCM(Y )
$$
by composition on the generators of $\cCM(X )$, as follows.
One defines $f^*(\frac{1}{1-\LL^a T^b})$ as  $\frac{1}{1-\LL^a T^b}$ seen inside $\cCM(Y )$. 
Secondly and finally, one defines
$$
f^*([Z,(\prod_{i=1}^n\alpha_i) \LL^\beta T^\gamma E(g)H(h)])
$$
by taking the fiber product of $Z$ and $X$ over $Y$, more precisely by
$$
[Z\times_Y X , (\prod_{i=1}^n\alpha'_i)  \LL^{\beta'} T^{\gamma'} E(g')H(h')]
$$
and where $g'=g\circ p_Z$ with $p_Z$ the projection $Z\times_Y X \to Z$, and similarly for $\alpha',\beta'$, $\gamma'$ and $h'$.
Clearly $f^*$ is a ring homomorphism which restricts to the subrings from \Cref{lem:id:in}
$$
\cCoM(X )\to \cCoM(Y ) \mbox{  and }\cCeM(X )\to \cCeM(Y ),\ \mathrm{etc.},
$$
which we also denote by $f^*$.

In the intended $p$-adic meaning, this just becomes the usual composition, see \Cref{sec:informal-meaning:pullback}. Even more generally, this is composition, see \Cref{lem:comp}.

\begin{defn}\label{defn:eval}
For $X$ an $\cS$-definable set, $F$ in $\cCM(X)$, and $x$ a point on $X$, define $F_{\cS(x)}$ in $\cCM(X_{\cS(x)})$ by writing $F$ as a sum of products of generators and replacing every $\cS$-definable object in these generators by the corresponding $\cS(x)$-definable object. That this is independent of the choices follows from the definitions.
\end{defn}

\begin{rem}\label{rem:eval}
Note that for $F$ in $\cCM(X)$ and $x$ a point on $X$, the element $F(x)$ of $\cCM(x)$ equals the pull-back of $F_{\cS(x)}$ under the inclusion map $\{x\}\to X_{\cS(x)}$ of $\cS(x)$-definable sets. 
Furthermore, evaluation of $F$ at $x$ preserves the subrings from \Cref{lem:id:in,lem:idx}: if $F$ lies in $\cC(X )$, then $F(x)$ lies in $\cC(x)$, etc.
\end{rem}

The following remark gives some informal explanation for the pull-back in the $p$-adic case corresponding to \Cref{sec:informal-meaning}.
\begin{remark}\label{sec:informal-meaning:pullback}
Consider an $\cS$-definable function $f:Y\to X$ between $\cS$-definable sets $X$ and $Y$  and let $F$ be in $\cCM(X )$.
Choose a prime $p$, and suppose that  $\QQ_p$ with some $\cL$-structure belongs to $\cS$. Choose an additive character $\psi$ on $\QQ_p$  which is trivial on $p\ZZ_p$ and non-trivial on $\ZZ_p$, and $H$ and $\tau=(\tau_j)_j$ as in \Cref{sec:informal-meaning-CMexp}.
Let $f_{\QQ_p}:Y_{\QQ_p} \to X_{\QQ_p}$, $F_{\QQ_p}$ and $(f^*(F))_{\QQ_p}$ be the $p$-adic specializations for these choices, with $F_{\QQ_p}$  and $(f^*(F))_{\QQ_p}$ as explained in \Cref{sec:informal-meaning-CMexp}. Then, one has
$$
(f^*(F))_{\QQ_p}(y)= F_{\QQ_p}\circ f_{\QQ_p}(y),
$$
for each $y$ in $Y_{\QQ_p}$. Thus, pull-back becomes composition. 
\end{remark}

At the more formal level than in \Cref{sec:informal-meaning:pullback}, pull-back corresponds to composition as well:

\begin{lem}\label{lem:comp}
Consider an $\cS$-definable function $f:Y\to X$ between $\cS$-definable sets $X$ and $Y$  and let $F$ be in $\cCM(X )$. Then, for each point $y$ on $Y$ one has
$$
(f^*(F))(y)= (F(f (y))_{\cS(y)}.
$$
\end{lem}

\begin{proof}
This follows from the definitions, see also~\cite{CH-eval}.
\end{proof}







\section{Integrals and integrability via Fubini-Tonelli}\label{sec:iterated.int}

In this section we define and develop integration and integrability for our rings of motivic functions. We develop integration in three steps. First we treat integration over residue rings, then over the value group, and finally over the valued field. We prove a version of Fubini--Tonelli in this framework and define the motivic Mellin transform. We also compare our rings and integration with previous work.

This section contains mostly definitions and only sketches of proofs. Full proofs of all results will be given later, in \Cref{sec:int:rev}, once we have developed the necessary geometrical results on $\cS$-definable sets.

\subsection{Integration over residue rings}\label{sec:push}

We first define integration over residue rings, which is rather straightforward.

\begin{def-lem}[Relative integrals over residue rings]\label{lem:D-rel}
Let $Y$ and $X\subset \RR_t\times Y $ be $\cS$-definable sets, for some basic tuple $t$ and write $p:X\to Y$ for the projection.
Then there is a unique ring homomorphism
$$
p_! : \cCM(X ) \to\cCM(Y )
$$
such that 
for each function $F$ in $\cCM(X)$ of the form
\begin{equation}\label{eq:gen-group}
\left(\prod_{i=1}^n  \frac{1}{1-\LL^{a_i} T^{b_i}}\right) \cdot [Z,f],
\end{equation}
for some $n\ge 0$ and with notation for generators of the forms  (G1) and (G2) from \Cref{defn:cCD}, one has
$$
p_!(F) = \left(\prod_{i=1}^n  \frac{1}{1-\LL^{a_i} T^{b_i}}\right) \cdot [Z,f],
$$
thus, the same expression as for $F$, but now taken inside $\cCM(Y)$. 
\end{def-lem}

\begin{rem}\label{rem:D-rel}
In \Cref{lem:D-rel}, if $Y$ is a singleton, then for  $F$ in $\cCM(X)$ we also write $\int_X F$ or $\int_{x\in X}F(x)$ for $p_{!}(F)$, and we further sometimes write $\int_{\RR_t}F$ or $\int_{x\in \RR_t}F(x)$ for $\int_X F$, with a small abuse of notation and where we tacitly extended $F$ by zero outside $X$ inside $\RR_t$. This value $\int_X F$ is called the integral of $F$ over $X$, or also the integral of $F$ in the fibers of $p$.
\end{rem}

\begin{rem}\label{rem:int-D}
The main difference between the expression interpreted in $\cCM(X)$ versus in $\cCM(Y)$ in \Cref{lem:D-rel} is that the relations have different meanings. In particular, one considers $Z$ as an $\cS$-definable subset of $\RR_{t'}\times X$ for some $t'$ in the first case, and in the second case, as a subset of $\RR_{t'}\times \RR_t\times Y = \RR_{(t',t)}\times Y$.
\end{rem}

\begin{proof}[Proof of \Cref{lem:D-rel}]
The uniqueness and existence of the  map $p_!$ as a group homomorphism follows from the fact that functions $F$ of the form (\ref{eq:gen-group}) generate $\cCM(X)$ as a group, and that the relations defining $\cCM(X)$ are compatible with the relations defining $\cCM(Y)$. That $p_!$ is a ring homomorphism is clear.
\end{proof}

\subsection{Integration over one $\VG$-variable}\label{sec:int:VG}

Integration over $\VG$ is close to the usual summation of countably many terms, more precisely, to geometric series, and the derivatives of geometric series like $\sum_{x \ge 0} x q^{-x}$ with $q>1$.

The following lemma is classical.
\begin{lem}\label{lem:geom}
Consider integers $a\ge 0$, $b<0$, and $c\in\NN^J$ for some finite set $J\subset \NN$.
Then there exists a unique rational function $G$ in the localisation of $\ZZ[q,\tau]$ with respect to 
the elements $1-q^n\tau^m$ with integer $n<0$, $m \in \NN^J$ and where $\tau=(\tau_j)_{j\in J}$, such that
\begin{equation}\label{f:sum:int:prop.1.1.0}
G(q,\tau) = \sum_{x\in \NN} x^a \cdot  q^{bx} \cdot \tau^{x\cdot c}
\end{equation}
for all real $q>1$ and all nonzero complex numbers $\tau_j$ with $|\tau_j|\le 1$, and where $\tau^{x\cdot c}$ stands for $\prod_{j\in J}\tau_j^{xc_j}$ and $\tau^m$ for $\prod_{j\in J}\tau_j^{m_j}$.
\end{lem}
\begin{proof}
This follows for example from Lemma 4.4.3 of \cite{CLoes}.
\end{proof}


\begin{lem}\label{lem:torsion:unique}
Let $y$ be an $\cS$-point.
Let $X$ be the $\cS(y)$-definable set $\VG_{\ge 0}$. Write $\id$ for the identity function $X\to X$.
Consider $F$ in $\cCM(X)$ of the form
\begin{equation}\label{f:sum:int:prop.1.0.0}
\sum_{(a,b,c)\in L}d_{a,b,c} \cdot  \id^{a} \cdot  \LL^{b\id} \cdot T^{\id\cdot c} 
\end{equation}
for some nonzero 
$d_{a,b,c}\in \cCM(\{y\})$, and finite sets $J\subset \NN$ and $L\subset \NN\times \ZZ\times \ZZ^{J}$. 
Suppose that $d_{a,b,c}$ is non-torsion for each $(a,b,c)\in L$ with $a>0$.
Then $F$ is zero if and only if $L$ is empty. In particular, the set $L$ of occurring exponents does not depend on the way of writing $F$ as in (\ref{f:sum:int:prop.1.0.0}).
\end{lem}

\begin{proof}[Sketch of proof.]
The idea of the proof is to induct on the degree of $F$, by which we mean the maximum of the exponent tuples $(a,b,c)$ appearing in $L$ (for some lexicographic order). By considering the function $F(n+1)-F(n)$ and similar expressions, one reduces to lower degree. The full proof will be given in \Cref{sec:weak.orth}.
\end{proof}

\begin{remark}
Let us amend on \cite{CH-eval} and correct an unauthorized omission. In Proposition 4.1 of \cite{CH-eval}, one should add the condition that the $c_{a,b,i,j}$ are non-torsion whenever $a_1+\ldots+a_{r_i}>0$. This is needed for the ``furthermore'' part of the statement. The correction for the proof is like the proof of \Cref{lem:torsion:unique}, and the applications of Proposition 4.1 of \cite{CH-eval} are unaffected by this extra condition on non-torsion.
\end{remark}

We define the integral over $\VG$ in three steps, and first for a special case based on \Cref{lem:geom,lem:torsion:unique}.

\begin{def-lem}[Special integral over $\VG_{\ge 0}$]\label{lem:sum:NN}
Let $y$ be an $\cS$-point.
Let $X$ be the $\cS(y)$-definable set $\VG_{\ge 0}$. Write $\id$ for the identity function $X\to X$.
Consider $F$ in $\cCM(X)$ of the form
\begin{equation}\label{f:sum:int:prop.1.0}
d \cdot  \id^a \cdot  \LL^{b\id} \cdot T^{\id\cdot c} 
\end{equation}
for some 
nonzero $d\in \cCM(\{y\})$,  $a\in\NN,\ b\in \ZZ,\ c\in \ZZ^{J}$, and some finite set $J\subset \NN$.
If $a>0$, then suppose furthermore that $d$ is non-torsion. Then the way of writing $F$ as in (\ref{f:sum:int:prop.1.0}) is unique.
Furthermore, $F$ of this form is called \emph{integrable over $X$} if 
\begin{equation}\label{cond:L.1}
b<0 \mbox{ and } c \in \NN^J.
\end{equation}
If (\ref{cond:L.1}) holds, 
then the integral of $F$ over $X$, denoted by $\int_X F$, is defined as
\begin{equation}\label{f:sum:int:prop.1.1}
d\cdot  G(\LL,T),
\end{equation}
in $\cCM(y)$, with $G$ as given by \Cref{lem:geom} and with $T=(T_j)_{j\in J}$.
We also write $p_{!}(F)$ or $\int_{x\in X} F(x)$ or $\int_{x\in \VG} F(x)$ or $\int_{\VG} F$ for $\int_X F$, with $p:X\to\{y\}$ the $\cS(y)$-definable function from $X$ to  $\{y\}$, and where we have tacitly extended $F$ by zero outside $X$ in $\VG$.
\end{def-lem}
\begin{proof}
The only point is to show that, given nonzero $F$ as in (\ref{f:sum:int:prop.1.0}), the $a,b,c,d$ are unique. That the exponents $a,b,c$ are unique follows from \Cref{lem:torsion:unique}. That also $d$ is unique follows from evaluating $F$ at $1\in X$, and by noting that $\LL^{b} \cdot T^{c}$ is invertible in $\cCM(\{y\})$.
\end{proof}


\begin{def-lem}[Integration in one $\VG$-variable]\label{defn:int:VG}
Let $y$ be an $\cS$-point, $X\subset \VG$ be an $\cS(y)$-definable set and $F$ be in $\cCM(X)$. 
Write $F$ as a finite sum of some terms $F_j$.
Take finitely many $\cS(y)$-injections $f_{i}:\{y\}\to X$ for $i=1,\ldots,k'$, and $f_i:\VG_{\ge 0} \to X$, for $i=k'+1,\ldots, k$ and some $k\ge k'\ge 0$, 
such that the images of the $f_i$ form a partition of $X$, and such that $f_i^*(F_j)$ is either zero or of the form as in \Cref{f:sum:int:prop.1.0} for each $i>k'$ and each $j$. Such choices of $F_j$, $k',k$ and $f_i$ exist.

If these choices of writing can be made such that the integrability condition from \Cref{cond:L.1} is met for each nonzero $f_i^*(F_j)$ for each $i>k'$ and each $j$, then we call $F$ \emph{integrable over $X$} and we define $\int_X F$ in $\cCM(y)$ as the sum
\begin{equation}\label{f:sum:int:VG}
\sum_j\left( \sum_{i=1}^{k'} f_i^*(F_j)
+\sum_{i=k'+1}^{k}  \int_{\VG_{\ge 0}} f_i^*(F_j) \right),
\end{equation}
with the integral for any $i>k'$ as defined by (\ref{f:sum:int:prop.1.1}) when $f_i^*(F_j)$ is nonzero and by zero otherwise. The expression (\ref{f:sum:int:VG}) is independent of the choices made. We also write $p_!(F)$ or $\int_{x\in X} F(x)$ or $\int_{x\in \VG} F(x)$ or $\int_{\VG} F$ instead of $\int_X F$, with $p:X\to\{y\}$ the $\cS(y)$-definable function from $X$ to  $\{y\}$ and where we have again tacitly extended $F$ by zero outside $X$ inside $\VG$.
%
%
%
\end{def-lem}

The value $\int_X F$ from (\ref{f:sum:int:VG}) is called the integral of $F$ over $X$, or also the integral of $F$ in the fibers of $p$.
The proofs of \Cref{defn:int:VG,defn:int:VG} rely on a Presburger rectilinearization result for the value group, given below as \Cref{cor:rec.1} and its \Cref{cor:rec:CM.1}. As such, the proof of \Cref{defn:int:VG} will be given in \Cref{sec:int:rev}, since it relies on this Presburger rectilinearization result.



\begin{def-lem}[Relative integral in one $\VG$-variable]\label{lem:VG-rel}
Let $Y$ and $X\subset \VG\times Y $ be $\cS$-definable sets, $p:X\to Y$ be the projection, and $F$ be in $\cCM(X )$.  Suppose that for each point $y$ on $Y$, the function $F_{|p^{-1}(y)}$ is integrable over $X_y$ in the sense of \Cref{defn:int:VG}. Then we call $F$ \emph{integrable in the fibers of $p$} and in this case there exists a unique function $p_!(F)$ in $\cCM(Y )$ such that for each point $y$ on $Y$ one has that
\begin{equation}\label{f:sum:VG:rel}
p_!(F)(y)=\int_{X_y} F(\cdot,y).
\end{equation}
Here the right hand side is from \Cref{defn:int:VG} and where $F(\cdot,y)$ stands for $F_{|p^{-1}(y)}$ in $\cCM(X_y)$.
The motivic function $p_!(F)$ in $\cCM(Y )$ given by (\ref{f:sum:VG:rel}) is called the \emph{integral of $F$ in the fibers of $p$}.
\end{def-lem}

\begin{rem}\label{rem:VG-rel}
In \Cref{lem:VG-rel}, if $Y$ is a singleton, then we also write $\int_X F$ or $\int_{x\in X} F(x)$  or  $\int_\VG F$ or $\int_{x\in \VG} F(x)$  for $p_!(F)$ (where we tacitly extend $F$ by zero outside $X$).
\end{rem}



Note that the finite value of $k$ from \Cref{defn:int:VG} can be unbounded when $y$ moves in $Y$, but this is taken care of in a definable way in the proof of \Cref{lem:VG-rel}, which will be given in \Cref{sec:int:rev}.


\subsection{Integration over one $\VF$-variable}\label{sec:int:VF}

We now define the integral over a single $\VF$-variable by reducing it to integration over residue rings and over the value group.

\begin{def-lem}[Integration over one $\VF$-variable, the $\cCeM$ case]\label{lem:VF-cCeM}
Let $y$ be an $\cS$-point, $X\subset \VF$ be an $\cS(y)$-definable set, and let $F$ be in $\cCeM(X)$. 

Take an $\cS(y)$-injection $f : X'  \to X$ such that $f(X')$ equals $X$ minus a finite $\cS(y)$-definable set, and where moreover $X' \subset X\times \VG\times \RR_t$ for some basic tuple $t$ and $f$ comes from the projection $X\times \VG\times \RR_t \to X$. Write $Z$ for the image of $X'$ under the projection $p' : X' \to \VG\times \RR_t$. Suppose for each point $z$ on $Z$ that $p'^{-1}(z)\subset \VF$ is a closed ball, say, of valuative radius $\alpha(z)$ and, that there is a function $F'$ in $\cCeM(Z)$ such that $f^*(F)= p'^{*} (F')$. Such $f$ exists, and $\alpha: Z\to \VG$ is automatically an $\cS(y)$-definable function.
Consider the function $G = \LL^{-\alpha}\cdot F'$ in $\cCeM(Z)$ 
and write $p_Z: Z \to \RR_t$ for the projection.
If $G$ is integrable in the fibers of $p_Z$  (as in \Cref{lem:VG-rel}), then we call $F$ \emph{integrable over $X$}, and we define $\int_X F$ in $\cCeM(y)$ by the iterated integral
\begin{equation}\label{f:sum:int:VF:abs}
\int_{\xi \in \RR_t}  \int_{\VG} G (\cdot,\xi ),
\end{equation}
with integrals from  \Cref{lem:D-rel,lem:VG-rel}. 
This is independent of the choices made.  We also write $p_!(F)$ or $\int_{x\in X} F(x)$ or $\int_{x\in \VF} F(x)$ or $\int_{\VF} F$ instead of $\int_X F$, with $p:X\to\{y\}$ the $\cS(y)$-definable function from $X$ to $\{y\}$, where we tacitly extended $F$ by zero outside $X$, and which we call the \emph{integral of $F$ over $X$}, or the integral in the fibers of $p$.
\end{def-lem}

See below in \Cref{sec:int:rev}  for the fully detailed proof of \Cref{lem:VF-cCeM}; for now we just give a sketch.

\begin{proof}[Sketch of the proof of \Cref{lem:VF-cCeM}]
  That $f$ and $F'$ exist as desired follows from the cell decomposition \Cref{prop:cell.1}. It remains to show the independence of the choices. This follows from taking common refinements, similarly as in the proof of Lemma 9.1.3 of \cite{CLoes}.
\end{proof}


\begin{def-lem}[Integrability over one $\VF$-variable]
\label{lem:VF:cCM:int}
Let $y$ be an $\cS$-point, $X\subset \VF$ be an $\cS(y)$-definable set, $F$ be in $\cCM(X)$. 
Take a finite cover $\sigma : \widetilde X \to X$ for some $\cS$-definable set $\widetilde X \subset X\times \RR_t$ and some basic tuple $t$, and  where $\sigma$ comes from the projection $X\times \RR_t \to X$.
Suppose that there is $G$ in $\cCeM(\widetilde X)$ and an $\cS$-imaginary function $g:\widetilde X\to\VF/\cM$ such that $\sigma_! (G \cdot E(g)) = F$. Such a choice of $\widetilde X$ always exists.
Then we call $F$ \emph{integrable over $X$} if the above choices  of  $\widetilde X$ and  $G$ can be made such that $G(\cdot,\xi)$ is integrable over $\widetilde X_\xi$ for each point $\xi$ in $\RR_t$ as in \Cref{lem:VF-cCeM}.
\end{def-lem}

See below in \Cref{sec:int:rev} for the fully detailed proof of \Cref{lem:VF:cCM:int}.


We give an alternative criterion for integrability which is often easier to verify.

\begin{lem}[Criterion for Integrability over one $\VF$-variable]\label{lem:VF:int:crit}
Let $y$ be an $\cS$-point, $X\subset \VF$ be an $\cS(y)$-definable set, and let $F$ be in $\cCM(X)$. 
Then $F$ is integrable over $X$ if and only if $F$ can be written as a finite sum of terms $c_i \cdot [Z_i,f_i]$ with $c_i$ a finite product of generators of type (G1) and $[Z_i,f_i] = [Z_i,\alpha_i,\beta_i,\gamma_i,g_i,h_i]$ a generator of type (G2), such that moreover each function  $F_i := [Z_i,\alpha_i,\beta_i,\gamma_i,0,h_i]$ is integrable over $X$ in the sense of \Cref{lem:VF-cCeM}.
\end{lem}

See below in \Cref{sec:int:rev}  for the fully detailed proof of \Cref{lem:VF:int:crit};
for now let us just note that it is similar to the proof of \cite[Proposition 4.3 (c)]{CH-eval}.

\begin{def-lem}[Integration over one $\VF$-variable, the $\cCM$ case]\label{lem:VF:cCM}
Let $y$ be an $\cS$-point, $X\subset \VF$ be an $\cS(y)$-definable set, and let $F$ be in $\cCM(X)$. 
Suppose that $F$ is integrable over $X$, as in \Cref{lem:VF:cCM:int}. Take an $\cS(y)$-definable bijection $f : X'  \to X$
where moreover $X' \subset X\times \VG\times \RR_t$ for some basic tuple $t$, and $f$ comes from the projection $X\times \VG\times \RR_t \to X$. Write $Z$ for the image of $X'$ under the projection $p' : X' \to \VG\times \RR_t$. Suppose for each point $z$ on $Z$ that there is a function $F'_{z}$ in $\cCeM(z)$ such that $(f^*(F))_{|p'^{-1}(z)   } = p'^*_z (F'_{z})$, with $p'_z:p'^{-1}(z)\to \{z\}$ the projection. Such $f$ always exist.

Then there is a unique function $G$ in $\cCM(Z)$ such that $G(z) = \int_{p'^{-1}(z)}p'^*_z(F'_{z})$ for each point $z$ on $Z$, with integral as in \Cref{lem:VF-cCeM}. Then $G(\cdot,\xi)$ is integrable over $\VG$ for each each $\xi$ in $\RR_t$ and we define $\int_X F$ by the
iterated integral
\begin{equation}\label{f:sum:VF:abs:exp}
\int_{\xi \in \RR_t}  \int_{\VG} (G (\cdot,\xi)   ),
\end{equation}
with integrals from  \Cref{lem:D-rel,lem:VG-rel}. 
This is independent of the choices made.
\end{def-lem}

See below in \Cref{sec:int:rev}  for the fully detailed proof of \Cref{lem:VF:cCM}; for now we just give a sketch.

\begin{proof}[Sketch of proof of \Cref{lem:VF:cCM}]
We have to show the existence of such $f$, and of $p_!(F)$ in $\cCM(Y )$, as the uniqueness of the latter is clear. The existence of such $f$ follows from cell decomposition \Cref{prop:cell.1}. The main point to show the existence of $p_!(F)$ 
 is to preserve the boundedness as this is not immediately clear from the 
 choices of $F'_{z}$ for each $z$. This uses the Jacobian property and compatible decomposition from Addendum \ref{add.Jacprop} to the cell decomposition theorem.
\end{proof}

\begin{def-lem}[Relative integration over one $\VF$-variable]\label{lem:VF-rel}
Let $Y$ and $X\subset \VF\times Y $ be $\cS$-definable sets,  $p:X\to Y$ be the projection and $F$ be in $\cCM(X )$.
Suppose that for each point $y$ on $Y$, the function $F_{|p^{-1}(y)}$ is integrable over $X_y$ in the sense of \Cref{lem:VF:cCM:int}. Then we call $F$ \emph{integrable in the fibers of $p$} and in this case there exists a unique function $p_!(F)$ in $\cCM(Y )$ such that for each point $y$ on $Y$ one has
\begin{equation}\label{eq:int:VF:rel}
p_!(F)(y)=\int_{X_y}  F(\cdot,y),
\end{equation}
where the right hand side is as in \Cref{lem:VF:cCM} and where $F(\cdot,y)$ stands for $F_{|p^{-1}(y)}$ in $\cCM(X_y)$. 
We call $p_!(F)$ the \emph{integral of $F$ in the fibers of $p$}.
\end{def-lem}
\begin{rem}\label{rem:VF-rel}
In \Cref{lem:VF-rel}, 
if  $Y$ is a singleton, then we also write $\int_X F$ or $\int_{x\in X}F(x)$ or $\int_{\VF}F$ or $\int_{x\in \VF}F(x)$ for $p_!(F)$, where we tacitly extended $F$ by zero outside $X$ and where we call $\int_X F$ the integral of $F$ over $X$.
\end{rem}

See below in \Cref{sec:int:rev} for the fully detailed proof of \Cref{lem:VF-rel}.



%
%
%
%
%
%
%
%
%
%
%
%

\subsection{General integration}\label{sec:int:general}

General integration is defined by using a variant of Fubini-Tonelli, which is classically a stronger Fubini statement for non-negatively valued functions, and a combination of the integrals over $\VG$, $\RR_t$, and $\VF$ as already defined. By the presence of (abstract) characters, we are not able to work with non-negatively valued functions. Instead, we work with restrictions to subsets: the Tonelli-aspect of Fubini-Tonelli is captured by the characteristic functions $\11_A$ of varying $\cS$-definable sets $A$ in \Cref{thm:gen-rel}; in our setting, unlike for real-valued functions, there is no possible subdivision in non-negative and non-positive functions. 


\begin{defn}[Dimension and measure zero sets]\label{defn:measure:zero}
Consider nonnegative integers $n,m,d$ and a basic tuple $t$.
Let  $A\subset \VF^{n}\times \VG^{m}\times \RR_t$ be an  $\cS$-definable set and let $A'$ be the image of $A$ under the coordinate projection $\VF^{n}\times \VG^{m}\times \RR_t\to \VF^n$. Then $A$ is said to have \emph{dimension less than $d$} if there exists an $\cS$-definable function $\ell:A' \to\VF^d$ such that $\ell$ has finite fibers and such that $\ell(A')$ has empty interior in $\VF^d$.

If $A$ has dimension less than $n$ then we say that $A$ has \emph{measure zero}. 

\end{defn}

In line with the convention just above \Cref{rem:class-set}, the two conditions in \Cref{defn:measure:zero} for having dimension less than $d$ mean that  $\ell_K$ has finite fibers and $\ell_K(A'_K)$ has empty interior in $\VF_K^d$ for the valuation topology, for each $K$ in $\cS$. The properties of dimension in this context are well understood and of tame nature. In particular, the $\cS$-definable function $\ell$ can be taken to be simply a $\QQ$-linear map.

For $\cS$-definable sets $A\subset X$, write $\11_A$ for the characteristic function of $A\subset X$, thus sending a point $x$ on $X$ to $1$ if $x$ lies on $A$ and to $0$ if $x$ lies on $X\setminus A$.  This function $\11_A$ can be seen inside $\cCM(X)$, or its subrings from \Cref{lem:id:in}.

The following definition works via Fubini-Tonelli, where the Tonelli-aspect is captured by varying $A$.
\begin{def-lem}[Motivic integrals, absolute case]\label{thm:gen-abs}
Consider an $\cS$-point $y$, a basic tuple $t=(t_1,\ldots,t_r)$, and nonnegative integers $n,m,r$.
Let
$$
X\subset \VF^{n}\times \VG^{m}\times \RR_t
$$
be an $\cS(y)$-definable set and let $F$ be in $\cCM(X)$. 
We define the integrability condition on $F$ and under that condition we define $\int_X F$ in $\cCM(y)$, the \emph{(motivic) integral of $F$ over $X$}, by induction on $n+m+r$.

If $n+m+r=1$ then we are already done, by \Cref{sec:push,sec:int:VG,sec:int:VF}.

If 
$n+m +r>1$, then write
$$
X\subset X_1\times X_2
$$
with
$$
X_1 = \VF^{n_1}\times \VG^{m_1} \times \RR_{(t_1,\ldots,t_{r_1})} 
$$
and
$$
X_2 = \VF^{n_2}\times \VG^{m_2} \times \RR_{(t_{r_1+1},\ldots,t_r)} 
$$
for some choice of nonnegative integers $n_i$, $m_i,r_i$ for $i=1,2$ with
$n=n_1+n_2,\ m = m_1+m_2,r_1\le r$
and with $n_1+m_1+r_1>0$ and $n_2+m_2+r_2>0$. Call $F$ \emph{integrable over $X$} if there is a measure zero $\cS(y)$-definable set $X_0\subset X_2$ such that for each $\cS'$-definable subset $A\subset X_{\cS'}$, with $\cS' := \cS(y)(a)$ and any  $\cS(y)$-point $a$, the following holds.

There is $G_A$ in $\cCM(X_{2,\cS'})$ which is integrable over $X_{2,\cS'}$, and
for each point $x_2$ on $(X_{2}\setminus X_{0})_{\cS'}$,
the function $\big(\11_A\cdot F_{\cS'}  \big) (\cdot,x_2)$ in $\cCM(X_{1,\cS'(x_2)})$ is integrable over 
$X_{1,\cS'(x_2)}$ with integral $\int_{x_1\in X_{1,\cS'(x_2)}}\big(\11_A\cdot F_{\cS'}  \big) (x_1,x_2)$ equal to $G_{A}(x_2)$.  In this definition, we  tacitly extend some functions by zero outside their natural domains.

If $F$ is integrable over $X$, then we define $\int_X F$ in $\cCM(y)$ by 
\begin{equation}\label{eq:gen:int}
\int_{x_2\in X_2\setminus X_0}\int_{x_1\in X_{1,\cS(x_2)}} F(x_1,x_2).
\end{equation}

This is independent of the choices, in particular, of the ordering of the coordinates on $X$, and the choices of $X_1,X_2$ and $X_0$.
\end{def-lem}


\begin{remark}
\label{add:int}
With data and notation from \Cref{thm:gen-abs}, it may be interesting to investigate whether to check the integrability of $F$ over $X$ it would be enough to check the corresponding integrability conditions of  \Cref{thm:gen-abs} only with $a=y$ (instead for all choices of $a$).  Note that $\cS'$ would then simply by $\cS(y)$. 
\end{remark}

\begin{proof}[Sketch of proof for \Cref{thm:gen-abs}]
The proof is reduced to several basic case, like for the case that $X$ is a subset of $\RR_t\times \RR_u$, of $\RR_t\times \VG$, $\RR_t\times \VF$, $\VG\times \VF$, and $\VF\times \VF$.
The proof for $\VF\times \VF$ uses a computation on bi-cells (that is, cells in $\VF^2$) from \Cref{prop:bi-cells} as in \cite{CLoes,CLexp}.
\end{proof}

We now come to the general definition of motivic integrals in the fibers of some general $\cS$-definable function $p:X\to Y$.

\begin{def-thm}[Motivic integrals]\label{thm:gen-rel}
Let $Y$ and
$$
X\subset \VF^{n}\times \VG^{m}\times \RR_t\times Y
$$
be $\cS$-definable sets  and let $F$ be in $\cCM(X )$ for some nonnegative integers $n,m$ and some basic tuple $t$.

Suppose that for each point $y$ on $Y$, the function $F_{|p^{-1}(y)}$ is integrable over $X_y$ in the sense of \Cref{thm:gen-abs}. Then we call $F$ \emph{integrable in the fibers of the projection $p:X\to Y$} and in this case there exists a unique function $p_!(F)$ in $\cCM(Y )$ such that for each point $y$ on $Y$ one has
\begin{equation}\label{eq:int:general:rel}
p_!(F)(y)=\int_{X_y}  F(\cdot,y),
\end{equation}
where the right hand side is as in \Cref{thm:gen-abs}. We call $p_!(F)$ the \emph{(motivic) integral of $F$ in the fibers of $p$}.
\end{def-thm}


\begin{rem}\label{rem:int-gen}
In \Cref{thm:gen-rel}, 
if  $Y$ is a singleton, then we also write $\int_X F$ or $\int_{x\in X}F(x)$ or $\int_{\VF^{n}\times \VG^{m}\times \RR_t}F$ or $\int_{x\in \VF^{n}\times \VG^{m}\times \RR_t}F(x)$ for $p_!(F)$, where we tacitly extended $F$ by zero outside $X$; we call $\int_X F$ the (motivic) integral of $F$ over $X$.
\end{rem}

The proof for \Cref{thm:gen-rel} will simply be a family version of the proof of \Cref{thm:gen-abs}, and similarly for both following \Cref{cor:Fubini,cor:Fubini-Ton}.

%



The following two corollaries say that the data witnessing integrability in \Cref{thm:gen-abs} can be taken definably in families.

\begin{cor}[Fubini]\label{cor:Fubini}
Let $X_1$, $X_2, Y$ and $X\subset X_1 \times X_2 \times Y$ be $\cS$-definable sets, let $F$ be in $\cCM(X )$ and write $p:X\to Y$ for the projection. Suppose that $F$ is integrable over $X$ in the fibers of $p$. Then there exists an $\cS$-definable set $X_0\subset X_2\times Y$ such that for each point $(x_2,y)$ on $ ( X_2\times Y ) \setminus X_0$ the following hold

1) $X_{0,y}$ has measure zero, 

2) $F(\cdot,x_2,y)$ is integrable over $p_1^{-1}(x_2,y)$, 
say, with integral $G(x_2,y)$,

3) $G$ belongs to $\cCM(X_{2}\times Y )$, $G$ is integrable in the fibers of $p_2$, and  
\begin{equation}\label{eq:Fub0}
p_{2!}(G) = p_!(F),
\end{equation}
with 
$p_1:X\to X_2\times  Y$ and $p_2:X_2\times  Y \to Y$   the projections. Here we have extended $F$ and $G$ by zero outside where they were not already defined.
\end{cor}

\begin{rem}\label{rem:fub:alt}
Note that \Cref{eq:Fub0} can be rewritten in the more familiar form as
$$
\int_{x_2\in X_{2,\cS(y)}} G(x_2,y) = \int_{(x_1,x_2)\in (X_1\times X_2)_{\cS(y)}}F(x_1,x_2,y),
$$ 
for each point $y$ on $Y$.
\end{rem}


\begin{cor}[Fubini-Tonelli]\label{cor:Fubini-Ton}
Let $X_1$, $X_2$ and $X\subset X_1 \times X_2 \times Y$ be $\cS$-definable sets, let $F$ be in $\cCM(X )$ and consider the projection $p:X\to Y$. Then $F$ is integrable over $X$ in the fibers of $p$ if and only if there exists an $\cS$-definable set $X_0\subset X_2\times  Y$ with measure zero fibers for $p_2$ such that for each $\cS$-point $a$ and each $\cS'$-definable subset $A\subset X_{\cS'}$, with $\cS' := \cS(a)$, the following holds.

There is $G_A$ in $\cCM((X_{2}\times  Y)_{\cS'})$ which is integrable over $X_{2,\cS'}$ in the fibers of the projection map to $  Y_{\cS'}$, and
for each point $(x_2, y)$ on $(X_{2}\times  Y \setminus X_{0})_{\cS'}$, 
the function $\big(\11_A\cdot F_{\cS'}  \big) (\cdot,x_2, y)$ in $\cCM( p_1^{-1}(x_2, y) )$ is integrable over 
$p_1^{-1}(x_2, y)$ with integral $\int_{x_1\in p_1^{-1}(x_2, y)}\big(\11_A\cdot F_{\cS'}  \big) (x_1,x_2, y)$ equal to $G_{A}(x_2, y)$.

If $F$ is integrable, then one has furthermore, for any such choice of $X_0$, $a$, $A$ and $G_A$,
\begin{equation}\label{eq:Fub}
p_{2,\cS'!}(G_A) = p_{\cS'!}(\11_A\cdot F_{\cS'} ),
\end{equation}
with 
$p_1:X\to X_2\times  Y$ and $p_2:X_2\times  Y \to  Y$   the projections. 
\end{cor}

The proofs of \Cref{cor:Fubini} and of \Cref{cor:Fubini-Ton} will be given in \Cref{sec:int:rev}.



For change of variables, we need a notion of $C^1$ maps and the Jacobian.

\begin{def-lem}[Jacobian]\label{lem:Diff}
Let $X_1$ and $X_2$ be $\cS$-definable subsets of $\VF^n$ for some $n\ge 0$, and let $f:X_1\to X_2$ be an $\cS$-definable bijection. For each $K$ in $\cS$, let $\mathrm{Diff}_{f,K}$ be the subset of $\VF^n_K$ consisting of those points $x=(x_0,K)$ on $X$ such that $f_K$ is $C^1$ at $x$, meaning that $x$ lies in the interior of $X_K$ and each of the first partial derivatives of $f_K$ exists around $x$ and is continuous at $x$. Then $\mathrm{Diff}_{f} = (\mathrm{Diff}_{f,K})_{K\in\cS} $ is an $\cS$-definable set, and the complement of $\mathrm{Diff}_{f}$ in $X_1$ has measure zero (it has dimension less than $n$). Furthermore, define $\Jac_{f,K}(x)$ for $x=(x_0,K)$ on $X$ as the determinant of the Jacobian matrix of $f_K$ at $x$ if $x$ lies in $\mathrm{Diff}_{f,K}$, and as $0$ otherwise.  Then $\Jac_f = (\Jac_{f,K})_{K\in\cS}$ is an $\cS$-definable function from $X_1$ to $\VF$.
\end{def-lem}
\begin{proof}
This will follow from the results of \Cref{sec:S}. 
\end{proof}

The following is the key form of the change of variables formula, as relative variants can be derived from it rather directly.

\begin{thm}[Change of variables]\label{thm:cov}
Let $X_1$ and $X_2$ be $\cS$-definable subsets of $\VF^n$ for some $n\ge 0$, and let $f:X_1\to X_2$ be an $\cS$-definable bijection. Let $F$ be in $\cCM(X_2)$.
Then one has
$$
\int_{X_1} f^*(F)\cdot \LL^{-\ord \Jac_f} = \int_{X_2}F,
$$
using the  notation from \Cref{lem:Diff}  and with the convention that $\LL^{-\ord ( 0 ) } = 0$.
\end{thm}

As for the relative case, we have the following.

\begin{cor}[Change of variables, relative case]\label{cor:thm:cov}
Let $Y$ and
\begin{align*}
X &\subset \VF^{n}\times \VG^{m}\times \RR_t\times Y,\\
X' &\subset \VF^n\times \VG^{m'}\times \RR_{t'}\times Y
\end{align*}
be $\cS$-definable sets for some nonnegative integers $n,m,m'$ and basic tuples $t,t'$ and consider the projection  $p:X\to Y$. Let $F$ be in $\cCM(X )$.
Suppose that $f: X\to X'$ is an $\cS$-definable bijection which commutes with the projections to $Y$. 

Then one has for each point $y$ in $Y$
$$
\int_{X_y} f^*(F)(\cdot,y)\cdot \LL^{-\ord \Jac_{f(\cdot,y)}}    =\int_{X'_y}F(\cdot,y),
$$
where $\Jac_{f(\cdot,y)}$ stands for the function sending a point $(x,z,y)=(x_0,z_0,y_0,K)$ with $(x,z)=(x_0,z_0,K)$ in $X_y$ and $x_0$ in
$\VF_K^n$ to  $\Jac_{f(\cdot,z,y),K}(x)$ from \Cref{lem:Diff} and with $\LL^{-\ord ( 0 ) } = 0$.
\end{cor}

The proofs of these change of variables formulas will be given in \Cref{sec:int:rev}.









\subsection{Motivic Mellin and Fourier transformation}\label{sec:mot-Mellin}

In \Cref{def-Mellin} we will define the motivic Mellin transform. We will later see that it specializes, in particular, to $p$-adic Mellin transforms of functions as in \Cref{sec:informal-meaning-CMexp} with $G=\CC$ and with the $H_t$ being multiplicative characters.
For an integrable motivic function $F$ in $\cCM((\VF^{\times})^n)$, we will write
$$
\cMmot(F)
$$
for its motivic Mellin transform, where $\VF^\times$ stands for the multiplicative units of the valued field, as usual.
More generally we will write
$$
\cMmot _{/Y}(F)
$$
when $F$ lies in $\cCM((\VF^{\times})^n\times Y )$ and is integrable in the fibers of the projection $p:(\VF^{\times})^n\times Y\to Y$, and call it the relative Mellin transform of $F$, over $Y$.

\begin{defn}[Motivic Mellin transforms]\label{def-Mellin} Let  $F$ be an integrable motivic function in $\cCM((\VF^{\times})^n)$. Let $0\le j_1< \ldots< j_n$ be large enough (given $F$), so that the $T_{j_i}$ and the $\lambda_{j_i}$ do not yet appear in $F$ in the sense that none of the $\lambda_{j_i}$ appear in any of the formulas defining the data in $F$ and none of the $T_{j_i}$ appear in any of the generators used in $F$.  
The \emph{motivic Mellin transform $\cMmot(F)$} of $F$ is defined as the function in $\cCM(\{0\})$ given by
$$
 \cMmot(F) := \int_{x\in (\VF^{\times})^n}  F(x) \prod_{i=1}^n T_{j_i} ^{\ord x_i}  H (\ac_{\lambda_{j_i}}(x_i)).
$$

More generally, let $F$ be in $\cCM((\VF^{\times})^n\times  Y )$ and assume that $F$ is integrable in the fibers of the projection $p_ Y:(\VF^{\times})^n\times  Y\to  Y$. Let $0\le j_1< \ldots< j_n$ be large enough so that the $T_{j_i}$ and $\lambda_{j_i}$ do not yet appear in $F$.
Then  $\cMmot _{/Y}(F)$ is the function in $\cCM(Y)$ given for each point $y$  in $ Y$ by
$$
\cMmot _{/Y}(F)( y) :=   \cMmot(F(\cdot,y)),
$$
where the right hand side is taken as above, with this choice of $T_{j_i}$ and $\lambda_{j_i}$.
\end{defn}


%
%
%

We note that the motivic Mellin transform remains within our framework, that is, the motivic Mellin transform is an element of $\cCM(\{0\})$. The $\lambda_{i_j}$ and $T_{i_j}$ in the definition of the Mellin transform should be thought of as variables. Our most important result around the Mellin transform states that taking Mellin transforms is injective.

\begin{thm}[Injectivity of the motivic Mellin transform]\label{thm:injM}
For $F_1$ and $F_2$ in $\cCM((\VF^{\times})^n)$, one has
$$
\cMmot(F_1)=\cMmot(F_2)
$$
if and only if $F_1=F_2$ almost everywhere, meaning that there is an $\cS$-definable set $X\subset (\VF^{\times})^n$ whose complement in $(\VF^{\times})^n$ has measure zero and such that $\11_{X}F_1=\11_{X}F_2$.
\end{thm}

Note that for this theorem to make sense, we are implicitly taking the same variables $\lambda_{i_j}$ and $T_{i_j}$ in defining both $\cMmot(F_1)$ and $\cMmot(F_2)$.

Our framework also includes motivic Fourier transforms, extending~\cite{CLexp}. If $F\in \cCM(\VF^n)$ is integrable, then we define the \emph{(motivic) Fourier transform} of $F$ as the function in $\cCM(\VF^n)$ which maps an $\cS$-point $y$ in $\VF^n$ to
\[
\cF(F)(y) = \int_{x\in \VF^n} F(x)E(x\cdot y).
\]
Similarly as the Mellin transform, the motivic Fourier transform is injective.
\begin{thm}[Injectivity of the motivic Fourier transform]\label{thm:injF}
Let $F_1, F_2\in \cCM(\VF^n)$ be integrable. Then $\cF(F_1) = \cF(F_2)$ if and only if $F_1 = F_2$ almost everywhere.
\end{thm}

We also have Fourier inversion for integrable motivic functions with integrable Fourier transform, see~\Cref{prop:Fourier.inversion}.

%






%
%
%
%
%
%


\subsection{Stability on some subrings and comparison}\label{sec:subrings}

Integration remains inside several natural subrings of $\cCM(X)$, as follows.

\begin{prop}[Stability on subrings]\label{thm:stab}

With notation of the general integral from \Cref{thm:gen-rel} for $X$, $Y$, $p$ and $F$, if $F$ furthermore lies in $\cCoM(X )$, resp.~in  $\cCeM(X )$, $\cCexp(X )$, $\cCe(X )$ or in $\cC(X )$,  then
$p_!(F)$ lies correspondingly in $\cCoM(Y )$, $\cCeM(Y )$, $\cCexp(Y )$, $\cCe(Y )$, or $\cC(Y )$.

If $F$ lies in $\cCeM(X)$, resp.~in $\cCoM(X)$, then $\cMmot _{/Y}(F)$ lies in $\cCeM(Y)$, resp.~in $\cCoM(Y)$.
\end{prop}
\begin{proof}
Clear by construction.
\end{proof}


We compare our framework with the earlier work in~\cite{CLoes} and~\cite{CLexp}.

\begin{prop}[Comparison with \cite{CLoes}, \cite{CLexp}]\label{result:cCexp-widehatC}
Let $\cL$ be the language $\cL_D$ and let $\cS$ satisfy  (C1) and (C2) from \Cref{app:hensel-min}.

Suppose that each $K$ in $\cS$ has residue field of characteristic zero, and for each $K$ in $\cS$, say with residue field $k$, there is a structure $K'$ in $\cS$ with $\VF_{K'}= k((t))$ and with the obvious $\cL$-structure on $k((t))$ satisfying $\ac_i(t)=1$ for all $i\ge 0$. Furthermore, let $T$ be the common theory of all occurring such $k$, in the language of rings.
Then any definable $T$-subassignment $X$ as defined in \cite{CLoes} gives an $\cS$-definable set $X_\cS$, and with notation from \cite{CLoes} \cite{CLexp} for $\cCexp(X)$ and its elements, one has an
injective ring homomorphism
\begin{equation}\label{eq:im:CLoes}
\cCexp(X) \to \cCM(X_\cS)
\end{equation}
sending $[Z,g,\xi]$ to $[Z_\cS,g_\cS+\xi_\cS,0]$, with $Z_\cS$, $g_\cS$, and $\xi_\cS$ the $\cS$-objects associated to $Z,g,\xi$, and $g_\cS+\xi_\cS$ the corresponding $\cS$-imaginary function from $X_\cS$ to $\VF/\cM$. Furthermore, these ring homomorphisms commute with integration, namely, if $p:X\to\{0\}$ is the projection to the point and $F$ is in $\cCexp(X)$, then
$$
p_!(F) = (p_{\cS})_!(F_{\cS}),
$$
with left hand side as in \cite{CLexp}. Furthermore, the image of the morphism (\ref{eq:im:CLoes}) lies in $\cCexp(X_\cS)$, with notation from this paper.
\end{prop}

One major difference between~\cite{CLoes, CLexp} and the current work is that we work with higher residue rings, while loc.\ cit.\ only works with the residue field. In particular, we only obtain injective ring morphisms, and not ring isomorphisms in \Cref{result:cCexp-widehatC}. Working with higher residue rings allows us in particular to include mixed characteristic, and to have a full Mellin transform.

One can also compare to the $p$-adic specializations of rational motivic functions of \cite{Kien:rational}
and even to the rational motivic functions of \cite{Kien:rational} but we leave this to the reader; the injectivity is less clear since the proof of the injectivity in \Cref{result:cCexp-widehatC} relies on \cite{CH-eval} which works with the frameworks of \cite{CLoes}, \cite{CLexp} and remains to be studied for the frameworks of \cite{Kien:rational}. Note that in \cite{CLoes}, further comparison results to the motivic integration frameworks from \cite{DLinvent} and \cite{DL} are given, and which thus also relate to the new class of $\cCM$ functions.

\section{Quantifier elimination}\label{sec:QE:Pres}

In the coming sections, we work towards a deeper understanding of the geometry of $\cS_D$-definable and $\cS$-definable sets and functions, which will form the backbone for the proofs of \Cref{sec:iterated.int}.

\subsection{Quantifier elimination}The point of this section is to obtain quantifier elimination of two kinds of quantifiers for $\cS_D$, namely, those running over the valued field, and those over the value group. This is different from the classical situation like in \cite{Pas} since we have a sort $\RR_t$ for each basic term $t$. However, it still resembles the more recent study of \cite[Section 5]{CHallp} with its treatment of arbitrary ramification.

Let us enlarge the language $\cL_D$ as follows:
\begin{enumerate}
\item For each basic term $t$ add  a symbol $A_t$ for the definable subset of $\RR_t$ consisting of elements of $\RR_t$ which have angular component one, namely, those $x$ in $\RR_t$ for which there exists a lift $\xi$ in $\cO$ for the projection $p_t : \cO \to \cO/B_t(0)$ with $\ac_t(\xi)=1$.
\item Add a symbol $\cross_t:\VG_{\geq 0}\to \RR_t$ for the function
sending  $i$ to $x$ in $A_t$ of valuation $i$ if such $x$ exists, and to $0$ otherwise. Note that the support of $\cross_t$ is bounded and contained in $[0, t]$.
\end{enumerate}

We denote the new language by $\cL'_{D}$. Note that the new symbols $A_t$ and $\cross_t$ are $\cL_D$-definable, and so every $K\in \cS_D$ is endowed naturally with an $\cL'_D$-structure, which we denote by $K'$.
Let $T'_D$ be the corresponding theory, namely of the $\cL'_{D}$-structures $K'$ for $K$ in $\cS_D$. 

\begin{thm}\label{QEZ}
The $\cL'_D$-theory $T'_D$ of the structures $K'$ with $K$ in $\cS_D$ eliminates quantifiers over the valued field and over the value group, in the language $\cL'_D$.
\end{thm}

\begin{proof}[Proof of \Cref{QEZ}]
We mimic the proof of quantifier elimination from \cite[Thm 4.1]{Pas} and \cite[Thm 5.1.2] {CHallp}.

We first note that the theory $T'_D$ has resplendent elimination of $\VF$-quantifiers, as follows from work by Basarab~\cite{Basarab} (or see ~\cite{Rid}). This implies that any $\cL'_D$-formula is equivalent to a formula of the form
\[
\varphi\left((\ord p_i(x))_i, (\ac_{t_i}(p_i(x)))_i, z, \xi\right),
\]
where $x$ is a tuple of $\VF$-variables, $z$ a tuple of $\VG$-variables, $\xi$ a tuple of variables in various residue rings $R_{t_i}$, the $p_i$ are in $\ZZ[x]$, the $t_i$ are basic terms, and $\varphi$ is a formula without $\VF$-quantifiers. Moreover, it suffices to show that we can eliminate one $\VG$-variable from such a formula. By looking at terms, it suffices to eliminate the quantifier $\exists y$ from a formula of the form
\[
\exists y\: \theta(y,z)\wedge \psi(\xi, \cross_{t_1}(s_1(y,z)), \ldots, \cross_{t_n}(s_n(y,z))),
\]
where $z$ is a tuple of $\VG$-variables, $\xi$ is a tuple of variables over some sorts $R_t$, the $t_i$ are basic terms, the $s_i$ are $\cL_{\Pres}(\Lambda)$-terms, $\theta$ is an $\cL_{\Pres}(\Lambda)$-formula, and $\psi$ is a formula without quantifiers over $\VF$ or $\VG$. We can assume that $y$ appears in each $s_i(y,z)$, since otherwise we can create an extra $\xi$-variable $\xi_{j(i)}$ and add a condition of the form
$\cross_{t_i}(s_i(z))=\xi_{j(i)}$ outside the quantifier $\exists y$.

Now, by refining $\theta(y,z)$ into finitely many cases we may assume that $\theta(y,z)$ implies, for each $i$, that $s_i(y,z)$ is in $[0, t_i]$. Indeed, if $s_i(y,z)$ is not in $[0, t_i]$, then we may simply replace $\cross_{t_i}(s_i(y,z))$ by $0$. If this eliminates all $s_i$, then we can use Presburger quantifier elimination to remove the quantifier $\exists y$ to conclude. Otherwise, letting $s_1(y,z) = ay + \alpha(z)$ for some non-zero integer $a$ and $\cL_{\Pres}(\Lambda)$-term $\alpha$, $\theta(y,z)$ implies that
\[
0\leq ay + \alpha(z) \leq t_1.
\]
We will further reduce to the case where $\theta(y,z)$ implies that $0\leq y\leq t$ for some basic term $t$. First note that we can always assume that $a\geq 1$. Next, we can reconstruct $\cross_{t_i}(s_i(y,z))$ from $\cross_{at_i}(as_i(y,z))$ in the following way. The element $\cross_{at_i}(as_i(y,z))$ is an $a$-th power in $A_{at_i}$, and the image under $\res_{at_i, t_i}$ of its $a$-th root in $A_{at_i}$ is exactly equal to $\cross_{t_i}(s_i(y,z))$. So we can modify $\psi$ to take $\cross_{at_i}(as_i(y,z))$ as input instead of $\cross_{t_i}(s_i(y,z))$. This procedure ensures that all $y$-coefficients in all $s_i$ become divisible by $a$. Hence, after replacing $ay$ by another variable $y'$, and adding the clause that $y'$ is divisible by $a$ into $\theta(y,z)$, we may assume that $a=1$. Finally, by replacing $y$ by $y-\alpha(z)$, we obtain that $\theta(y,z)$ implies that $0\leq y \leq t$ for some basic term $t$.

Fix an $i$ and write $s_i(y,z) = by + s_i'(z)$ for some integer $b\geq 1$ and an $\cL_{\Pres}(\Lambda)$-term $s_i'$. Since $\theta(y,z)$ implies that $0\leq y\leq t$, we also have that $-bt\leq s_i'(z)\leq t_i$. Now, if $\theta(y,z)$ holds, we can recover $\cross_{t_i}(s_i(y,z))$ from $\cross_t(y)$ and $\cross_{t_i+bt}(s_i'(z)+bt)$ purely via the residue rings. Indeed, for $k\geq k'$ the natural injection $A_{k'}\to A_k$ is definable as the inverse of $\res_{k,k'}$, and so we can map
\[
(\cross_{t}(y), \cross_{t_i+bt}(s_i'(z)+bt))\mapsto \res_{t_i+bt, t_i} \left(\cross_{t}(y)^b \cdot \cross_{t_i+bt}(s_i'(z)+bt)\right).
\]
So we can replace the formula $\psi$ by a formula of the form
\[
\psi'\left(\xi, \cross_t(y), (\cross_{t_i'}(s_i''(z)))_i\right),
\]
where $\psi'$ is only over residue ring sorts, $t$ and $t_i'$ are basic terms, and $s_i''$ are $\cL_{\Pres}(\Lambda)$-terms. By Presburger quantifier elimination, we may write $\theta(y,z)$ piecewise as
\[
\theta_0(z) \wedge \underbrace{\beta_1(z)\leq cy \leq \beta_2(z)\wedge y\equiv c \bmod \ell}_{(*)},
\]
where $\theta_0$ is some quantifier free $\cL_{\Pres}(\Lambda)$-formula, $\beta_1, \beta_2$ and $c$ are $\cL_{\Pres}(\Lambda)$-terms, and $\ell$ is a positive integer. Since $\theta(y,z)$ implies that $0\leq y\leq t$, it also implies that $0\leq \beta_j(z)\leq ct$, and we can recover the truth of $(*)$ after adding $\cross_{ct}(\beta_j(z))$ to $\psi'$. Indeed, the order of a bounded interval of $\VG$ can be recovered in a definable way via some map $\cross$ purely from the residue rings, since each $A_i$ is closed under multiplication. At this point, the variable $y$ only appears in $\cross_t(y)$ in our formula, and so we can replace it by a variable from $A_t$ and quantify over $R_t$ instead.
\end{proof}

\subsection{Description of formulas}Theorem \ref{QEZ} on quantifier elimination gives a precise description of formulas.

\begin{thm}\label{QEform}
Any $\cL_D'$-formula in variables $x,\xi,z$ (in the valued field, residue rings, and value group respectively) is $T_D'$-equivalent to a finite disjunction of formulas of the form
\[
\Theta(z, (\ord p_i(x))_i)\wedge \Phi(\xi, (\ac_{t_i}(p_i(x)))_i, (\cross_{t_j}(s_j((\ord p_i(x))_i, z)))_j),
\]
where $\Phi$ is a formula on the residue rings, $\Theta$ is a quantifier-free $\cL_{\Pres}(\Lambda)$-formula, the $t_i$ are basic terms, the $s_j$ are $\cL_{\Pres}(\Lambda)$-terms, and the $p_i$ are polynomials in $x$.
\end{thm}

\begin{proof}
This follows directly from the previous Theorem \ref{QEZ} on quantifier elimination.\end{proof}

\section{Geometrical properties of $\cS$-objects}\label{sec:S}


We spell out key geometric properties of $\cS$-definable sets and $\cS$-definable functions. These are variants for $\cS$ of the classical situation from \cite{Denef2}, \cite{Pas} and \cite{Pas2}. These variants use the reasonings from \cite{CHallp}, \cite{CHR} \cite{CHRV}, \cite{Rid} to generalize the more classical 
cell decomposition results from \cite{Denef2} \cite{Pas} and the Jacobian property from \cite[Theorem 7.5.3 (5)]{CLoes}, where only the residue field is used instead of the more general residue rings $\RR_t$ that we use.

The results of this section will be used later on to show several properties of motivic integrals, in particular that they are well-defined, similarly as in \cite{CLoes} \cite{CLexp} but with the simplified, point-wise approach from \cite{CH-eval}.


\subsection{Cell decomposition in one variable}\label{sec:cell}
Recall that $\cS$ is as specified at the start of \Cref{sec:motivic}.

\begin{defn}[Cells in dimension one]\label{defn.cell.1}

Consider $\cS$-definable sets $Y$ and  $X\subset Y\times \VF$.
Fix $j$ in $\{0,1\}$, an $\cS$-definable function $c:Y\to \VF$, a basic term $t$ with $t\ge 0$, 
and an $\cS$-definable set
$$
R \subset Y\times \RR_t \times (\ZZ\cup\{\infty\})
$$
where furthermore $R \subset Y\times \RR_t^\times \times \ZZ$ when $j=1$ and $R \subset Y\times \{0\} \times \{\infty\}$ when $j=0$, and where $\RR_t^\times$ stands for the units in $\RR_t$. 
Then $X$ is called a \emph{cell over $Y$ with center $c$, depth $t$, type $j$, and parameter set $R$} if it is of the form
$$
X = \{(y,x) \in Y \times \VF\mid  (y,
\ac_t (x-c(y)),\
\ord  (x-c(y))     ) 
   \in R  \}.
$$

If $j=1$, then, any nonempty set of the form
$$
\{x \in  \VF_K \mid  (y, \ac_t (x-c_{K}(y)  ),\  \ord (x-c_{K}(y) ))  = r  \},  
$$
for some point $r=(r_0,K)$ on $R$ and $y=(y_0,K)$ on $Y$
is called a \emph{ball above $y$ of the cell $X$ over $Y$ with center $c$, 
depth $t$, and parameter $r$}.

If $X$ is a cell over $Y$, then, for any model $K$ in $\cS$, the set $X_K$ is also called a cell over $Y_K$.

For an $\cS$-definable set $W\subset X \times \RR_i\times \VG^{k}$ for some basic tuple $i$ and some $k\ge 0$, with projection map $w:W\to Y$ and $X$ a cell over $Y$, call $W$ \emph{prepared on $X$ (over $Y$)} if for each $y=(y_0,K)$ in $Y$ and each ball $B$ above $y$, one has $W_{K,x}=W_{K,x'}$ for each $x,x'$ in $B\times \{y\}$.


\end{defn}


\begin{thm}[Cell decomposition]\label{prop:cell.1}
Let $Y$ and $X\subset Y\times \VF$ be $\cS$-definable sets. 
Then there exists a finite partition of $X$ into $\cS$-definable sets $X_i$ and for each $i$ a basic tuple $t_i$ and an $\cS$-definable set $\widetilde X_i\subset X_i\times \RR_{t_i}$ which is a cell over $Y\times \RR_{t_i}$ such that the projection
$$
p_i: \widetilde X_i \to X_i
$$
is a bijection. 
\end{thm}


\begin{addendum}[Preparation]\label{add.prep.1}
With assumptions and notation from \Cref{prop:cell.1}, let moreover an $\cS$-definable set  $W\subset X \times \RR_t\times \VG^{k}$ be given for some basic tuple $t$ and some $k\ge 0$, with projection map $w:W\to Y$.
Then the $X_i$ and $\widetilde X_i$ of \Cref{prop:cell.1} can be chosen such that furthermore, for each $i$, the $\cS$-definable set
$$
\widetilde W_i :=  W \otimes_{X} \widetilde X_i
$$
is prepared on the cell $\widetilde X_i$ (over $Y\times \RR_{t_i}$).
\end{addendum}
\begin{proof}
The proof of Theorem \ref{prop:cell.1} is similar to that of \cite[Theorem 7.2.1]{CLoes}. First of all, we use Theorem \ref{QEform} to give a description of $X$ by an $\cL'_D$-formula $\Theta$ without valued field quantifiers and value group quantifiers. Suppose that $Y\subset \VF^m\times \RR_t\times \VG^r$ for a basic tuple $t$, and that $p_1(z,x),...,p_d(z,x)$ are the polynomials appearing in $\Theta$, where $z=(z_1,...,z_m)$ is the valued field coordinate function of $Y$ and $(z,x)$ is the valued field coordinate function of $X$. We apply the cell decomposition theorem in \cite{Pas} to the polynomials $p_1,...,p_d$ to obtain a partition of $\VF^{m+1}$ into finitely many cells $Z_i$ over $Y_i\subset \VF^m\times \RR_{\ord N}^n$ for some integer $N>0$, with parameters in $V_i\subset\VF^m\times \RR_{\ord N}^{n+1}\times \VG^e$ ($0\leq e\leq 1$) and an $\cS$-definable center $c_i:Y_i\to \VF$ for each $i\in I$.  

For $(y,x)\in X$ write $y=(z,\xi)$. Then for each $(y,x)\in X$ there exists a unique $i$ and a unique $\xi\in \RR_{\ord N}^n$ such that $(z,\xi)\in Y_i$, $\ordalt(p_j(z,x))$ depends only on $\ordalt h_{ijq}(z,\xi),\ordalt(x-c_i(z,\xi))$ and for each basic term $t'$ one has that $\ac_{t'}(p_j(z,x))$ depends only on $\ac_t(h_{ijq}(z,\xi)),\ac_{t'}(x-c_i(z,\xi)),\cross_{t'}(\ord(x-c_i(z,\xi))),\cross_{t'}(\ord(h_{ijq}(z,\xi)))$ for all $1\leq j\leq d, 1\leq q\leq \ell_j$, where $(h_{ijq})_{i,j,q}$ are $\cS$-definable functions from $Y_i$ to $\VF$. Let $\tilde{t}$ be the maximum taken over all basic terms appearing in $\Theta$.  By using $\Theta$, the functions $\res_{\tilde{t},t'}$ for each basic term $t'$ appearing in $\Theta$ and a similar argument with that of \cite[Theorem 7.2.1]{CLoes}, for each $i$,  it is easy to construct an $\cS$-cell $\widetilde {X}_i\subset X\times \RR_{\ord N}^{n}$ over $\widetilde {Y}_i= (Y_i\times \RR_t\times \VG^r)\cap (Y\times \RR_{\ord N}^{n})$ with parameters in $R_i\subset \widetilde {Y}_i\times \RR_{\tilde{t}}\times\VG^e$ such that  the projection $p_i$ from $\widetilde {X}_i$ to $X$ is injective for all $i\in I$ and $\{p_i(\widetilde {X}_i)|i\in I\}$ forms a partition of $X$.

In order to prove Addendum \ref{add.prep.1}, we also use Theorem \ref{QEform} to give a description of $W$ by a $\cL'_D$-formula $\Theta'$ without valued field quantifiers and value group quantifiers. Now we repeat the above process for the above polynomials $p_1(z,x),...,p_d(z,x)$ and the polynomials $q_1(z,x),...,q_{d'}(z,x)$ appearing in $\Theta'$  to forms  cells $\widetilde {X}_i$ for $i\in I$ such that the projection $p_i$ from $\widetilde {X}_i$ to $X$ is injective for all $i\in I$ and $\{p_i(\widetilde {X}_i)|i\in I\}$ is a partition of $X$ (here, $\tilde{t}$ is the maximum taken over all basic terms appearing in $\Theta$ and $\Theta'$). Now, by using our construction and the above notation, for each $i\in I$, the formula  $(y,x,\zeta,\eta)\in W\wedge(y,x)\in X_i=p_i(\widetilde {X}_i)$  is $\cS$-equivalent with a formula in $y,\zeta,\eta,\ord(x-c_i(z,\xi)), \ac_{\tilde{t}}(x-c_i(z,\xi))$, where $\zeta\in \RR_t$ and $\eta\in \VG^k$. This implies our claim.
\end{proof}


\subsection{Finiteness properties, weak orthogonality, and basic rectilinearisation}\label{sec:weak.orth}
Recall that $\cS$ is as specified at the start of \Cref{sec:motivic}.

The following two basic results are close to results of \cite{CHallp}, and so are their proofs.
The first one repeats \Cref{prop:finite0} from above.
\begin{prop}[Finiteness]\label{prop:finite}
Let $t$ be a basic tuple and let $x$ be an $\cS$-point. Every $\cS(x)$-definable function $f: \RR_t \to \VG$ has finite image.
\end{prop}

\begin{proof}
By \Cref{QEform}, the graph of $f$ is described by a formula
\[
\Theta(z) \wedge \Psi(\xi, (\cross_{t_j}(s_j(z)))_j),
\]
where $z\in \VG$, $\xi\in \RR_t$, $\Theta(z)$ is an $\cL_{\Pres}(\Lambda, x)$-formula, the $t_j$ are basic terms, the $s_j$ are $\cL_{\Pres}(\Lambda, x)$-terms, and $\Psi$ is a formula on the residue rings. We may assume that all of the $s_j(z)$ are non-constant terms in $z$. Since we are working in $\cS(x)$, there are then only finitely many $z\in \VG$ for which there exists a $j$ with $s_j(z)\in [0, t_j]$. Equivalently, for all but finitely many $z$ in $\VG$, all $\cross_{t_j}(s_j(z))$ are zero. But the image of $f$ can contain at most one element $z$ for which all $\cross_{t_j}(s_j(z))$ are zero, and we conclude that $f$ has finite image.
\end{proof}

The following result gives a certain independence between $\RR_t$ and $\VG$ which may be considered as a weak form of orthogonality; at the same time, it gives that the structure on $\VG$ is essentially the Presburger structure.  The weakness of these properties comes from needing the map  $\sigma$. 
This weak form of orthogonality helps to treat unbounded ramification, and is inspired by \cite{CHallp}.

\begin{prop}[Weak orthogonality between $\VG$ and $\RR_t$]\label{prop:ortho}
For each $\cS $-definable sets $Y$ and  $X\subset Y \times \VG^m$, there exist a basic tuple $t$ and a bijective $\cS $-definable function
$$
\sigma : X \to \widetilde  X \subset \RR_t \times X
$$
over $X$ such that the family $\widetilde  X_{z}\subset \VG^m$ is $\cL_{\Pres}(\Lambda)$-definable uniformly in $z$ running over $\RR_t\times Y$. More precisely, there exist an $\cS$-definable function $\gamma:\RR_t \times Y\to\VG^k$ for some $k\ge 0$ and an $\cL_{\Pres}(\Lambda)$-definable set $\Psi\subset \VG^{m+k}$ such that for every point $z$ on $\RR_t \times Y$, the set $\widetilde  X_{z}$ equals $\Psi_{\gamma(z)}$.
 \end{prop}

\begin{proof}
By \Cref{QEform}, we may assume that $X$ is defined by a formula of the form
\[
\Theta(z, g(y))\wedge \Phi(h(y), \cross_t(s(g(y), z)) ),
\]
where $z\in \VG^m$, $y\in Y$, $t = (t_1, \ldots, t_n)$ is a basic tuple, $s$ consists of $\cL_{\Pres}(\Lambda)$-terms, $g: Y\to \VG^n$ and $h: Y\to R_t$ are $\cS$-definable, $\Theta$ is an $\cL_{\Pres}(\Lambda)$-formula, and $\Phi$ is a formula over the residue rings.

For $(y,z)\in X$ define $\zeta(z,y) = \cross_t(s(g(y), z)) \in R_t$, and define
\[
\widetilde {X}_{\zeta, y} = \begin{cases} \{z\in \VG\mid \Theta(g(y), z)\wedge \cross_t(s(g(y), z)) = \zeta & \text{ if } \Phi(h(y), \zeta) \text{ holds},\\
\emptyset & \text{else}.
\end{cases}
\]
Then the map $\sigma: X\to \widetilde {X}$ which maps $(y,z)$ to $\zeta(z,y)$ is a bijection. We define $\gamma$ as
\[
\gamma: \begin{cases} R_t\times Y\to \VG^{2n+1} \\
 (\zeta, y)\mapsto \begin{cases} (g(y), \cross_t^{-1}(\zeta), t) & \text{ if } \Phi(h(y), \zeta) \text{ holds}, \\
(0, \ldots, 0, -1) & \text{else}. \end{cases}\end{cases}
\]
If $\cross_t^{-1}(\zeta)$ is not well-defined in the description of $\gamma$, then we simply take $\cross_t^{-1}(\zeta) = -1$. Finally, we define
\begin{multline*}
\Psi = \{(u_1, u_2, u_3, z)\in \VG^n\times \VG^n\times \VG\times \VG^n \mid  \\
\Theta(u_1, z)\wedge \cross_t(s(u_1,z)) = \cross_t(u_2)\wedge u_3 \neq -1\}.
\end{multline*}
Then we will indeed have that $\widetilde {X}_{\zeta, y} = \Psi_{\gamma(\zeta,y)}$. Also, $\Psi$ is $\cL_{\Pres}(\Lambda)$-definable, because the condition that $\cross_t(a) = \cross_t(a')$ can be expressed by using $t=u_3$.
\end{proof}



\begin{cor}[Further finiteness properties]\label{cor:finite}
Let $\ell>0$, $n\geq 0$ be integers, $t$ be a basic tuple, and $x$ be an $\cS$-point.
\begin{enumerate}
\item Every  $\cS(x)$-definable function $\VG^n\to \RR_t$ has finite image.
\item Every $\cS(x)$-definable function $\VG^n\times \RR_t\to \VF$ has finite image.
\item There does not exist an injective $\cS(x)$-definable function $\VG^\ell \to \VG^{\ell-1}\times \RR_t \times \VF^n$.
\end{enumerate}
\end{cor}

\begin{proof}
For the first item, assume towards a contradiction that $f$ has infinite image. Put a Presburger definable well-order on $\ZZ^n$, for example by injecting $\ZZ^n$ into $\NN^{2n}$ with the lexicographical order. Consider the $\cS(x)$-definable subset $A$ of $\ZZ^n$ containing precisely the minimal elements in $f^{-1}(\xi)$ for each $\xi$ in the image of $f$. This way, we get an injective $\cS(x)$-definable function $A\to \RR_t$ with infinite image,  which contradicts \Cref{prop:finite}.

For the second item, let us suppose by contradiction that $f:\ZZ^n \times \RR_t \to \VF$ has infinite image.  Again using the definable well-order on $\ZZ^n$, consider the $\cS(x)$-definable subset $A$ of $\ZZ^n\times \RR_t$ consisting of tuples $(z,\xi)$ such that $z$ is minimal in $\ZZ^n$ with respect to the property that $f(z,\xi)=x$ for some $x$ in the image of $f$. This way, we get an $\cS(x)$-definable function $A\to \VF$ with infinite image which factors through the projection $A\to \RR_t$. Let $g: \RR_t\to \VF$ be the induced function. Then $g$ is an $\cS(x)$-definable function with infinite image. But any infinite $\cS(x)$-definable subset of $\VF$ contains a ball by \Cref{prop:cell.1}, and thus, by taking a point $b$ inside this ball in the image of $g$, we can construct an $\cS(x,b)$-definable function $g: \RR_t \to \ZZ$ with infinite image, a contradiction to \Cref{prop:finite}.

We now prove the third item. Let us suppose by contradiction that $f:\ZZ^\ell \to \ZZ^{\ell-1} \times \RR_t \times\VF^n$ is an injective $\cS(x)$-definable function.
For each $a$ in $\ZZ^{\ell-1}$, the set
$$X_a := f^{-1}(\{a\}\times \RR_t\times\VF^n)\subset \ZZ^\ell
$$
is finite by the previous items.  Define
$$
X := \{(z,a)\in \ZZ^\ell\times \ZZ^{\ell-1}\mid z\in X_a,\ a\in \ZZ^{\ell-1} \}
$$
and consider a bijective $\cS(x)$-definable function
$$
\sigma : X \to \widetilde  X \subset \RR_t \times X
$$ as given by \Cref{prop:ortho}. The $\cS(x)$-definable set $p(\widetilde  X)$ is finite, with $p$ the projection to $\RR_t$, by the previous items. Let $p(\widetilde  X)_K$ be $\{\xi_1,\ldots,\xi_m\}$ for some $K$ in $\cS(x)$. Then each of the $\cS(x,\xi)$-definable sets $Y_i := p^{-1}(\xi_i)$ is a Presburger set by the above use of \Cref{prop:ortho}. By construction we have finitely many Presburger sets $Y_i\subset \ZZ^\ell\times \ZZ^{\ell-1}$ such that the projection $p_2 : Y_i\to \ZZ^{\ell-1}$ (on the second factor) has finite fibers for each $i$. Also by construction, the union of the sets $p_1(Y_i)$ equals $\ZZ^\ell$, with $p_1$ the projection on the first factor. But this is impossible for Presburger sets, as it contradicts the dimension theory for Presburger sets from \cite{CPres}.
\end{proof}

\begin{cor}[Basic rectilinearization of sets in the value group]
 \label{cor:rec.1}
Let $Z$ and $X\subset \VG 
\times Z$ be $\cS $-definable sets. 
Let $f_1,\ldots,f_N:X \to\VG$ 
be finitely  many $\cS$-definable functions. 
Then, there exist basic tuples $t_i$, a finite partition of $X$ into $\cS $-definable sets $X_i$, and  
$\cS $-bijections
$$
\theta_i:  \widetilde  X_i \subset \VG_{\geq 0} \times \RR_{t_i}  \times Z \to X_i
$$
over $Z$, 
such that the following hold, with $A_i$ the image of $\widetilde  X_i$ under the projection $\widetilde  X_i\to \RR_{t_i}  \times Z$. For each point $\eta$ in $A_i$, the fiber $X_{i,\eta}$ is either $\VG_{\geq 0}$ or the interval $\{a\in\VG_{\geq 0} \mid a\le \alpha(\eta)\}$ for some $\cS$-definable function $\alpha: A_i\to\VG$.
Furthermore, there are integers $a_{i,j}$ and $\cS$-definable functions $\gamma_{i,j}:A_i\to\VG$ such that 
$$
f_j(\theta_i(x,\eta)) = \gamma_{i,j}(\eta) +   a_{i,j}  x
$$
for each point $(x,\eta)$ in $\widetilde  X_i$ with $x$ in $\VG_{\geq 0}$. 
\end{cor}

\begin{proof}By the method in the proof of \cite[Proposition 5.2.6]{CHallp}, we apply Proposition  \ref{prop:ortho} to the graph $G_F\subset Z\times \VG^{1+N}$ of the map $F:=(f_1,...,f_N):X\to \VG^N$. This yields a basic tuple $t$, a bijective $\cS$-definable function $\sigma: G_F\to V\subset G_F\times \RR_t$ over $G_F$,  a natural number $k$, an $\cS$-definable function $\gamma: Z\times \RR_t\to \VG^k$ and an $\cL_{\Pres}(\Lambda)$-definable set $W\subset \VG^{1+N+k}$ such that for each $y\in Z\times \RR_t$ one has $V_{y}=W_{\gamma(y)}$. Our claims follow by using the parametric rectilinearisation result from \cite[Theorem 3]{CPres} for $W\subset \VG\times \VG^{N+k}$ and the fact that the projection $\pi:\VG^{1+N}\to \VG$ induces a bijective $\cS(y)$-definable map between $V_y$ and  its image $\pi(V_y)$ for every $y\in Z\times \RR_t$.
\end{proof}


We now turn to the promised proof of \Cref{lem:torsion:unique}.

\begin{proof}[Proof of \Cref{lem:torsion:unique}]
Assume that $F = 0$. Since $\LL$ and $T$ are invertible, we may always suppose that $b\in \NN$ and $c\in \NN^J$ for each $(a,b,c)\in L$. We order the exponent tuples $(a,b,c)$ in reverse lexicographic order and induct on $\max L$. We will call $\max L$ the \emph{degree of $F$}. Consider first the case where $\max L = (a_k, 0, 0)$, so that $L\subset \NN\times \{0\}\times \{0\}^J$. If $a_k \le 1$ then the statement is clear: first take $F(0)$ to find that $a_0=0$ and then take $F(1)$ to find that $a_1=0$. If $a_k > 1$, we look at the following $\cS(y)$-definable function on $\VG_{\geq 0}$
\[
G(n) := F(n+1) - F(n) = \sum_{a=0}^k d_a (n+1)^a - \sum_{a=0}^k d_a n^a = 0.
\]
Note that $G$ has lower degree than $F$, and that by evaluating at $0$ the constant terms of $F$ and $G$ are both zero. There exists an integer $\ell>0$ such that all coefficients of $G(\ell n)$ are non-torsion. Hence by induction and since $G = 0$, all of its coefficients must be zero. The top degree coefficient of $G(\ell n)$ is $\ell^{a_k-1}(a_k-1)d_{a_k-1} = 0$. But since $d_{a_k-1}$ is non-torsion, we have that $d_{a_k-1} = 0$. The next coefficient of $G(\ell n)$ is $\ell^{a_k-1}d_{a_k-1}\binom{a_k}{2} + \ell^{a_k-2}d_{a_k-2}(a_k-2) = 0$, and so $\ell^{a_k-2}d_{a_k-2}(a_k-2) = 0$. But since $d_{a_k-2}$ is non-torsion we obtain that it is zero. Continuing in this fashion, we see that $d_i = 0$ for $i = 0, \ldots, a_k-1$. To see that then also $d_{a_k} = 0$, simply evaluate $F$ at $1$.

Now let us denote $\max L = (a_k, b_{\max}, c_{\max})$ for the degree of $F$ and write \[F = F_0 + \LL^{b_{\max} \cdot \id}T^{c_{\max} \cdot \id}F_1,\]
where $F_1$ is the top-degree coefficient given by
\[
F_1  = \sum_{a: (a,b_{\max},c_{\max})\in L} d_{a,b_{\max},c_{\max}}\id^a \LL^{b_{\max} \cdot \id}T^{c_{\max}\cdot \id}.
\]
We look at the function on $\VG_{\geq 0}$ given by
\begin{multline*}
  G(n) := F(n+1) - \LL^{b_{\max}} T^{c_{\max}} F(n) \\
= \LL^{b_{\max}(n+1)}T^{c_{\max} (n+1)} (F_1(n+1)-F_1(n)) + F_0(n+1) - \LL^{b_{\max}}T^{c_{\max}}F_0(n) = 0.
\end{multline*}
Note that $G$ has lower degree than $F$. Take again some integer $\ell > 0$ such that all coefficients of $G(\ell n)$ are non-torsion. Then all coefficients of $G(\ell n)$ are zero by induction. In particular, all coefficients of $F_1(\ell n+1) - F_1(\ell n)$ are zero. Then a similar argument as above, using that all of the $d_{a,b,c}$ are non-torsion, shows that $F_1 = 0$. Hence $F = F_0$ and we are done by induction.
\end{proof}

The following will be needed to show that integration over the value group is well-defined.

\begin{cor}
\label{cor:rec:CM.1}
Fix $\cS$-definable sets $Y$ and $X\subset \VG 
\times Y$. 
Let $F_1,\ldots,F_N$ be in $\cCM(X )$ for some integer $N\ge 0$.
Then, there exist a finite partition of $X$ into $\cS$-definable sets $X_i$ and 
bijective $\cS$-definable functions
$$
\theta_i:  \widetilde  X_i \to X_i
$$
 as in \Cref{cor:rec.1} and with its notation for $A_i$, 
such that furthermore, for each $i$ and each $j$,
$\theta_i^*(F_j)$
is of the form
\begin{equation}\label{f:sum:int:prop.1}
\sum_{(a,b,c) \in L_{ij}}  d_{a,b,c}\cdot  \id^a \cdot  \LL^{b\id} \cdot T^{c\cdot \id}
\end{equation}
with nonzero $d_{a,b,c}$ in $\cCM(A_i)$ and finite sets $J\subset \NN$ and
$$
L_{ij}\subset \NN \times \ZZ\times \ZZ^{J},
$$
where $T^{c \cdot \id}$ stands for $\prod_{j\in J} T_{j}^{
c_{j}\id }$ with $\id$ the identity function on $\NN$ and where $0^0=1$ by convention, and such that $d_{a,b,c}$ is non-torsion for each $(a,b,c)\in L$ with $a>0$.
Moreover, in the case that $X_{i,\eta}$ is  $\VG_{\geq 0}$ for each point $\eta$ in $A_i$ and for some $i$, then, for each $j$,
$$
F_j  =0 \mbox{ if and only if  }L_{ij}=\emptyset. 
$$
%
\end{cor}
%


\begin{proof}
The first part about existence follows from applying \Cref{cor:rec.1} to all the $\VG$-valued definable functions $f_\ell$ occurring in the $F_j$.
Indeed, by the weak orthogonality result \Cref{prop:ortho}, and up to making $n$ bigger, we may suppose that each of the brackets occurring in the $\theta_i^*(F_j)$ factors through the projection to $A_i$. On the other hand, if $d_{a,b,c}$ is torsion for some $(a,b,c)\in L_{ij}$ with $a>0$, namely $md_{a,b,c}=0$ for an integer $m>1$, then we can replace $\widetilde {X}_i$ by the images of the $\cS$-definable maps
\[
\widetilde {X}_{ie}:=\begin{cases} \{km+e\mid k\in \NN\}\times A_i\to \NN\times A_i \\ (\ell,z)\mapsto (\frac{\ell-e}{m},z) \end{cases}
\]
for $0\leq e\leq m-1$. By using this repeatedly,  we obtain (\ref{f:sum:int:prop.1}) after finitely many steps.
%
The moreover statement can be checked above each point $\eta$ of $A_i$ separately, and follows from the definition of $\cCM(\ZZ)$ and \Cref{lem:torsion:unique}.
\end{proof}


\subsection{Cell decomposition and Jacobian property}\label{sec:cell.finer}
Recall that $\cS$ is as specified at the start of \Cref{sec:motivic}. In this section we treat higher-dimensional geometry of $\cS$-definable objects. In particular we discuss the Jacobian property, which will be needed for change of variables.


\begin{defn}[Cells in higher dimensions]\label{defn.cell.n}
Consider an integer $n\geq 0$, $\cS$-definable sets $Y$ and $X\subset Y\times \VF^n$, and for $i=1,\ldots, n$, a value $j_i$ in $\{0,1\}$, a basic term $t_i$ with $t_i\ge 0$, and an $\cS$-definable function $c_i:Y\times \VF^{i-1}\to \VF$, and an $\cS$-definable set
$$
R \subset Y\times \RR_t \times (\ZZ \cup \{\infty\})^n
$$
where $x_{<i}=(y,x_1,\ldots,x_{i-1})$,
$t$ stands for $(t_i)_{i=1}^n$, and where furthermore the projection of $R$ to the factor $R_{t_i}$ is a subset of $R_{t_i}^\times$ when $j_i=1$, and of $\{0\}$ when $j_i=0$, and where the projection of $R$ to the $i$-th factor of $(\ZZ\cup \{\infty\})^n$ is a subset of $\ZZ$ when $j_i = 1$ and of $\{\infty\}$ when $j_i = 0$.
Then $X$ is called a \emph{cell over $Y$ with center $c=(c_i)_{i=1}^ n$, depth $t$, type $j= (j_i)_{i=1}^ n$ and parameter set $R$}, if it is of the form
$$
X = \{(y,x) \in Y\times \VF^n\mid  (y,
\big(\ac_{t_i} (x_i-c_i(x_{<i})),
\ord (x_i-c_i(x_{<i}))     \big)_{i=1}^n ) 
   \in R  \}.
$$

If $X$ is a cell over $Y$ with depth $t$, then, for any model $K$ in $\cS$, the set $X_K$ is correspondingly called a cell over $Y_K$, and for any point $r=(r_0,K)$ on $R$, any nonempty subset of the form
$$
\{x \in Y_K\times \VF_K^n\mid  (y, \big( \ac_{t_{i,K}} (x_i-c_{i,K}(x_{<i})  ), \ord (x_i-c_{i,K}(x_{<i}) )    \big)_{i=1}^n )  = r  \},  
$$
of $X_K$ is called a \emph{twisted box of the cell $X_K$ over $Y_K$ with center $c_K$, above $r$, and with depth $t_{K}$}.  By a twisted box of the cell $X$ we mean a twisted box of $X_K$ for some $K$ in $\cS$.
The $\cS$-definable function
$$
T_R: \begin{cases}X\to R \\(x,y) \mapsto (y,    \big( \ac_{t_i } (x_i-c_i(x_{<i}) ), \ord (x_i-c_i(x_{<i}))     \big)_{i=1}^n  )\end{cases}
$$
is called the \emph{twist map} of the cell $X$ over $Y$ with center $c$ and depth $t$. 

\end{defn}

The following is the higher dimensional variant of cell decomposition with some extra addenda.

\begin{thm}[Cell decomposition]\label{prop:cell}
Let $Y$ and $X\subset Y\times \VF^n$ be $\cS$-definable sets for some $n$. 
Then there exists a finite partition of $X$ into $\cS$-definable sets $X_i$ and for each part a basic tuple $t_i$ and an $\cS$-definable set $\widetilde X_i\subset X_i\times \RR_{t_i}$ which is a cell over $Y\times \RR_{t_i}$ such that the projection
$$
p_i: \widetilde X_i \to X_i
$$
is a bijection. 
\end{thm}

\begin{addendum}[Preparation]\label{add.prep}
With the data and assumptions of \Cref{prop:cell}, let moreover an $\cS$-definable set  $W\subset X \times \RR_k\times \VG^{k'}$ be given, for some basic tuple $k$.
Then the $X_i$ and $\widetilde X_i$ can be chosen in such a way that moreover $W$ has constant fibers above $p_i(B)$ for each $i$ and each twisted box $B$ of $\widetilde X_i$. Namely, for any point two points $x=(x_0,K)$ and $x'=(x'_0,K)$  in  $p_i(B)$ for any $i$ and any twisted box $B$ of $\widetilde X_i$, the fiber of $W$ above of $x$, for the projection $W\to X$, and seen as subset of $\RR_k\times \ZZ^{k'}$, equals  the fiber of $W$ above of $x'$.
\end{addendum}

\begin{addendum}[Lipschitz centers]\label{add.Lipschitz}
With the data and assumptions of \Cref{prop:cell}, if we allow for each piece $X_i$ a reordering of the $n$ coordinates on $\VF^n$, then we can furthermore ensure that the center $c$ of the cell $\widetilde X_i$ is Lipschitz continuous (in all but the variables  running over $Y\times \RR_{t_i}$)  with additively written Lipschitz constant $\ord q$ for some nonzero rational number $q$. That is, for each point $z$ on $Y \times \RR_{t_i}$, one has  $\ord ( c(z,x) - c(z,x')  ) \ge \ord (q\cdot (x-x'))$, for all $x$ and $x'$ such that  $(z,x)$ and $(z,x')$ belong to $\widetilde X_i$.
\end{addendum}

\begin{addendum}[Jacobian property and compatible decomposition]\label{add.Jacprop}
With the data and assumptions of \Cref{prop:cell}, suppose that $n=1$ and that an $\cS$-definable function $f:X\to \VF$  and a basic term $t$ are given. Then one can take $X_i$ and $\widetilde X_i$ as in \Cref{prop:cell} such that furthermore the following hold. For each $i$, the function $f$ is either constant or injective and $C^1$ on each twisted box of  $X_i$; if $f $ is injective on $\widetilde X_i$, then for each twisted box $B$ of the cell  $\widetilde X_i$ which is not a singleton (thus, $B$ is a ball in $\VF_K$ for some $K$ in $\cS$), one has for all $(y,x)$ and $(y,x')$ in $p_i(B)$
$$
\ord ( f(y,x) - f(y,x')) = \ord \left( \frac{\partial f(y,x)}{\partial x} \cdot  (x-x') \right),
$$
and
$$
\ac_{t}( f(y,x) - f(y,x')  )  = \ac_{t} \left( \frac{\partial f(y,x)}{\partial x}\cdot  (x-x')\right),
$$
and in particular, both $\ac_{t} ( \partial f(y,x)/\partial x )$  and $\ord ( \partial f(y,x)/\partial x ) $ are constant on $p_i(B)$.
If we let $\widetilde f_i : \widetilde X_i \mapsto \VF\times Y\times \RR_{t_i}$ be the $\cS$-definable function mapping $(a,y,\xi)$ to $(f(p_i(a,y,\xi)),y,\xi)$ with $(y,\xi)\in Y\times\RR_{t_i}$, then $\widetilde f_i (\widetilde X_i )$ is a $1$-cell  over $Y\times \RR_{t_i}$
and there is a correspondence of the twisted boxes of $\widetilde X_i$ with the ones of $\widetilde f_i (\widetilde X_i )$. 
\end{addendum}
\begin{proof}[Proof of \Cref{prop:cell} with \Cref{add.prep,add.Lipschitz,add.Jacprop}]

\Cref{prop:cell} and Addendum \ref{add.prep} are proved similarly by induction on $n$ as in \cite[Theorem 5.2.4 and Addendum 1]{CHR}. Note that the claims for the case $n=1$ are discussed in Theorem \ref{prop:cell.1}  and Addendum \ref{add.prep.1}.

We can adapt the argument in \cite[Sections 2.1, 2.2]{CCL-PW} based on Theorem \ref{QEform} to prove Addendum \ref{add.Lipschitz}. As mentioned in \cite[Theorem 2.2.3]{CCL-PW}, the Lipschitz constant can be bounded in terms of the depths appearing in the definition of the center. But as in the proof of Theorem \ref{prop:cell.1}, we can bound the depths appearing in the definition of the center by $\ord N$ for some integer $N>0$.

In order to prove Addendum \ref{add.Jacprop}, we use the argument in \cite[Theorem 3.14, Corollary 3.2.7]{CHR} and compactness to show that there is an $\cS$-definable subset $C$ of $Y\times \VF$ such that for each $K\in \cS$, each $y\in Y_K$, the subset $C_y$ of $K$  is finite of cardinality $N(y)\leq N$ for a positive integer $N$,  the map $f_y:X_y\setminus C_y\to K$ is a $C^1$-function and for each tuple $(\lambda,\xi):=(\lambda_a,\xi_a)_{a\in C_y}\in \ZZ^{N(y)}\times \RR_t^{N_y}$ for any basic term  $t_0$, for all $x_1\neq x_2$ in the intersection of $X_y$ and the ball
\[
B_y(\lambda,\xi):=\cap_{a\in C_y}\{x\in K | \ord(x-a)=\lambda_a, \ac_{t_0}(x-a)=\xi_a\},
\]
the following properties hold:
\begin{itemize}
\item[(i)]$\ord(\frac{\partial f_y}{\partial x})$ and $\ac_t(\frac{\partial f_y}{\partial x})$ are constant on $X_y\cap B_y(\lambda,\xi)$.
\item[(ii)]$\ord(f_y(x_1)-f_y(x_2))=\ord (\frac{\partial f_y}{\partial x}(x_1))+\ord(x_1-x_2)$ and $\ac_{t_0}(f_y(x_1)-f_y(x_2))=\ac_{t_0} (\frac{\partial f_y}{\partial x}(x_1))\ac_{t_0}(x_1-x_2)$
\item[(iii)] for any open ball $B'\subset X_y\cap B_y(\lambda,\xi)$, $f(B')$ is either a point or an open ball.
\end{itemize}
By Theorem \ref{prop:cell}, we have a cell decomposition of $C$ by cells $(\widetilde {C}_j)_{j\leq M}$ over $Y\times \RR_{t'}$ for a basic tuple $t'$. Since $C_y$ is a finite set for each $K\in\cS$ and each $y\in Y_K$, there are definable subsets $W_j$ of $Y\times \RR_{t'}$ and definable maps $d_j:W_j\to \VF$ such that for each $K\in\cS$, and each $y\in Y_K$ we have that $W_{jy}$ is a set of at most one element and $\cup_{j\leq M} d_j(W_j)=C_y$. By using \Cref{prop:cell} we can also find a cell decomposition of $X$ by cells $(\widetilde {X}_i)_{i\leq M'}$ over $Y\times \RR_{t}$ for a basic tuple $t$.  Let $1\leq i\leq M'$, there are a basic term $t_i$, a definable subset $Z_i$ of $Y\times \RR_t\times \RR_{t_i}\times \ZZ$ and a definable map $c_i:Y\times \RR_t\to \VF$  such that
$$
\widetilde {X}_i=\{(y,\xi,x)\in Y\times \RR_t\times \VF \mid (y,\xi,\ac_{t_i}(x-c_i(y,\xi)),\ord(x-c_i(y,\xi)))\in Z_i\}
$$
and the projection $p_i:\widetilde {X}_i\to X_i$ is a bijection. We set
$$
Z_i'=\pi^{-1}\left(Z_i\setminus \{(y,\xi,\ac_{t_i}(a-c_i(y,\xi)),\ord(a-c_i(y,\xi))\mid (y,\xi)\in V_i, a\in C_y\}\right),
$$
where $V_i$ is the image of $Z_i$ by the projection from $Y\times \RR_t\times \RR_{t_i}\times \ZZ$ to $Y\times \RR_t$ and $\pi:= (\textnormal{id},\textnormal{id},\res_{t+t_0,t},\textnormal{id}):Y\times \RR_t\times \RR_{t_i+t_0}\times \ZZ\to Y\times \RR_t\times \RR_{t_i}\times \ZZ$. Then
$$
\widetilde {X}_i'=\{(y,\xi,x)\in Y\times \RR_{t}\times \VF \mid (y,\xi,\ac_{t_i+t_0}(x-c_i(y,\xi)),\ord(x-c_i(y,\xi)))\in Z_i'\}
$$
is a cell over $Y\times \RR_{t}$. Moreover the projection $p_i':\widetilde {X}_i'\to Y\times \VF$ is injective and has image contained in $X$. For each $1\leq j\leq M$, we set
\begin{multline*}
\widetilde {X}_{ij}:=\{(y,\xi,x)\in Y\times \RR_t\times \VF\mid (y,\xi,\ac_{t_i}(x-c_i(y,\xi)),\ord(x-c_i(y,\xi)))\in H_{ij}\}\\
 \cap \{(y,\xi,x)\in Y\times \RR_t\times \VF\mid \ord(x-d_j(W_{jy}))  \\  =\max_{1\leq \ell\leq M}\ord(x-d_\ell(W_{\ell y}))>\max_{1\leq \ell<j}\ord(x-d_\ell(W_{\ell y}))\},
\end{multline*}
where
\begin{multline*}
H_{ij}=Z_i\cap \{(y,\xi,\ac_{t_i}(d_j(W_{jy})-c_i(y,\xi)),\\
\ord(d_j(W_{iy})-c_i(y,\xi))| (y,\xi)\in V_i, a\in C_y\}
\end{multline*}
and we use the convention that $\widetilde {X}_{ijy}=\emptyset$ if $d_j(W_{jy})=\emptyset$. Then we have $\widetilde {X}_i= (\sqcup_{1\leq j\leq M}\widetilde {X}_{ij})\sqcup \widetilde {X}_i'$. For each $1\leq j\leq M$,  we consider the definable subset $Z_{ij}$ of $Y\times \RR_t\times \RR_{t'}\times \RR_{t_i+t_0}\times \ZZ$ consisting of elements $(y,\xi,\xi',\zeta,\eta)$ such that $(y,\xi')\in W_{j}$ and there exists an $x\in \VF$ for which $(y,x,\xi)\in \widetilde {X}_{ij}$, $\zeta=\ac_{t_0+t_i}(x-d_j(y,\xi'))$ and $\eta=\ord(x-d_j(y,\xi'))$. Define $\tilde{d}_j:Y\times \RR_t\times \RR_{t'}\to \VF: (y,\xi,\xi')\mapsto d_j(y,\xi')$. We consider the cell $\widetilde {X}_{ij}'$ over $Y\times \RR_t\times \RR_{t'}$ with respect to the center $\tilde{d}_j$ and the parameter set $Z_{ij}$. Then it is clear that the projection from $\widetilde {X}_{ij}'$ to $\widetilde {X}_{ij}$ is a bijection.

In summary, $\{p_i'(\widetilde {X}_i'), p_{ij}'(\widetilde {X}_{ij}')\}_{1\leq i\leq M', 1\leq j\leq M}$ forms  a cell decomposition of $X$, where $p_i':\widetilde {X}_i'\to X$ and $p_{ij}':\widetilde {X}_{ij}'\to X$ are the projections. By our construction and the above properties (i),(ii),(iii), it is easy to show our claim excepting the last one related to $\tilde{f}_i$. But this last claim follows by a similar way as in \cite[Addendum 3]{CHR} combining with compactness.
\end{proof}

We have the following strengthening of Addendum~\ref{add.Lipschitz}, where the cell centres become $1$-Lipschitz, but only on the twisted boxes of the cell decomposition.

\begin{lem}[$1$-Lipschitz centres] \label{lem:1.Lipschitz.centres}
Let $X\subset \VF^n$ be $\cS$-definable. Then there exists a cell decomposition of $X$ such that for each twisted box $B$ there exists a reordering of the coordinates such that $B$ is of cell type $(1, \ldots, 1, 0, \ldots, 0)$ and has $1$-Lipschitz centres.
\end{lem}

\begin{proof}
We begin by taking any cell decomposition of $X$. By compactness we may focus on a single twisted box $B$, which after a coordinate reordering is of cell type $(1, \ldots, 1, 0, \ldots, 0)$ and of the form
\[
\Big( \ac_t(x_i - c_i(x_1, \ldots, x_{i-1}), \ord(x_i - c_i(x_1, \ldots, x_{i-1})\Big)_{i=1, \ldots, n} = r.
\]
We may assume that all $c_i$ are $C^1$, and that all partial derivatives of the $c_i$ have constant norm. We first ensure that all partial derivatives of $c_n$ have norm at most $1$. Consider $B$ as the graph of $c_n$ over some twisted box $B'\subset K^{n-1}$. Take $i\in \{1, \ldots, n-1\}$ such that
\[
|\partial_i c_n|
\]
is maximal among $|\partial_j c_n|$ for $j = 1, \ldots, n-1$. If $|\partial_i c_n|\leq 1$ already, then there is nothing to do. Otherwise, reversing the roles of $x_n$ and $x_i$, a similar computation as in~\cite[Cor.\,2.1.14]{CCL} shows that then all partial derivatives of $c_n$ (for the new order of coordinates) are bounded in norm by $1$.

We now refine the first $n-1$ coordinates further so that $c_n$ has the supremum Jacobian property as in~\cite[Def.\,3.3.7]{CHRV} on $B'$. That is, for $x,y\in B'$ we have that
\[
|c_n(x) - c_n(y) - (\grad c_n)(y)\cdot (x-y)| < |(\grad c_n)(y)||x-y|.
\]
The triangle inequality, and the fact that $|\grad c_n|\leq 1$ on $B'$ shows that
\[
|c_n(x) - c_n(y)|\leq |x-y|,
\]
so that $c_n$ is $1$-Lipschitz on $B'$. Using further refinements of the first $n-1$ coordinates, we are done by induction on $n$.
\end{proof}

Cell decomposition with Addendum 1 yields parameterizations by boxes, as follows. By a \emph{box} we mean a Cartesian product of balls. This should not be confused with the notion of twisted boxes, which are more general.

\begin{cor}[Parameterization by boxes]\label{prop:twisted:boxes}
Let $f_1, \ldots, f_N$ be in $\cCeM(X)$ for some $\cS$-definable sets $Z$ and  $X\subset \VF^n\times Z$. Then there exist $\cS$-definable subsets $Y_i$  of $\VF^i\times Z\times \RR_t\times \VG^i$ for $i=0,\ldots,n$, and some basic tuple $t$, and injective $\cS$-definable functions
$$
\varphi_i: Y_i\to X
$$
over $Z$ such that the following hold for each $i$, each $j$, and for each point $z$ on $Z_i:= p_i(Y_i)$ with $p_i:Y_i\to Z\times \RR_t\times \VG^i$ the projection map.


\begin{itemize}
 \item The sets $X_i := \varphi_i(Y_i)$ are disjoint and their union equals $X$,

\item the fiber $Y_{i,z}= p_i^{-1}(z)$ is a box in $\VF^i$, and hence, is of (valued field) dimension $i$,

\item for each point $z$ on $Z_i$, the map $\varphi_{i,z}$ (namely, the restriction of $\varphi_i$ to $Y_{i,z}$), is a $C^1$ map which is Lipschitz continuous with Lipschitz constant $1$, and thus, if $i=n$, then the order of the Jacobian of $\varphi_{i,z}$ equals $0$.  

\item for each point $z$ on $Z_i$, the restriction of  $\varphi_i^*(f_j)$ to $Y_{i,z}$ is constant. Namely, there is $F_{j,z}$ in $\cCM(z)$ such that, for any $x$ in  $Y_{i,z}$, one has $F_{j,z}(x) = \varphi_i^*(f_j)(z,x)$.
\end{itemize}
\end{cor}

\begin{proof}
Since the $f_j$ are in $\cCeM(X)$, we may use \Cref{add.prep} to find a cell decomposition of $X$ over $Z\times \RR_{t'}$ such that on each twisted box of the cell decomposition all data of $f_j$ is constant. Note that this uses the boundedness of the functions $g: Z\to \VF/\cM$ in \Cref{defn:cCD}. By furthermore using \Cref{lem:1.Lipschitz.centres}, we can assume that all cells have type $(1, \ldots, 1, 0, \ldots, 0)$, and that all cell centres of the cell decomposition are $C^1$, and $1$-Lipschitz continuous on twisted boxes. For $i=0, \ldots, n$, let $X_i$ be the union of all $i$-dimensional twisted boxes of this cell decomposition.

Let $B$ be a cell of dimension $i$ of this cell decomposition, say of the form
\[
B = \{(z,x)\in Z\times \RR_{t'}\times \VF^n\mid (z,\big(\ac_{t_i} (x_i-c_i(x_{<i})), \ord (x_i-c_i(x_{<i}))     \big)_{i=1}^n )    \in R\}.
\]
Consider the $\cS$-definable map $\psi_i': X_i\to Z\times \RR_{t'}\times \VF^i$ sending an element $(z,x)\in B$ to
\[
(z,x_1-c_1, x_2-c_2(x_1), \ldots, x_i-c_i(x_{<i})).
\]
Note that this map restricts to a bijection on $B$, and maps every twisted box in $X_i$ over a fixed $z\in Z\times \RR_{t'}$ bijectively to a box in $\VF^i$. Since there are only finitely many cells, $\psi_i$ is finite-to-one, and since finite sets may be uniformly embedded in $\RR_{t''}$ for some basic tuple $t''$ as in~\cite[Lem.\,2.3.1(4)]{CHRV}, we obtain an $\cS$-definable injective map $\psi_i'': X_i\to Z\times \RR_{t'}\times \RR_{t''}\times \VF^i$ which maps every twisted box in $X_i$ over $z\in Z\times \RR_{t'}\times \RR_{t''}$ to a box in $\VF^i$. Now define $t = (t',t'', t_1, \ldots, t_i)$ and consider the $\cS$-definable map $\psi_i: X_i\to \VF^i\times Z\times \RR_t\times \VG^i$ given by
\[
(z,x)\mapsto (\psi''_i(z,x), \big(\ac_{t_j} (x_j-c_j(x_{<j})), \ord (x_j-c_j(x_{<j}))     \big)_{j=1}^i).
\]
This is an injective map, and we conclude by taking $Y_i = \psi_i(X_i)$ and $\varphi_i = \psi_i^{-1}$.
\end{proof}

One can use \Cref{prop:twisted:boxes} to describe motivic integrals, namely by pulling them back using the maps $\varphi_i$. We will not use such a description or \Cref{prop:twisted:boxes} in this paper.

The following is needed for proving Fubini. Recall that a box in $\VF^2$ is a Cartesian product of two open balls.

\begin{cor}[Coordinate switch for cells in $\VF^2$]\label{prop:bi-cells}
Let $X\subset \VF^2$ be a twisted box of type $(1,1)$ with centre $c=(c_1, c_2)$ of depth $t = (t_1,t_2)$, say
\begin{multline*}
X = \{(x_1,x_2)\in \VF^2\mid \ac_{t_1}(x_1-c_1) = \xi_1,\ \ord(x_1-c_1) = y_1, \\
	\ac_{t_2}(x_2-c_2(x_1)) = \xi_2,\ \ord(x_2-c_2(x_1)) = y_2\}.
\end{multline*}
Let $X_i$ be the projection of $X$ onto the $i$-th coordinate. Suppose that $c_2$ is $C^1$, and has the Jacobian property on $X_1$ with depth $t_2$, namely for all $x,x'\in X_1$ one has
\begin{align*}
\ord (c_2(x) - c_2(x')) &= \ord ( \partial c_2/\partial x(x)  \cdot  (x-x') ), \\
\ac_{t_2}( c_2(x) - c_2(x'))  &= \ac_{t_2} ( \partial c_2/\partial x(x) \cdot (x-x')).
\end{align*}
Then either $X$ is a box, or
\begin{multline*}
X = \{(x_1,x_2)\in \VF\times X_2\mid \ac_{t_2}(x_1-c_2^{-1}(x_2)) = -\ac_{t_2}(c_2')^{-1}\xi_2, \\  \ord(x_1-c_2^{-1}(x_2)) = y_2-\ord(c_2')\}.
\end{multline*}
\end{cor}

\begin{proof}
This is similar to the proof of Lemma  5.4.12 of \cite{CHR}. Since $c_2$ has the Jacobian property, the image of $X_1$ under $c_2$ is an open ball of radius $\rho = y_1 + \ord(c_2')$. If $\rho \geq y_2$, then $X$ is a box, and we are done. Otherwise, if $\rho < y_2$, then $c_2^{-1}$ is defined on all of $X_2$, and we have
\[
\ac_{t_2}(x_2-c_2(x_1)) = \xi_2, \, \ord(x_2-c_2(x_1)) = y_2
\]
if and only if
\[
\ac_{t_2}(x_1-c_2^{-1}(x_2)) = -\ac_{t_2}(c_2')^{-1}\xi_2, \ord(x_1-c_2^{-1}(x_2)) = y_2-\ord(c_2'). \qedhere
\]
\end{proof}

\subsection{Rectilinearization in the value group}

The following follows from \Cref{prop:ortho}, and is a higher-dimensional rectilinearization result in the value group.

\begin{cor}[Rectilinearization of sets in the value group]
 \label{cor:rec}
Let $Z$ and $X\subset \VG^{r} 
\times Z$ be $\cS $-definable sets for some $r\ge 0$.  
Let $f_j:X \to\VG$ 
be finitely  many $\cS$-definable functions. 
Then, there exist a finite partition of $X$ into $\cS $-definable sets $X_i$ and 
$\cS $-bijections
$$
\theta_i:  \VG_{\geq 0}^{r_i}\times A_i \to X_i
$$
over $Z$ for some $r_i\le r$ and some $\cS $-definable set $A_i\subset \VG^{r-r_i}\times \RR_t\times Z$ and some basic tuple $t$, where for each $i$,  $A_{i,z}$ is a finite set for every point $z$ on $Z$, and
there are integers $a_{i,j,\ell}$ and $\cS$-definable functions $\gamma_{i,j}:\RR_t\times Z\to\VG$ such that 
$$
f_j(\theta_i(x,y)) = \gamma_{i,j}(y)+   a_{i,j,0} + \sum_{\ell=1}^r a_{i,j,\ell} \cdot  x_\ell
$$
for each point $(x,y)$ such that $x$ is on  $\VG^r$  
and $(x,y)$ on $\VG_{\geq 0}^{r_i}\times A_i$.
Moreover, given $X_i$, the number $r_i$ is independent on the choice of $\theta_i$, and 
$\max_i (r_i)$ 
is independent of the choice of the $X_i$.

\end{cor}

\begin{proof}
This follows from Proposition \ref{prop:ortho} and \Cref{cor:rec.1}  just as Proposition 5.2.6 of \cite{CHallp} is proved. The uniqueness statement follows from the finiteness properties from \Cref{prop:finite} and \Cref{cor:finite}.
\end{proof}


\begin{cor}
\label{cor:rec:CM}
 Fix $\cS$-definable sets $Z$ and $X\subset \VG^{r} 
\times Z$ for some $r\ge 0$.  
Let $F_j$ be in $\cCM(X )$ for $j$ in a finite set $J$.
Then, there exist a finite partition of $X$ into $\cS$-definable sets $X_i$ and 
bijective $\cS$-definable functions
$$
\theta_i:  \VG_{\geq 0}^{r_i}\times A_i \to X_i
$$
as in Corollary \ref{cor:rec} 
such that furthermore, for each $j$,
$\theta_i^*(F_j)$
is of the form
\begin{equation}\label{f:sum:int:prop}
\sum_{(a,b,c) \in L_{ij}}  d_{a,b,c}\cdot  \id^a \cdot  \LL^{b \cdot \id} \cdot T^{c\cdot \id}
\end{equation}
with nonzero $d_{a,b,c}$ in $\cCM(A_i)$ and finite sets $J_i\subset \NN$ and
$$
L_{ij}\subset \NN^{r_i} \times \ZZ^{r_i}\times \ZZ^{J_i\times r_i},
$$
where $\id^a(x)$ for $x$ an $\cS$-point in $\NN^{r_i}$ stands for $x_1^{a_1}\cdots x_{r_i}^{a_{r_i}}$,
$b \cdot \id(x)$ stands for $b_1\cdot x_1 + \ldots+ b_{r_i} \cdot x_{r_i}$ and  $T^{c \cdot \id(x)}$ stands for $\prod_{j\in J_i} T_{j}^{c_{j1}x_{1}+\ldots + c_{jr_i}x_{r_i}}$, with $\id$ the identity function on $\NN^{r_i}$, and such that $d_{a,b,c}$ is non-torsion for each $(a,b,c)\in L_{ij}$ with $\sum_{i=1}^{r_i}a_i>0$.
Moreover, for $j\in J$, one has
$$
f_j=0 \mbox{ if and only if  $L_{ij}=\emptyset$ 
for all  $i$}.
$$
%
\end{cor}
\begin{proof}
The first part about existence follows from applying \Cref{cor:rec} to all the $\VG$-valued definable functions $f_\ell$ occurring in the $F_j$.
Indeed, by the weak orthogonality result \Cref{prop:ortho}, and up to making $n$ bigger, we may suppose that each of the brackets occurring in the $\theta_i^*(F_j)$ factors through the projection to $\RR_{t_i}\times Z$.
%
The moreover statement can be checked above each point $P$ of $A_i$ separately, and follows as in the proof of \Cref{lem:torsion:unique}.
\end{proof}

\section{Integrability, integrals: proofs}\label{sec:int:rev}
We implement the remaining proofs of the above results from \Cref{sec:iterated.int}.
Recall that $\cS$ is as specified at the start of \Cref{sec:motivic}.

\subsection{Integration over one variable}\label{Proofs:int:one}

We begin with our proofs about integration over one variable, either over a residue ring, the value group, or the valued field.

\begin{proof}[Proof of \Cref{defn:int:VG}]
The existence of the $F_j$ and $f_i$ with the desired properties follows from the Presburger rectilinearization result \Cref{cor:rec.1} and its \Cref{cor:rec:CM.1}.
For the independence of the choices of $F_j$ and $f_j$, since one can always take common refinements it is enough to treat the case $k_1:=k = 1$ and  $k'_1:=k'=0$ (where the first set of data is indexed by $1$ and the refining data by $2$). Furthermore, we can assume that $X=\VG_{\ge 0}$ and that $F$ is already of the form (\ref{f:sum:int:prop.1.0.0}). But then a finite definable partition of $\NN$ must consist of finitely many singletons and finitely many equivalence classes modulo some integers, on which the expression (\ref{f:sum:int:VG}) clearly coincides by the corresponding additivity when partitioning the summation set in \Cref{lem:geom}.
\end{proof}

\begin{proof}[Proof of \Cref{lem:VG-rel}]
We have to show the existence of $p_!(F)$ in $\cCM(Y)$, as the uniqueness is clear, where it is key to find a global description of $p_!(F)$.
Use the parametric rectilinearization from \Cref{cor:rec.1} and its \Cref{cor:rec:CM.1} to find data as in these corollaries. In the case that the fiber $X_{i,\eta}$ is $\NN$ we are fine by summation over $\NN$ as above in \Cref{defn:int:VG} based on \Cref{lem:geom}. In the case that $X_{i,\eta}$ is the interval $\{a\in\NN \mid a\le \alpha(\eta)\}$ for some $\cS$-definable function $\alpha: A_i\to\VG$, we use Lemma 4.4.3 of \cite{CLoes} instead of \Cref{lem:geom} to find a global description.
\end{proof}

\begin{proof}[Proof of \Cref{lem:VF-cCeM}]
That $f$ and $F'$ exist as desired follows from the cell decomposition \Cref{prop:cell.1} with its addendum \Cref{add.prep.1} to prepare all data appearing in $F$. To show that the integrability of $G$ and the expression (\ref{f:sum:int:VF:abs}) is independent of the choice of $f=f_1$ and $F'=F'_1$, we suppose first that  $Z$ is a singleton $\{z\}$, that $p'^{-1}(z)=B$ is a closed ball of valuative radius $\alpha(z)$ of the form
\[
B=\{x\in\VF\mid \ac_{t_1}(x-c(z))=\xi_1,\ \ord (x-c(z)) = \alpha-t_1-1\}
\]
for some definable $c(z)$ and some $\xi$ in $\R_{t_1}^\times$ and some basic term $t_1$. 
Let a refinement be given, with some $f_2$ and $F_2'$ as in the statement of \Cref{lem:VF-cCeM}.

\textbf{Case 1} There is a basic term $t_2\ge t_1$ such that the ball $B$ is refined into the disjoint union of the closed balls of valuative radius $\alpha + t_2 - t_1$ each one of which is given by conditions
$$
B_{\xi_2} := \{x \mid \ac_{t_2}(x-c(z)) = \xi_2,\ \ord (x-c(z)) = \alpha -t_1 -1\}
$$
for some $\xi_2$ satisfying $\res_{t_2,t_1}(\xi_2)=\xi_1$. The partition then consists of varying $\xi_2$. The expression (\ref{f:sum:int:VF:abs}) for $f_2$ and $F'_2$ is compared to the one with $f_1$ and $F'_1$ via finitely many iterations of relation (R3), and the integrability of $G_1=G$ and $G_2$ correspond.

\textbf{Case 2}
There is $d(z)\in B$ and a basic term $t_2$ such that the ball $B$ is refined into the disjoint union of the closed balls each one of which is given by conditions
$$
B_{\xi_2,\beta} := \{x\mid \ac_{t_2}(x-d(z)) = \xi_2,\ \ord (x-d(z)) = \beta\}
$$
for some $\xi_2$ in $\RR_{t_2}$ and $\beta$ running over $\VG_{>\alpha}$ and $z_2$ over $Z_2$. Note that the volume of $B_{\xi_2,\beta}$ equals $\LL^{-\beta-t_2-1}$, so that the expression (\ref{f:sum:int:VF:abs}) for $f_2$ and $F'_2$ amounts to the calculation
$$
\int_{\xi_2\in \RR_{t_2}} \int_{\beta > \alpha}F'_2(z_2) \LL^{-\beta-t_2-1} = F\cdot[\RR_{t_2}^\times,0]\cdot\LL^{-\alpha-t_2-1}/(1-\LL^{-1})
$$
which equals $F.\LL^{-\alpha}$ because
$$
[\RR_{t_2}^\times,0] = (\LL-1)\LL^{t_2},
$$
which holds by (R3), and since $p_2'^*(F'_2)$ equals $F$ which is constant on $B$. The integrability of $G_1$ and $G_2$ clearly correspond.

\textbf{Case 3}
In the general case with general $f_1,F'_1, Z_1 =Z$ and $f_2,F'_2,$ and $Z_2$, after a further refinement, we can find a basic term $t_2\ge t_1$ and a definable function $d(z_2)$ with finite image such that for finitely many balls above $Z_1$ we are in Case 2 and for the other balls we are in Case 1. 
We create an intermediary refinement with $f_3,F'_3,$ and $Z_3$ where one is always in Case 1. 
Then the expression (\ref{f:sum:int:VF:abs}) for $f_2$ and $F'_2$ is compared to the intermediary expression (\ref{f:sum:int:VF:abs}) with $f_3$ and $F'_3$ by Case 2 which is further compared to the expression (\ref{f:sum:int:VF:abs}) with $f_1$ and $F'_1$ by Case 1 and by noting that the computations in these cases work in families.  Likewise for the integrability of $G_1$, $G_2$ and $G_3$.
\end{proof}

\begin{rem}\label{rem:Fub-Ton:VGVFRt}
We note a basic version of Fubini-Tonelli here: for a function $F$ in $\cCeM(X)$ and $X$ an $\cS$-definable subset of either $\VG\times\RR_t$ or $R_t\times\VF$ for some basic tuple $t$, \Cref{cor:Fubini-Ton} holds. This follows from the constructions so far.
\end{rem}

\begin{lem}[Lifting Lemma]\label{lem:lift:of:g}
Let $X\subset \VF \times \RR_t\times \VG$ be a definable set, for some basic tuple $t$. Consider an element $F$ in $\cCM(X)$ of the form
\[
[X, g, h, \alpha, \beta, \gamma]
\]
for some definable functions $g: X\to \VF/\cM$, $h: X\to \RR_{t'}, \alpha: X\to \VG^k, \beta: X\to \VG, \gamma: X\to \VG^J$. Then there exists a definable set $X'\subset X\times \RR_s$ and a definable function $d: X'\to \VF$ such that denoting by $p: X'\to X$ the projection we have that $p$ is finite-to-one, that
\begin{equation}\label{eq:tildeX:p*F}
p^*(F) = [X', d, h', \alpha', \beta', \gamma'][X', E(g'-d)]
\end{equation}
and such that $[X', E(g'-d)]$ is in $\cCeM(X')$ and where we have written $g'$ for $g\circ p$,  $h'$ for $h'\circ p$, and so on.
\end{lem}

The important part of this lemma is that the imaginary function $g: X\to \VF/\cM$ has been replaced by a definable function $d: X'\to \VF$, at the cost of introducing some $\RR_s$-parameters, and a correction factor in $\cCeM(X')$; note that such a correction factor can be considered small by the bounded depth property for the additive character that is built in the definition of $\cCeM$.

\begin{proof}[Proof of \Cref{lem:lift:of:g}]
Denote by $G \subset X\times \VF$ the pullback of the graph of $g: X\to \VF/\cM$. Take a cell decomposition of $G$, with cell centre of the form $d: X'\subset X\times \RR_s\to \VF$ and of depth $s'$. Then it is clear that $p^*(F)$ can be written in the way given above. Every twisted box of the cell decomposition for $G$ has a condition of the form
\[
\ac_{s'}(g(x)-d(x,z)) = \xi, \ \ord(g(x)-d(x,z)) = \delta,
\]
for some $\xi\in \RR_{s'}$ and some $\delta\in \VG$, for $(x,z)\in X'$. Now note that this condition describes an open ball which is contained in a translate of $\cM$, and so $\delta \geq -s'$. Hence $g-d$ is bounded on $X'$ and so $[X', g'-d, h', \alpha', \beta', \gamma']$ is indeed in $\cCeM(X')$.
\end{proof}


\begin{proof}[Proof of \Cref{lem:VF:cCM:int}]
We have to show the existence of $\widetilde X$, $g$ and $G$.

First write $F$ as a finite sum of products of generators, as follows
\begin{equation}\label{eq:proof:int:VF:F}
F = \sum_i F_i := \sum_i a_i [Z_i, (\prod_j \alpha_{ij}) E(g_i)H(h_i)\LL^{\beta_i}T^{\gamma_i}]
\end{equation}
for some $a_i$ which are products of generators of type (G1), where $F_i$ is the $i$-th term.
By an iterated use of the Lifting \Cref{lem:lift:of:g}, there is $X'\subset X\times R_s$ and for each $i$ a definable function $d_i: X'\to \VF$ such that denoting by $p: X'\to X$ the projection we have for each $i$ that
\[
p_i^*(F_i) = [X', \alpha_i, \beta_i, \gamma_i, d_i, h_i][X'_i, E(g_i-d_i)]
\]
and such that $[X', E(g_i-d_i)]$ is in $\cCeM(X')$, where we write $\alpha_i$ for $\alpha_i\circ p$ and so on, and where $X_i'$ is a subset of $X'$. Here, $d_i(\cdot,x)$ is an $\cS$-definable function on $X'_{x}\subset \RR_{s}$.

Note that by our finiteness results of \Cref{cor:finite} there is an integer $N>0$ such that the image $I_{i,x}$ of $d_i(\cdot,x)$ is finite and of size less than $N$ for each point $x$ on $X$.  Let $I_x$ be the union of the $I_{i,x}$ over all $i$ and consider a bijection $\iota_x: I_x \to Y_x\subset R_t$ for some basic tuple $t$. Such a bijection exists similarly as in~\cite[Lem.\,2.3.1(4)]{CHRV}.
Consider
$$
\widetilde X := \{(x',\xi)\mid x'\in X',\ \xi\in Y_x\},
$$
where $x'$ lies above $x$. We may suppose that  $\widetilde X$ is $\cS$-definable, as we can take $\iota_x$ to depend definably on $x$. Let $\sigma:\widetilde X\to X$ be the projection. For $G(x,\xi)$ we take the sum
\begin{equation}\label{eq:proof:int:VF}
\sum_{i} a_i [X'_{x,\xi}, (\prod_j \alpha_{ij}) \LL^{\beta_i}T^{\gamma_i}H(h_i)],
\end{equation}
with
$$
X'_{x,\xi} := \{ \zeta\in X'_x\mid  d_i(\zeta,x)= \iota^{-1}(x,\xi)\}.
$$
and for $E(g(x))$ we take $E(\iota^{-1}(x,\xi))$ so that we have $E(g)G\in \cCM(\widetilde X)$. That $\sigma_! (G \cdot E(g)) = F$ now follows from construction.
\end{proof}

\begin{proof}[Proof of \Cref{lem:VF:int:crit}]
First we treat the easy direction which is the implication from the criterion to the definition of integrability. So, assume that $F$ can be written as a finite sum of terms $c_i \cdot [Z_i,f_i]$ with $c_i$ a finite product of generators of type (G1)  and $[Z_i,f_i] = [Z_i,\alpha_i,\beta_i,\gamma_i,g_i,h_i]$ a generator of type (G2) such that moreover each function  $F_i := [Z_i,\alpha_i,\beta_i,\gamma_i,0,h_i]$ is integrable over $X$. Repeat the construction of the proof of \Cref{lem:VF:cCM:int} to construct $g$ and $G$ on $\widetilde X$. Each of the terms in (\ref{eq:proof:int:VF}) is clearly integrable, hence $G$ itself is integrable as a finite sum of integrable functions.

For the converse direction, let us take $\widetilde X$, $G$ and $g$ witnessing the integrability of $F$. Write $G$ as a finite sum of $G_i$, which are products of generators of types (G1) and (G2), as we did for $F$ in (\ref{eq:proof:int:VF:F}). Then each $\sigma_!(G_i)$ is integrable over $\VF$ by construction and since integration over $R_t$ cannot kill integrability (see Remark \ref{rem:Fub-Ton:VGVFRt}), they lie in $\cCeM(X)$. By construction we have found a way of writing $F$ as a sum of terms  $F_i$ as desired, as we can now write
$$
F = \sum_i \sigma_!(G_i E(g)),
$$
and, the $\sum_i \sigma_!(G_i)$ are integrable and of the form $$
\sum_i a_i [Z_i, (\prod_j \alpha_{ij}) E(0)\LL^{\beta_i}T^{\gamma_i}]
$$
as desired.
\end{proof}

\begin{proof}[Proof of \Cref{lem:VF:cCM}]
By the criterion from \Cref{lem:VF:int:crit}, we may suppose that $F$ is of the form
$$
F = a[Z,g,h,\alpha,\beta,\gamma]
$$
in $\cCM(X)$, for some $a$ which is a finite product of generators of the form (G1).
By the Lifting \Cref{lem:lift:of:g}, up to replacing $X$ by $X'$ and $\VF$ by $\VF\times \RR_s\times \VG$ for some basic tuple $s$ and up to multiplying $a$ with an element from $\cCeM(X')$ of a special form which does not affect integrability, we may suppose that
$g$ is the reduction modulo $\cM$ of an $\cS$-definable function $\tilde g:X'\to \VF$.

Furthermore, by cell decomposition \Cref{prop:cell} with its Addenda \ref{add.prep} and \ref{add.Jacprop}, we may suppose that the following properties hold on $X'$, where we write $p':X'\to \RR_s\times \VG$ for the projection and $X'_z$ for $p'^{-1}(z)$ for any point $z$ on $\RR_s\times \VG$.

First of all, $F$ restricted to $X'_z$ lies in $\cCeM(X'_z)$, the $h,\alpha,\beta,\gamma$ are constant on $X'_z$, $X'_z$ is a closed ball of valuative radius $\delta(z)$, and $X'_z$ is also a cell with center $c(z)$, $\tilde g$ is as in Addendum \ref{add.Jacprop} (for both the Jacobian property and the compatible preparation of domain and range), so that the image of the set $X'_z$ under $\tilde g$ is either a singleton or at the same time a closed ball of radius $\delta(z)$ and a cell with center $d(z)$ and depth $u$. Furthermore, all of this data depends definably on $z$.

Now we can compute $\int_{X'_z} F_{|X'_z}$ which equals $p_!(F_{|X'_z}$) when we write $p$ for the projection to $\{y\}$. On the part $X''$ of $X'$ where $\delta(z)>0$, we have that $F_{|X''}$ is of class $\cCeM$. Indeed, the order of $\tilde g-d(z)$ is at least $-u$ on that part and $d(z)$ can only take finitely many values when $z$ moves over $\RR_s\times \VG$, so that  in total the valuation of $\tilde g$ is uniformly bounded below on the part $X''$. On the part $X'''$ where  $\delta(z)\le 0$, one has by the orthogonality relation (R4), applied iteratively and for each $z$, that $p_!(F_{|X'''_z})=0$.

On $X''$, the integrability of $F_{|X''}$ is clear by the definitions of integrability, since we are restricting to a subset.

The independence of the choice of $X'$ and $f'$ follows from taking common refinements and relation (R4).
\end{proof}

\begin{proof}[Proof of \Cref{lem:VF-rel}]
The global existence of the object $p_!(F)$ follows from the proof of \Cref{lem:VF:cCM}, which works well in families, where one notes that also the criterium from \Cref{lem:VF:int:crit}, the Lifting \Cref{lem:lift:of:g}, and the proof of \Cref{lem:VF-cCeM} work well in families.
\end{proof}

\subsection{General integration}\label{Proofs:int:general:int}

We finish up the remaining proofs from \Cref{sec:int:general,sec:mot-Mellin}.

\begin{proof}[Proof of \Cref{thm:gen-abs}]

Since one can swap two adjacent coordinates, and a finite composition of those can become any permutation, it is enough to consider the following cases for proving \Cref{thm:gen-abs}.

\textbf{Case 1: switching with $\RR_t$.} Let $F\in \cCM(X)$ with $X\subset \VG\times \RR_t$. Then it is clear by construction that
\[
\int_{\VG} \int_{\RR_t} F = \int_{\RR_t}\int_{\VG} F,
\]
and similarly when $X\subset \VF\times \RR_t$ that
\[
\int_{\VF} \int_{\RR_t} F = \int_{\RR_t}\int_{\VF} F.
\]
Also the Tonelli-aspect for defining integrability is clear in this case, namely that the integrability condition is indepedent of the choice of coordinates. Indeed, integration over $\RR_t$ does not influence integrability.

\textbf{Case 2: switching $\VG$ with $\VG$.} Let $F\in \cCM(X)$ with $X\subset \VG\times \VG$. Since integration and integrability on $\VG\times \VG$ corresponds to summability of certain series of real numbers as in \Cref{lem:geom}, this case follows from the usual Fubini-Tonelli theorem on $\ZZ\times \ZZ$.

\textbf{Case 3: Switching $\VG$ with $\VF$.} Assume first that $F\in \cCeM(X)$ with $X\subset \VG\times \VF$. Without loss of generality we may suppose that $X=A$. By Presburger rectilinearization, we can find a definable set
\[
\widetilde{X}\subset \VG\times \VF\times \RR_t
\]
such that the projection $p: \widetilde{X}\to X$ is a bijection, up to removing a measure zero set from $X$ coming from a finite $\cS$-definable subset of $\VF$, such that for every $\xi\in \RR_t$ the fibre $\widetilde{X}_\xi$ is of the form
\[
\{(n,x)\in \VG\times \VF\times \mid \ac_t(x-c(\xi)) = \xi, (n,\ord(x-c(\xi)))\in L_\xi\},
\]
for some Presburger definable set $L_\xi\subset \VG^2$, and such that all data in $F$ is constant on the balls of the form
\[
\{(n,x)\in \VG\times \VF\mid \ac_t(x-c(\xi)) = \xi, n = n_1, \ord(x-c(\xi)) = n_2\},
\]
for some $\xi\in \RR_t$ and $(n_1,n_2)\in L_\xi$. By the previous cases, we may ignore the factor $\RR_t$, and simply assume that $X$ itself is of the form
\[
\{(n,x)\in \VG\times \VF\mid (n,\ord(x-c))\in L\}
\]
for some Presburger definable set $L\subset \VG^2$.

We now compute the integral of $F$ over $\VG\times \VF$ as
\[
\int_{n\in \VG}\int_{x\in\VF} F(n,x) = \int_{n\in \VG} \int_{m\in L(n, \cdot)} \LL^{m-t-1} F'(n,m).
\]
Computing the integral of $F$ over $\VF\times \VG$ instead gives
\[
\int_{x\in \VF}\int_{n\in \VG} F(n,x) = \int_{m\in \VG} \LL^{m-t-1} \int_{n\in L(\cdot, m)} F'(n,m),
\]
and so we conclude by the previous case. Indeed, also the integrability condition for $A=X$ is preserved after swapping the coordinates. Since we can do this for any $A$, we are done for the definition of integrability and for the definition of the integral for $F$ in $\cCeM(X)$.

Now assume that $F$ is in $\cCM(X)$, with still $X\subset \VG\times \VF$. Note that $\widetilde X$ and the corresponding set $L$ can be chosen compatibly with construction in the proof for \Cref{lem:VF:cCM}, where a reduction to the case $\cCeM(X)$ is made when integrating over $\VF$, and that this remains the same reduction when swapping the coordinates. This reduces to the case that $F$ lies in $\cCeM(X)$ and finishes Case 3.

\textbf{Case 4: Switching $\VF$ with $\VF$.} First consider $F$ in $\cCeM(X)$ for some $X\subset \VF^2$.
Take a cell decomposition adapted to all data appearing in $F$ and adapted to $X$, and parameterize with auxiliary variables in $\VG\times \RR_t$ for some $t$. The measure zero set $X_0\subset \VF$ we remove consists of the projections of all cells of type $(0,0)$ or $(0,1)$ of this cell decomposition. Take an $\cS$-definable set $A$ and refine the cell decomposition to now also partition $A$. We may suppose that $A=X\setminus X_0$ by shrinking $X$. By the previous cases, the situation now reduces to that of a bi-cell as in \Cref{prop:bi-cells}, namely a cell in $\VF^2$ for which we know explicitly how we can swap the coordinate order. That the volume of the bi-cell matches with the volume after swapping coordinates follows from \Cref{prop:bi-cells} and the Jacobian property of \Cref{add.Jacprop}.

For $F$ in $\cCM(X)$ for some $X\subset \VF^2$, parameterizing with auxiliary variables in $\VG\times \RR_t$ for some $t$, we can suppose that in each fiber we are of class $\cCeM$ instead of $\cCM$ and finish by the above treatment for $\cCeM$, together with using the previous cases to get rid of the parametrization.

That the definition of the integrability via the Tonelli-aspect is independent of the coordinate order follows similarly from the explicit bi-cell computation and the previous cases.

\end{proof}

\begin{proof}[Proof of \Cref{thm:gen-rel} and \Cref{cor:Fubini,cor:Fubini-Ton}]
These sta\-tements follow from the proof of \Cref{thm:gen-abs} since that clearly works well in definable families.
\end{proof}


\begin{proof}[Proof of \Cref{thm:cov} and \Cref{cor:thm:cov}]
This is similar to the proof of change of variables in \cite{CLoes}. In more detail, the case that $X_1$ and $X_2$ are $\cS$-definable subsets of $\VF$ is clear from the cell decomposition \Cref{prop:cell} with its Addendum \ref{add.Jacprop}.
Up to doing an parametrization with auxiliary coordinates $\RR_t$ and working piecewise, one can factorize $f$ as a finite composition of maps which are constant in all but one $\VF$-variable, which is treated by the family version of the first treated case with $n=1$.  All this works well in definable families, and thus \Cref{thm:cov} and \Cref{cor:thm:cov} are proved.
\end{proof}

\subsection{Mellin and Fourier transformations}

Now we come to our proofs on the Mellin transform.

\begin{proof}[Proof of \Cref{thm:injM}]
One implication is trivial, so let us assume for the other implication that $\cMmot(F_1)=\cMmot(F_2)$, where we recall that we use the same set of variables $\lambda_{i_j}, T_{i_j}$ not appearing in $F_1$ or $F_2$.
Clearly there is an $\cS$-definable subset $W\subset (\VF^\times)^n$ of measure zero such that $F_1$ and $F_2$ are locally constant outside $W$. Indeed, one can make all data appearing in $F_1$ and in $F_2$ locally constant away from such a measure zero $\cS$-definable subset $W$.

Choose an $\cS$-point $y$ in $(\VF^\times)^n\setminus W$. Write $y=(y_0, K)$ where $K$ is in $\cS$, and hence has $\cL$-structure. Let $B\subset (K^\times)^n$ be a box containing $y_0$ on which $F_1$ and $F_2$ are constant. Consider now the point $y' = (y_0, K)$ which is the same as $y$, except that the interpretation of $\lambda_{j_i}$ is chosen so that the set
\[
\{x\in K^n\mid \ord(x_i) = \ord(y_{0i}), \ac_{\lambda_{j_i}}(x_i) = \ac_{\lambda_{j_i}}(y_{0i})\}
\]
is completely contained in $B$. Since the $\lambda_{j_i}$ do not appear in $F_1$ or $F_2$, we have that if $F_1(y') = F_2(y')$ then necessarily also $F_1(y)=F_2(y)$.

We show how to recover $F_1(y')$ from $\cMmot(F_1)$. Define $\xi_i = \ac_{\lambda_{j_i}}(y_{0i}) \in \RR_{\lambda_{j_i}}$ and $\delta_i = \ord(y_i)$. Write $F_1$ as a combination of generators, and consider one generator of the form
\[
[Z, (\prod_{i}\alpha_i) \LL^\beta T^\gamma E(g) H(h)].
\]
After computing $\cMmot(F_1)$ the argument of the multiplicative character of this generator is of the form $(h', \id):Z'\subset \RR_s\times \RR_t\times \RR_{\lambda}\to \RR_i\times \RR_\lambda$, where $\lambda = (\lambda_{j_1}, \ldots, \lambda_{j_n})$. Restrict $\cMmot(F_1)$ to the fibre over $(\xi_1, \ldots, \xi_n)\in \RR_\lambda$, call this function $G\in \cCM(y')$. The function $G$ is a rational function in the variables $T_{j_1}, \ldots, T_{j_n}$, and by construction can be expanded as a power series with coefficients in $\cCM(y')$. Then the coefficient of $T_{j_1}^{\delta_1}\cdots T_{j_n}^{\delta_n}$ in $G$ is exactly the value of $F_1(y')$. Hence we may recover the value $F_1(y')$ from $G$, as desired.
\end{proof}

To prove that the Fourier transform is injective we follow~\cite{CLexp}. For $\alpha\in \ZZ$ let $\phi_\alpha\in \cCM(\VF^n)$ be the characteristic function of
\[
Z_\alpha = \{x\in \VF^n\mid \ord(x_i)\geq \alpha\}.
\]
\begin{lem}\label{lem:Fourier.phi.alpha}
We have that
\[
\cF(\phi_\alpha) = \LL^{-\alpha n}\phi_{-\alpha+1}.
\]
\end{lem}

\begin{proof}
By induction we can assume that $n=1$. We have that
\[
\cF(\phi_\alpha)(x) = \int_{y\in \VF} \phi_\alpha(y)E(xy) = \int_{\ord y \geq \alpha} E(xy).
\]
If $\ord(x) + \alpha \leq 0$ then this integral becomes zero by (R4). For $\ord(x) + \alpha > 0$ the value becomes $\LL^{-\alpha}$.
\end{proof}

If $F,G\in \cCM(\VF^n)$ are both integrable, then for each $x\in \VF^n$ the motivic function $y\mapsto F(y)G(x-y)$ is also integrable, and we define the \emph{convolution $F\ast G$} of $F$ and $G$ as the motivic function in $\cCM(\VF^n)$ given by
\[
(F \ast G)(x) = \int_{y\in \VF^n} F(y)G(x-y).
\]

\begin{lem}\label{lem:Fourier.conv}
Let $F,G\in \cCM(\VF^n)$ be integrable. Then
\[
\cF(F\ast G) = \cF(F)\cF(G).
\]
\end{lem}

\begin{proof}
We compute for an $\cS$-point $x$ that
\[
\cF(F\ast G)(x) = \int_{y\in \VF^n}(F\ast G)(y)E(xy) = \int_{y\in \VF^n}\int_{z\in \VF^n} F(z)G(y-z)E(xy).
\]
Now perform the coordinate transformation $u = y+z, z=z$ to see that this integral is equal to $\cF(F)(x)\cF(G)(x)$.
\end{proof}

For $F\in \cCM(\VF^n)$ denote by $F^\vee\in \cCM(\VF^n)$ the motivic function given by $x\mapsto F(-x)$.

\begin{lem}\label{lem:double.Fourier}
Let $F\in \cCM(\VF^n)$ be integrable. Then for every $\alpha\in \ZZ$ the function $\phi_\alpha \cF(F)\in \cCM(\VF^n)$ is integrable and
\[
\cF(\phi_\alpha \cF(F)) = \LL^{-\alpha n}F^\vee \ast \phi_{-\alpha+1}.
\]
\end{lem}

\begin{proof}
For $x$ an $\cS$-point we compute
\begin{align*}
\cF(\phi_\alpha \cF(F))(x) &= \int_{y\in \VF^n} \phi_\alpha(y) \cF(F)(y)E(xy) \\
&= \int_{y\in \VF^n}\int_{z\in \VF^n} \phi_\alpha(y)F(z)E(y(x+z)) \\
&= \int_{z\in\VF^n} F(z) \cF(\phi_\alpha)(x+z) \\
&= \int_{z\in \VF^n} F(z)\LL^{-\alpha n} \phi_{-\alpha+1}(x+z) \\
&= \LL^{-\alpha n} (F^\vee\ast \phi_{-\alpha+1})(x),
\end{align*}
where we have used \Cref{lem:Fourier.phi.alpha}, and the coordinate transform $z\mapsto -z$.
\end{proof}

\begin{lem}\label{lem:conv.zero}
Let $F\in \cCM(\VF^n)$ and assume that for every $\alpha\in \ZZ$ we have that $F\ast \phi_\alpha = 0$. Then $F = 0$ almost everywhere.
\end{lem}

\begin{proof}
There exists an $\cS$-definable set $W\subset \VF^n$ of measure zero such that $F$ is locally constant outside $W$. If $x\in K^n\setminus W$, and $F$ is constant on the closed box $B$ of radius $\alpha$ around $x$, then
\[
(F\ast \phi_\alpha)(x) = \LL^{-\alpha n} F(x) = 0,
\]
and hence $F(x) = 0$.
\end{proof}

We have all the ingredients to prove that the Fourier transform is injective.

\begin{proof}[Proof of \Cref{thm:injF}]
One direction is trivial, so assume that $\cF(F) = 0$. We want to conclude that $F=0$ almost everywhere. By \Cref{lem:double.Fourier} we obtain that $F^\vee \ast \phi_{-\alpha+1} = 0$ for every $\alpha\in \ZZ$. Hence we conclude by \Cref{lem:conv.zero}.
\end{proof}

Let us also prove Fourier inversion for integrable motivic functions with integrable Fourier transform.

\begin{prop}\label{prop:Fourier.inversion}
Let $F\in \cCM(\VF^n)$ be integrable and assume that $\cF(F)$ is also integrable. Then
\[
\cF(\cF(F)) = \LL^{-n} F^\vee
\]
almost everywhere.
\end{prop}

\begin{proof}
For $\alpha\in \ZZ$, we denote by $W_\alpha$ the definable set consisting of all $x\in Z_\alpha$ such that $F$ is constant on $x+Z_{-\alpha+1}$. Then $(W_\alpha)_\alpha$ is a definable family, and the definable set $K^n\setminus \cup_{\alpha} W_\alpha$ has measure zero. Let $V_\alpha = W_\alpha \setminus W_{\alpha-1}$, which is still a definable family over $\ZZ$, and $K^n\setminus \cup_{\alpha} V_\alpha$ is still of measure zero. For $\alpha\in \ZZ$ denote by $F_\alpha$ the restriction of $F$ to $V_\alpha$. Then
\[
F = \int_{\alpha\in \VG} F_\alpha
\]
almost everywhere. Therefore, Fubini shows that it is enough to prove that for every $\alpha \in \ZZ$ we have
\[
\cF(\cF(F_\alpha)) = \LL^{-n} F_\alpha^{\vee}.
\]

By construction, $F_\alpha$ is integrable and has the property that
\[
F_\alpha\ast \phi_{-\alpha+1} = \LL^{(\alpha-1)n}F_\alpha.
\]
We now simply compute, using \Cref{lem:Fourier.phi.alpha} and \Cref{lem:Fourier.conv}
\begin{align*}
\cF(\cF(F_\alpha)) &= \cF(\cF(\LL^{(-\alpha+1)n} F_\alpha\ast \phi_{-\alpha+1})) \\
&= \cF(\phi_{\alpha}\cF(F_\alpha)) = \LL^{-\alpha n} F_\alpha^\vee \ast \phi_{-\alpha+1} = \LL^{-n}F_\alpha^\vee.\qedhere
\end{align*}
\end{proof}

\begin{rem}
The motivic functions $F_\alpha$ constructed in the above proof are the motivic analogue of Schwartz--Bruhat functions in the classical context, see~\cite[Sec.\,7.5]{CLexp}. The proof may be summarized as first proving Fourier inversion for Schwartz--Bruhat functions, and then approximating an arbitrary function by Schwartz--Bruhat functions almost everywhere.
\end{rem}

Finally, we prove compatibility with earlier treatments of motivic integration.

\begin{proof}[Proof of \Cref{result:cCexp-widehatC}]
That the maps are ring homomorphisms which are compatible with motivic integration follows from the constructions. The injectivity follows from the first main theorem of \cite{CH-eval}.
\end{proof}

\section{$p$-adic specialization and transfer}\label{sec:p-adic-transf}

\subsection{Specialization}\label{sec:special}


We give several specialization results of our motivic functions, either to $p$-adic fields, or to local fields of large positive characteristic.

In \Cref{sec:informal-meaning}, it is indicated that one may take values inside a commutative $\CC$-algebra $G$, and that one may take the functions $H_t$ arbitrary valued inside $G$.  In this section, we focus on the most important case that $G=\CC$, and that the $H_t$ are multiplicative characters on the $p$-adic residue rings. As a general function $H_t$ on a finite abelian group is a finite $\CC$-linear combination of characters on that group, this is merely a small change of perspective. A reader can have the more general choice of an arbitrary map $H_t$ instead of $\chi_t$ in mind, as in \Cref{sec:informal-meaning}.

We discuss three levels of specialization, each one becoming more concrete.

Fix $\cS = \cS_D$ with notation from \Cref{subsec:def}. Let $X$ be an $\cS$-definable set. If $K$ is a finite extension of $\QQ_p$, then we may base change $X$ to the $\cS_{D,K}$-definable set $X_{\cS_{D,K}}$, where $\cS_{D,K}\subset \cS_D$ consists only of the field $K$ with any possible $\cL_D$-structure. If $F\in \cCM(X)$ then all of the data defining $F$ are $\cS_D$-definable, and so by base change we obtain $F_{\cS_{D,K}}\in \cCM(X_{\cS_{D,K}})$.   
In this way we obtain a \emph{specialization map}
\[
\cCexp_{M 
}(X)\to \cCexp_{M 
}(X_{\cS_{D,K}}): F \mapsto F_{\cS_{D,K}}.
\]
If $K$ is a local field of positive characteristic, then this procedure mapping $F\in \cCM(X)$ to $F_{\cS_{D,K}}$ makes sense whenever $K$ is of characteristic $>N$, for some positive integer $N$ depending on $F$.

\begin{thm}[Specializing integration]\label{thm:special}
Let $X$ be $\cS_D$-definable and let $F\in \cCM(X)$ be integrable. Then there exists an integer $N$ such that for every local field $K$ of characteristic zero, or of characteristic $> N$ we have that
\[
\left(\int_X F\right)_{\cS_{D,K}} = \int_{X_{\cS_{D,K}}} F_{\cS_{D,K}}.
\]
In particular, specialization preserves the Mellin transform.
\end{thm}

\begin{proof}
The integral $\int_X F$ is computed by taking an $\cS_D$-definable cell decomposition of $X$ adapted to $F$. Such a cell decomposition automatically restricts to an $\cS_{D,K}$-definable cell decomposition, at least when the characteristic of $K$ is large enough, from which the result follows.
\end{proof}

Elements in the rings $\cCM(X)$ and $\cCM(X_{\cS_{D,K}})$ are built for abstract (motivic) additive and multiplicative characters, and so it is possible to specialize further to $\cCexp_{M, \cS_{D,K}}(X^{\enrich})$ from \Cref{sec:p-adic:mellin}.

\begin{thm}\label{thm:spec:KS}
Let $K$ be a local field of characteristic zero, and let $X$ be $\cS_{D,K}$-definable. Then there exists a natural morphism
\[
\Psi: \cCM(X)\to \cCexp_{M, \cS_{D,K}}(X^{\enrich}).
\]
If $F\in \cCM(X)$ is integrable (in the sense of \Cref{sec:iterated.int}), then $\Psi(F)$ is $L^1$ (for any choice of the $\lambda$ and the $\pi_K,\psi,\chi_t, \tau_i$ as in \Cref{defn:integrability:K}), and in that case we have
\[
\Psi\left( \int_X F\right)(\lambda,\pi_K,\psi,\chi_t,\tau_i) = \int_{x\in X_{(K,\cL_D)}} (\Psi(F)(x,\lambda,\pi_K,\psi,\chi_t,\tau_i)).
\]
where the $(\lambda,\pi_K,\psi,\chi_t,\tau_i)$ play their natural role, and where the $\cL_D$-structure on $K$ in the right hand side index of $X$ comes from the choice of $\lambda,\pi_K$.
\end{thm}

Note that this result gives that integrability in our abstract sense of \Cref{sec:iterated.int} implies integrability in the classical $L^1$ sense, as spelled out in \Cref{defn:integrability:K}.

\begin{proof}[Proof of \Cref{thm:spec:KS} and \Cref{defn:integrability:K}]
We define $\Psi$ in the following way. Firstly, we put $\Psi(\LL) = q_K$ which fixes the image of $\Psi$ on the generators of type (G1). Secondly, let $F = [Z,\left(\prod_i \alpha_i\right) \LL^\beta T^\gamma E(g)H(h)]$ be a generator of $\cCM(X_{\cS_{D,K}})$ of type (G2). Let $(x,\psi,(\chi_t)_t)$ be an enriched point on $X$; note that this in particular also fixes $\lambda,\pi_K$.  Then we define
\[
\Psi(F)(x, \psi, (\chi_t)_t) = \sum_{\substack{\xi\in Z_x \\ \xi' = (x,\xi')}} \left(\prod_i \alpha_i(\xi')\right) q_K^{\beta(\xi')} T^{\gamma(\xi')} \psi(g(\xi')) \left(\prod_i \chi_{t_i}(h_i(\xi'))\right).
\]
The map $\Psi$ respects the relations (R1)-(R5) since these relations hold over $K$. Therefore $\Psi$ indeed gives a well-defined morphism as stated.

The conditions on integrability of $F$ imply corresponding results on the summability of certain series over $\ZZ$ appearing when computing the integral of $\Psi(F)$ via a cell decomposition. In particular, if $F$ is integrable, then these abstract conditions imply that $\Psi(F)$ is integrable, and the integral may be computed via the same cell decomposition.
\end{proof}

Finally, we mention one further specialization. Fix a $p$-adic field $K$ and $\cL_D$-structure on $K$, and consider $\{(K, \cL_D)\}$ as a subset of $\cS_{D,K}$. If $X$ is $\cS_{D,K}$-definable, then we obtain a further specialization morphism
\[
\cCexp_{M, \cS_{D,K}}(X^{\enrich})\to \cCexp_{M, (K, \cL_D)}((X_{(K, \cL_D)})^{\enrich}),
\]
which commutes with integration. Essentially, in $(K, \cL_D)$ we have fixed the values of $\pi_K$ in $\cM_K$ and of $\lambda$ in $\ZZ$ for all $\lambda\in \Lambda$, but one can still vary the characters $\psi$ and $\chi_t$ for all basic terms $t$.

\subsection{Transfer principles}\label{sec:transfer}

We have the following transfer principle between characteristic zero and positive characteristic local fields, extending the one from \cite{CLexp}. In our setting, we need quite a strong condition on an isomorphism of some higher residue rings $R_t$ (in \cite{CLexp} this was needed only for the residue fields). We also need a lemma about characters on local fields, generalizing~\cite[Lem.\,9.2.3]{CLexp}.

\begin{lem}\label{lem:9.2.3}
Let $K$ be any local field. Let $r<s<0$ be integers, fix  a nontrivial additive character $\psi_0$ on $B_{s}(0)$ such that $\psi_0(\cO_K)\not=0$ and $\psi_0(\cM_K)=0$ and write $\cD_{K,s}$ for the collection of additive characters on $K$ such that $\psi(x) = \psi_0(x)$ for all $x$ in $B_{s}(0)$. Suppose that $x_1,\ldots x_n$ are in $B_r(0)$ satisfying $\ord(x_i-x_j)<s$ for each $i\not=j$, and let $c_1,\ldots,c_n$ be in $\CC$.
Put
\begin{equation}\label{eq:ortho:0}
S_\psi := \sum_{i=1}^n c_i\psi(x_i)
\end{equation}
and suppose that $S_\psi = 0$ for all $\psi$ in $\cD_{K,s}$.
Then one must have $c_i=0$ for each $i=1,\ldots,n$.
\end{lem}

\begin{proof}
The quotient $G = B_r(0)/B_s(0)$ is a finite group under addition. Denote by $\widehat{G}$ the dual group of $G$, and consider the function $S: \widehat{G}\to \CC$ which sends a character $\chi$ on $G$ to the sum
\[
\sum_i c_i \psi_0(x_i) \chi(x_i).
\]
If $\psi$ is as described in the statement of the lemma, then $\chi(x) = \psi(x)/\psi_0(x)$ is a character on $G$, and all characters on $G$ arise in this way. For such a character we have that
\[
S(\chi) = S(\psi / \psi_0) = \sum_i c_i \psi(x_i) = 0,
\]
so that $S$ is the zero function. But since the $x_i$ are distinct elements in $G$, $S$ is the Fourier transform of the function $G\to \CC$ which sends $y$ to $c_i \psi_0(x_i)$ if $y \equiv x_i\bmod B_s(0)$ and to $0$ else. Hence also this function must be zero, and we conclude that $c_i = 0$ for all $i$.
\end{proof}

\begin{thm}\label{thm:transfer}
Let $X$ be an $\cS_D$-definable set and take $F_1, F_2\in \cCM(X)$. Then there exists an integer $N > 0$ such that for every integer $A\ge 0$ there exists an integer $A'\ge 0$ with the following property. Let $K_1, K_2$ be local fields of residue characteristic $\geq N$, and with $\RR_{K_1, A'}$ isomorphic to $\RR_{K_2, A'}$. Assume that for every enriched point $x=(x_0, \psi, (\chi_t)_t, (\lambda)_{\lambda\in \Lambda})$ on $X_{\cS_{D,K_1}}$ with $\lambda \leq A$ for all $\lambda\in \Lambda$ it holds that
\[
F_{1, \cS_{D,K_1}}(x) = F_{2, \cS_{D,K_1}}(x).
\]
Then we also have that for every enriched point $x'=(x'_0, \psi, (\chi_t)_t, (\lambda)_{\lambda\in \Lambda})$ on $X_{\cS_{D,K_2}}$ with $\lambda\leq A$ for all $\lambda\in \Lambda$ that
\[
F_{1, \cS_{D, K_2}}(x') = F_{2, \cS_{D, K_2}}(x').
\]
\end{thm}

\begin{proof}
Put $F = F_1-F_2$, and assume first that $F\in \cCeM(X)$. Then there exists a basic tuple $t = (t_1, \ldots, t_n)$, an $\cS_D$-definable map $f: X\to \RR_t\times \VG^m$, and a function $G\in \cCeM(\RR_t\times \VG^n)$ such that $F = f^*G$. By the results on specialization, there exists a positive integer $N > 0$ such that this equality $F = f^* G$ holds in all local fields of characteristic $> N$. Take $A'$ such that if $\lambda\leq A$ for all $\lambda\in \Lambda$, then all $t_i$ and all basic terms appearing in all formulas defining the data of $F$ are bounded by $A'$, and such that the valuations of the arguments of the additive character are bounded below by $-A'$ (which is possible since $F$ is in $\cCeM(X)$ as opposed to $\cCM(X)$). Let $z$ be an enriched point on $\RR_t\times \VG^m$ on $K_1$ with $\lambda\leq A$ for all $\lambda\in \Lambda$. Then $G_{\cS_{D, K_1}}(z) = 0$. By the assumption that $\RR_{K_1, A'}$ is isomorphic to $\RR_{K_2, A'}$, $z$ also makes sense as an enriched point on $K_2$ (where the additive character plays a role only on the bounded region  $B_{-A'}(0) / \cM$ of $\VF/\cM$).  But then also $G_{\cS_{D, K_2}}(z) = 0$.

Now assume that $F\in \cCM(X)$, and write $F$ as a sum of generators. Consider one such generator of the form
\[
[Z,\left( \prod_i \alpha_i\right) \LL^{\beta}T^\gamma E(g) H(h)],
\]
where $Z\subset X\times \RR_t$ and $g: Z\to \VF / \cM$. We claim that there exists a basic term $s$ such that for each $x\in X$, the map $Z_x\to \VF/B_{-s}(0): \xi\mapsto g(x,\xi)$ has finite image. Let $G\subset Z\times \VF$ be the definable set above the graph of $g$ in $Z\times \VF/\cM$. By taking a cell decomposition of $G$, we find a basic term $s$ and for every $x\in X$ a finite $\cS_D(x)$-definable set $D(x)\subset \VF$ such that for every $\xi\in Z_x$ the set $G_{x,\xi}\subset \VF$ is a union of twisted boxes of depth $s$ and centre in $D(x)$. Since $G_{x, \xi}$ is a translate of $\cM$, we see that $G_{x,\xi}$ is contained in the union of balls $B_{-s}(d)$ for $d\in D(x)$. Hence we indeed obtain that $g(x, \cdot): Z_x\to \VF / B_{-s}(0)$ has finite image for every $x\in X$.

We now furthermore assume that $A'\geq s$, so that also $\RR_{K_1, s}$ and $\RR_{K_2,s}$ are isomorphic. Let $x=(x, \psi, (\chi_t)_t, (\lambda)_{\lambda\in \Lambda})$ be an enriched point on $X_{\cS_{D,K_1}}$. Take $g_1, \ldots, g_n\in K_1$ such that $\ord(g_i-g_j) > s$ and such that the image of $g$ modulo $B_{-s}(0)$ is exactly $\{g_1, \ldots, g_n\}$ in $K_1$. For each $i$, let $Z_{x,i}$ be the set of $\xi\in Z_x$ for which $g(\xi) \equiv g_i$ modulo $B_{-s}(0)$. Then
\begin{multline*}
F_{\cS_{D,K_1}}(x) = \sum_{\xi\in Z_x} \left(\prod_j \alpha_j(\xi)\right) q^{\beta(\xi)} T^{\gamma(\xi)}\psi(g(\xi)) \left(\prod_j \chi_{t_j}(h_j(\xi))\right) \\
= \sum_i \left(\sum_{\xi\in Z_{x,i}} \left(\prod_j \alpha_j(\xi)\right) q^{\beta(\xi)} T^{\gamma(\xi)}\psi(g(\xi)-g_i) \left(\prod_j \chi_{t_j}(h_j(\xi))\right)\right) \psi(g_i)\\
 = 0.
\end{multline*}

Since this holds for every choice of $\psi$, the previous lemma implies that
\[
\sum_{\xi\in Z_{x,i}} \left(\prod_j \alpha_j(\xi)\right) q^{\beta(\xi)} T^{\gamma(\xi)}\psi(g(\xi)-g_i) \left(\prod_j \chi_{t_j}(h_j(\xi))\right) = 0.
\]
By assumption, the rings $B_{-s, K_1}(0) / \cM_{K_1}$ and $B_{-s, K_2}(0) / \cM_{K_2}$ are isomorphic, and the images of the $g_i\bmod \cM$ under this isomorphism map to elements which may be lifted to $g_i'\in K_2$ with the same properties as the $g_i$. But then exactly the same computation shows that for every enriched point $x'=(x', \psi', (\chi_t)_t, (\lambda)_{\lambda\in \Lambda})$ on $X_{\cS_{D,K_2}}$ with $\lambda\leq M'$ for all $\lambda\in \Lambda$ we have that
\[
F_{\cS_{D,K_2}}(x') = 0. \qedhere
\]
\end{proof}

\subsection{Some perspectives}\label{sec:perspectives}

Various questions and challenges remain open in this new enriched framework, compared to the $\cCexp$-class functions from \cite{CLexp} in the motivic case and from \cite{CHallp} in the uniform $p$-adic case, which are already more studied.

We already mentioned in \Cref{sec:res:Qp:int} that the main results on integrability of \cite{CGH} are interesting to study also for our more general functions. By extension, it may be interesting to perform a consequent study of various loci (of integrability, of local constancy, and so on) as in \cite{CGH5}, and the corresponding transfer principles for integrability, local constancy, and so on, also as in \cite{CGH5}.
Such a study of loci would enable one to generalize the theory of uniform $p$-adic distributions and wave front set from \cite{CHLR}, \cite{AizC} to our class of uniform $p$-adic $\cCM$-functions, and possibly the motivic ones from \cite{Raib-motivic} and distributions in the style of \cite{HK} to our class of $\cCM$-functions.

A big challenge seems to be to investigate whether our transfer principle from \Cref{thm:transfer} also holds for local fields $K_1$, $K_2$ with just isomorphic residue fields $k_{K_i}$, instead of the stronger assumption of isomorphic residue rings $\RR_{K_i, A'}$ (that is, to investigate whether one can take $A'=0$).   Note that our transfer principle already generalizes the one from \cite{CLexp}, and we do not know whether a further generalization is possible.

Possible generalizations of motivic Poisson summation as in \cite{HKPoiss} to new situations seem an interesting challenge to us.

Finally, when one studies the Harish-Chandra local character expansion around the unit, it is plausible that the coefficients are of motivic nature in the sense of \cite{CLexp}, but around other semisimple elements than the unit it is expected that multiplicative characters come into play (if one tries to write the coefficients in a field independent or motivic way) and the framework of this paper may be useful to show that the coefficients are still of a motivic nature, see also \cite{GordonSpice}.

\appendix

\section{Some complements on more general $\cS$}\label{app:hensel-min}

\subsection{}
\label{subs:coars}

In this appendix, we give some possible generalizations, namely by considering more general $\cS$. Recall that we imposed already the following conditions on $\cS$, from the beginning of \Cref{sec:motivic} on.
\\

\begin{itemize}
\item[(C0)]  $\cL$ contains $\cL_D$ and has precisely the same 
sorts. 
\\

\item[(C1)] $\cS$ is a nonempty collection of $\cL$-structures each of which expands a structure from $\cS_D$.
\end{itemize}

\noindent
Item (C1) says that for each structure $K$ in $\cS$, the $\cL_D$-reduct of $K$ 
belongs to $\cS_D$. We now give some more conditions on $\cS$, related to the notion of Hensel minimality from \cite{CHR} and \cite{CHRV}, and to the results from sections \ref{sec:weak.orth} and \ref{sec:cell.finer}. 

\begin{defn}\label{defn:hen}
Let 
$\cS$ satisfy conditions (C0) and (C1).
Say that $\cS$ is  \emph{D-h-minimal} if moreover the following properties hold. \\

\begin{itemize}
\item[(C2)] For each integer $N>0$, the collection of the residue rings $\cO_K/N\cdot \cM_K$ over all $K$ in $\cS$, with their induced structure, forms an elementary class.
\\

\item[(C3)] The finiteness property from Proposition \ref{prop:finite}.
\\

\item[(C4)] The weak orthogonality between $\VG$ and $\RR_t$ from Proposition \ref{prop:ortho}.
\\

\item[(C5)]  For each $K$ from $\cS$, for any valued field $K'$ with $\cL$-structure which is elementary equivalent to the valued field $\VF_K$ with its $\cL$-structure and any (nontrivial) coarsening $\cO'$ of the valuation $\cO_{K'}$ on $K'$ so that each of the basic terms becomes a unit of the coarsened valuation ring, the valued field $(K',\cO')$ with its induced $\cL$-structure is $1$-h-minimal in the sense of \cite{CHR}.





\end{itemize}
If we replace $1$-h-minimal by $\omega$-h-minimal in (C5) while preserving the conditions (C0) -- (C4), then we say that $\cS$ is  \emph{D-$\omega$-h-minimal}.
\end{defn}

\begin{prop}\label{pop:CHR-to-D}
Suppose that $\cS$ satisfies (C0) up to (C4).  Then $\cS$ satisfies (C5) if and only if all of the properties of \Cref{sec:S} hold for the $\cS$-objects.
\end{prop}
\begin{proof}
See \cite{CHR} and \cite{CHRV}. Indeed, by the results of \cite{CHR}, one has all the results of \Cref{sec:S} in the coarsened structure, and by compactness (as in \cite{CHRV}), there then exist a basic tuple such that parametrization with $\RR_t$ makes the desired results hold in $K$ as well, for any data. For the other implication, note that the geometric criterion for $1$-h-minimality from \cite{CHR} and \cite{CHRV} is satisfied by the results in \Cref{sec:S}, more precisely by the cell decomposition \Cref{prop:cell} and the Jacobian property and compatible decomposition from Addendum \ref{add.Jacprop}.
\end{proof}

Finally we recall a result which follows almost directly from \cite{CHR,CHallp}.

\begin{prop}[\cite{CHR,CHallp}]\label{result:base-language*i}
The collection $\cS_D$ is D-h-minimal, and even D-$\omega$-h-minimal.
If $\cS$ is D-h-minimal and
if $x$ is an $\cS$-point, then $\cS(x)$ is also D-h-minimal, and similarly for D-$\omega$-h-minimality.
 \end{prop}

Note that axiom (C2) is only used in this paper for \Cref{result:cCexp-widehatC}, when we use the main result from \cite{CH-eval}. Axiom (C2) may be used similarly to give a more global meaning to our motivic functions; note that in this paper we chose a point-wise approach, in the spirit of \cite{CH-eval}, where in \cite{CLoes} and \cite{CLexp} the objects were defined more globally.


\subsection{Examples}\label{sec:examples}

We give a few examples of $\cS$ that are D-h-minimal.

\begin{itemize}

\item If we are given a collection $\cS$ which is D-h-minimal, then any sub-collection  $\cS'$ of $\cS$ in which (C2) is still satisfied is again D-h-minimal. This follows directly from the axioms. As an example, one could take a single $p$-adic field and consider on it all possible $\cL_D$-structures. As a final example of this kind, one could consider a single $p$-adic field and with a fixed $\cL_D$-structure. Note that the collections used for our specialization are of this form, starting from $\cS_D$. Note that one could also choose $\cS'$ to consist of only the $p$-adic fields in $\cS_D$, that is, the finite field extensions of $\QQ_p$ equipped with all possible $\cL_D$-structures, but this needs to be enriched with equicharacteristic zero fields in $\cS_D$ with pseudo-finite residue field in order to satisfy (C2).

\item Let us be given a collection $\cS$ which is D-h-minimal. Then, one can add extra structure on the residue rings $\RR_t$, for basic terms $t$, and an extra theory about these expanded residue rings $\RR_t$. If one includes sufficiently many such expansions in a new collection $\cS'$ so that (C2) holds for $\cS'$, then $\cS'$ is also D-h-minimal. Likewise one may add extra constant symbols in any of the sorts and an extra theory about it.

\item One may add to each $K$ in $\cS_D$ some analytic structure, say, coming from a Weierstrass system over $\ZZ$ or over $\ZZ\llp t\rrp$, as in \cite{vdDAx}, \cite{CLip} or \cite{CLips}. Again, if one includes sufficiently many such expansions in a new collection $\cS'$ so that (C2) holds for $\cS'$, then $\cS'$ is also D-h-minimal.

\end{itemize}



%
%


\bibliographystyle{amsplain}
\bibliography{anbib}
\end{document}